\documentclass[11pt]{article}

\usepackage[utf8]{inputenc}
\usepackage[british]{babel}
\usepackage[a4paper, margin=2.5cm]{geometry}
\usepackage{amsmath, amsthm, amssymb, amsfonts}
\usepackage{mathtools}
\usepackage{thmtools}
\usepackage[shortlabels,inline]{enumitem}
\usepackage[font={small, it}]{caption}
\usepackage{subcaption}
\usepackage{microtype}
\usepackage[x11names]{xcolor}
\usepackage{mathabx}
\usepackage{hyperref}
\usepackage{bm}

\newcommand{\cE}{\ensuremath{\mathcal E}}

\newcommand{\cG}{\ensuremath{\mathcal G}}
\newcommand{\cH}{\ensuremath{\mathcal H}}

\newcommand{\cL}{\ensuremath{\mathcal L}}

\newcommand{\cP}{\ensuremath{\mathcal P}}

\newcommand{\cS}{\ensuremath{\mathcal S}}

\newcommand{\cW}{\ensuremath{\mathcal W}}

\newcommand{\eps}{\varepsilon}
\renewcommand{\phi}{\varphi}
\renewcommand{\rho}{\varrho}

\DeclareMathOperator*{\E}{\mathbb{E}}
\DeclareMathOperator*{\N}{\mathbb{N}}

\DeclareMathOperator*{\R}{\mathbb{R}}

\DeclarePairedDelimiter\ceil{\lceil}{\rceil}
\DeclarePairedDelimiter\floor{\lfloor}{\rfloor}
\DeclarePairedDelimiter\paren{\lparen}{\rparen} 
\DeclarePairedDelimiter\set{\lbrace}{\rbrace}

\let\setminus=\smallsetminus

\newcommand{\osref}[2]{%
  \setlength\abovedisplayskip{5pt plus 2pt minus 2pt}
  \setlength\abovedisplayshortskip{5pt plus 2pt minus 2pt}
  \ensuremath{\overset{\text{#1}}{#2}}
}

\newcommand{\Gnp}{G_{n, p}}

\newcommand{\cGEk}{\ensuremath{\mathcal{G}_{\mathrm{exp}}^k}}
\newcommand{\Fabs}{\ensuremath{F_{\mathrm{abs}}}}

\newcommand{\Fsw}{\ensuremath{F_{\mathrm{sw}}}}
\newcommand{\Fconn}{\ensuremath{F_{\mathrm{conn}}}}

\newcommand{\Bin}{\ensuremath{\mathrm{Bin}}}

\newcommand*{\bfm}[1]{\ensuremath{\mathbf{#1}}}

\colorlet{RoyalRed}{red!70!black}
\definecolor{RoyalAzure}{rgb}{0.0, 0.22, 0.66}

\declaretheorem[parent=section]{theorem}
\declaretheorem[sibling=theorem]{lemma}
\declaretheorem[sibling=theorem]{proposition}
\declaretheorem[sibling=theorem]{claim}

\declaretheorem[sibling=theorem,style=definition]{definition}

\setlist{itemsep=0.1em, topsep=0.1em, parsep=0.1em, partopsep=0.1em}

\hypersetup{
	colorlinks,
	linkcolor={RoyalRed},
	citecolor={SpringGreen4},
	urlcolor={RoyalBlue1}
}

\newlength{\bibitemsep}
\newlength{\bibparskip}
\let\oldthebibliography\thebibliography
\renewcommand\thebibliography[1]{%
  \oldthebibliography{#1}%
  \setlength{\parskip}{\bibitemsep}%
  \setlength{\itemsep}{\bibparskip}%
}

\newcounter{propcnt} 

\newlist{alphenum}{enumerate}{1}
\setlist[alphenum,1]{%
  label=\normalfont{\bfseries{(\Alph{propcnt}{\arabic*})}},
  ref=\normalfont{(\Alph{propcnt}{\arabic*})},
  leftmargin = \parindent+3.5em,
}

\title{Blow-up lemma for cycles in sparse random graphs}
\author{
  Milo\v{s} Truji\'{c}\thanks{Institute of Theoretical Computer Science, ETH
  Z\"{u}rich, 8092 Z\"{u}rich, Switzerland. Email:
  \texttt{mtrujic@inf.ethz.ch}.}
}
\date{}

\begin{document}
\maketitle

\begin{abstract}
  In a recent work, Allen, B\"{o}ttcher, H\`{a}n, Kohayakawa, and Person
  provided a first general analogue of the blow-up lemma applicable to sparse
  (pseudo)random graphs thus generalising the classic tool of Koml\'{o}s,
  S\'{a}rk\"{o}zy, and Szemer\'{e}di. Roughly speaking, they showed that with
  high probability in the random graph $\Gnp$ for $p \geq C(\log
  n/n)^{1/\Delta}$, sparse regular pairs behave similarly as complete bipartite
  graphs with respect to embedding a spanning graph $H$ with $\Delta(H) \leq
  \Delta$. However, this is typically only optimal when $\Delta \in \{2,3\}$ and
  $H$ either contains a triangle ($\Delta = 2$) or many copies of $K_4$ ($\Delta
  = 3$). We go beyond this barrier for the first time and present a sparse
  blow-up lemma for cycles $C_{2k-1}, C_{2k}$, for all $k \geq 2$, and densities
  $p \geq Cn^{-(k-1)/k}$, which is in a way best possible. As an application of
  our blow-up lemma we fully resolve a question of Nenadov and \v{S}kori\'{c}
  regarding resilience of cycle factors in sparse random graphs.
\end{abstract}

\addtocounter{propcnt}{21}

\section{Introduction}

Problems concerning embedding a spanning graph $H$ into a host graph $G$ under
various conditions have always been among the most challenging topics to study
in extremal combinatorics. One of the strongest tools in this area is certainly
the \emph{blow-up lemma} of Koml\'{o}s, S\'{a}rk\"{o}zy, and
Szemer\'{e}di~\cite{komlos1997blow}. It led to several deep and beautiful
results, some gems including spanning trees~\cite{komlos1995proof,
komlos2001spanning}, powers of Hamilton cycles~\cite{komlos1998proof},
$H$-factors~\cite{komlos2001proof}, bounded degree
subgraphs~\cite{bottcher2009proof}, and many more. We refer an interested
reader to great surveys and gentle introduction into using the blow-up lemma
and related tools~\cite{komlos1996szemeredi, kuhn2009embedding,
rodl2010regularity}.

In order to apply it the host graph $G$ is required to be highly structured and
dense, in a sense that it contains $\Omega(n^2)$ edges, which is perhaps its
main drawback. A natural next step is to ask whether this powerful tool can be
`transferred' to a sparse setting, in which the host graph has only $o(n^2)$
edges. Arguably the most interesting and thoroughly studied instances of such
graphs are (pseudo)random graphs, notably the binomial Erd\H{o}s-R\'{e}nyi
random graph\footnote{$\Gnp$ stands for the probability distribution over all
graphs on vertex set $[n] := \{1, \dotsc, n\}$ where each edge is present with
probability $p := p(n) \in (0, 1)$ independently.} $\Gnp$.
(see~\cite{conlon2014combinatorial} for an overview of some influential research
regarding transference of combinatorial results to a sparse random setting).

In context of a sparse blow-up lemma, the host graph $G$ would ideally be given
as a collection of \emph{sparse regular pairs}. For $p \in [0, 1]$ and $\eps > 0$
a pair of sets $(V_1,V_2)$ is $(\eps,p)$-\emph{regular} (in a graph $G$) if for
every $V_i' \subseteq V_i$, $i \in \{1, 2\}$, with $|V_i'| \geq \eps|V_i|$, the
density $d(V_1',V_2')$ of edges between $V_1'$ and $V_2'$ in $G$ is such that
\[
  |d(V_1,V_2) - d(V_1',V_2')| \leq \eps p.
\]
However, this basic notion of regularity is not sufficient for embedding a
spanning graph $H$ even for $p = 1$, as $(\eps,1)$-regular pairs can have
isolated vertices. An $(\eps,p)$-regular pair $(V_1,V_2)$ is said to be {\em
$(\eps,\alpha,p)$-super-regular} if additionally every $v \in V_i$ satisfies
$\deg_G(v,V_{3-i}) \geq (1-\eps)|V_{3-i}|\alpha p$, for $i \in \{1,2\}$. Then
the original blow-up lemma~\cite{komlos1997blow} says $H$ can be embedded into a
certain collection of $(\eps,\alpha,p)$-super-regular pairs. Rather
unfortunately, it is known that this cannot be adapted in a straightforward way
to a setting in which $p$ is an arbitrary decreasing function of the number of
vertices of $H$---for instance, there are graphs on vertex set $V_1 \cup V_2
\cup V_3$ where each $(V_i,V_j)$ is $(\eps,1,p)$-super-regular, but contain no
triangles (see~\cite{gerke2005sparse, kohayakawa2003regular}). Hence, further
strengthening is needed.

One such strengthening was proposed by Balogh, Lee, and
Samotij~\cite{balogh2012corradi} in a work on triangle factors in subgraphs of
random graphs. In a graph $G$, sets $\{V_i\}_{i \geq 2}$ which are pairwise
$(\eps,\alpha,p)$-super-regular are said to have the \emph{regularity
inheritance property} if for every $V_i,V_j,V_k$ and $v \in V_i$, the pair
$\big(N_G(v,V_j),N_G(v,V_k)\big)$ is $(\eps,p)$-regular of density at least
$d(V_j,V_k) - \eps p$, inheriting regularity from the pair $(V_j,V_k)$. Using
this definition they proved that with high probability\footnote{A property is
said to hold with high probability (w.h.p.\ for short) if the probability for it
tends to $1$ as $n \to \infty$.} for $p \gg (\log n/n)^{1/2}$ every subgraph $G$
of $\Gnp$ on sets $\{V_i\}_{i \in [3]}$ of linear size which are pairwise
$(\eps,\alpha,p)$-super-regular and have the regularity inheritance property,
contains a disjoint collection of triangles covering all of its vertices. This
result can be considered as the first real blow-up type statement for sparse
graphs.

Allen, B\"{o}ttcher, H\`{a}n, Kohayakawa, and Person~\cite{allen2016blow}
recently established several sought-after variants of a general blow-up lemma
for sparse random and pseudorandom graphs together with many relevant
applications. Simply put, they showed\footnote{We are not completely true to
word when presenting this result due to sheer load of technicalities involved.
The result is much more general and specific than presented here, but we
highlight all the main points and provide no further details.} that for every
$\Delta \geq 2$, w.h.p.\ in the random graph $\Gamma \sim \Gnp$, if $p \gg
(\log n/n)^{1/\Delta}$, any $r$-colourable graph $H$ on $n$ vertices with
$\Delta(H) \leq \Delta$ and colour classes $X_1 \cup \dotsb \cup X_r$, can be
found as a subgraph of every graph $G \subseteq \Gamma$ on vertex set
$\{V_i\}_{i \in [r]}$, with $|V_i| = |X_i|$, where every $(V_i, V_j)$ is
$(\eps,\alpha,p)$-super-regular and $\{V_i\}_{i \in [r]}$ have the regularity
inheritance property. This on one hand completes the quest for a `general
version' of the blow-up lemma applicable to sparse graphs, putting many results
concerning embedding large graphs into the random graph $\Gnp$ under a unified
framework, but on the other leaves a major question unresolved: how sparse can
the graph $G$ actually be?

The assumption $p \gg (\log n/n)^{1/\Delta}$ poses a both `natural' and
`technical' barrier. The former is reflected in the fact that at this point the
random graph allows for a `vertex-by-vertex' type of embedding schemes as
typically every set of at most $\Delta$ vertices has a large common
neighbourhood. The latter, and arguably more difficult to surpass, is related
to the regularity inheritance property. It is known (see \cite{gerke2007small})
that in an $(\eps,p)$-regular pair most sets of size $\Omega(1/p)$ inherit
regularity. Consequently, regularity inheritance can only be established if the
density $p$ is such that $|N_G(v, V_i)| \gg 1/p$, and as typically
$|N_\Gamma(v, V_i)| \approx np$, this forces $p \gg n^{-1/2}$. That being said,
the sparse blow-up lemma of \cite{allen2016blow} is optimal up to the log
factor when $\Delta = 2$ and $H$ contains a triangle, but also when $\Delta =
3$ and $H$ contains many copies of $K_4$ (for more precise details see
\cite[Section~7.2]{allen2016blow}). However, this lower bound on $p$ is
probably very far from the truth in the general case.

The main result of this paper is to break this barrier and show a variant of
the sparse blow-up lemma which is applicable at much lower densities. In order
to fully and precisely state our result we need a definition. A pair
$(V_1,V_2)$ is said to be {$(\eps,p)$-\emph{lower-regular} if for every $V_i'
\subseteq V_i$, with $|V_i'| \geq \eps|V_i|$, the density $d(V_1',V_2')$
satisfies $d(V_1',V_2') \geq d(V_1,V_2) - \eps p$. Let $\cGEk(C_t, n, \eps, p)$
denote the class of graphs whose vertex set is a disjoint union $V_1 \cup
\dotsb \cup V_t$, with all $V_i$ of size $n$, $(V_i,V_{i \pm 1})$ forms an
$(\eps, p)$-regular pair of density $(1\pm\eps)p$, and every $v \in V_i$
satisfies: $\deg_G(v, V_{i \pm 1}) = (1\pm\eps)np$, $|N_G^j(v, V_{i \pm j})|
\geq (1-\eps)(np)^j$ for every $j \in [k-1]$, and
\begin{itemize}
  \item if $t = 2k - 1$, $(N_G^{k-1}(v, V_{i+(k-1)}), N_G^{k-1}(v,
    V_{i-(k-1)}))$ is $(\eps,p)$-lower-regular;
  \item if $t = 2k$, $(N_G^{k-1}(v, V_{i+(k-1)}), V_{i+k})$ and $(N_G^{k-1}(v,
    V_{i-(k-1)}), V_{i-k})$ are $(\eps,p)$-lower-regular.
\end{itemize}

\begin{theorem}\label{thm:blow-up-lemma}
  Let $k \geq 2$ and $t \in \{2k-1, 2k\}$. For every $\alpha > 0$, there exists
  a positive $\eps$ with the following property. For every $\mu > 0$, there is a
  $C > 0$ such that if $p \geq Cn^{-(k-1)/k}$, then w.h.p.\ $\Gamma \sim \Gnp$
  satisfies the following. Every $G \subseteq \Gamma$ which belongs to
  $\cGEk(C_t, \tilde n, \eps, \alpha p)$, with $\tilde n \geq \mu n$, contains a
  disjoint collection of cycles $C_t$ covering all vertices of $G$.
\end{theorem}

This is the first variant of the blow-up lemma, that the author is aware of, in
which the density $p$ is \emph{significantly smaller} than $n^{-1/2}$ (or
$n^{-1/\Delta}$ for that matter), making all the extremely convenient things
that come along regularity inheritance void. Importantly, we do not require $G$
to exhibit the regularity inheritance property among all pairs/triples of sets
in $\{V_i\}_{i \geq 2}$ but only expansion along the edges of $C_t$ as stated
above. This is a rather reasonable assumption, as w.h.p.\ the underlying random
graph $\Gnp$ behaves in a similar way.

The value $p \geq Cn^{-(k-1)/k}$ is optimal in the following way. Suppose $p =
o(n^{-(k-1)/k})$. Then w.h.p.\ in $\Gamma \sim \Gnp$ every set of size
$\eps(np)^{k-1}$ expands to at most $2\eps(np)^k = o(n)$ vertices, so no $G \in
\cGEk(C_{2k}, \tilde n, \eps, \alpha p)$ appears as a subgraph of $\Gamma$. It
may well be that imposing a different natural condition on top of regularity is
not sufficient to go below this bound. Perhaps the only room for improvement
regarding density would be requiring that every vertex of $G$ belongs to
$\Omega(n^{t-1}p^t)$ copies of $C_t$, or in other words, a positive fraction of
all copies it closes in $\Gnp$. Optimistically, under this assumption one can
hope to go all the way down to the natural bound $p \geq Cn^{-(t-2)/(t-1)}$, at
which point w.h.p.\ all copies of $C_t$ can be removed from $\Gnp$ by deleting a
tiny proportion of all edges and the regularity setting stops making sense.

Our proof is based on the absorbing method, which is discussed in great detail
in Section~\ref{sec:absorbing-method}. The theorem itself is then proven in
Section~\ref{sec:blow-up}. Akin to both \cite{balogh2012corradi} and
\cite{allen2016blow}, we showcase the usefulness of our blow-up lemma by
providing an optimal resilience result for the random graph $\Gnp$ with respect
to containing a $C_t$-factor\footnote{An $H$-factor in a graph $G$ is a
vertex-disjoint collection of copies of $H$ covering the whole vertex set of
$G$.}.

Resilience of (random) graphs has received a lot of attention lately, ever since
the paper of Sudakov and Vu~\cite{sudakov2008local} who first coined down the
term officially (even though implicitly it had been studied before, see
e.g.~\cite{alon2000universality}).

\begin{definition}
  Let $G$ be a graph and $\cP$ a monotone\footnote{A graph property is monotone
  if it is preserved under addition of edges.} graph property. We say that $G$
  is $\alpha$-\emph{resilient} with respect to $\cP$, for some $\alpha \in [0,
  1]$, if $G - H$ contains $\cP$ for every $H \subseteq G$ with $\deg_H(v) \leq
  \alpha \deg_G(v)$ for all $v \in V(G)$.
\end{definition}

This notion is in the literature known as \emph{local resilience}. Many of the
famous results in extremal combinatorics can be looked at through the lenses of
resilience. A prime example of those is Dirac's theorem~\cite{dirac1952some}:
every graph on $n$ vertices with minimum degree $\delta(G) \geq n/2$ contains a
Hamilton cycle. In other words, the complete graph on $n$ vertices $K_n$ is
$(1/2)$-resilient with respect to Hamiltonicity. Problems of this type have
recently been intensively studied in sparse random graphs by several groups of
researchers. Some of the most notable results include
Hamiltonicity~\cite{lee2012dirac, montgomery2019hamiltonicity,
nenadov2019resilience}, almost spanning trees~\cite{balogh2011local}, triangle
factors~\cite{balogh2012corradi}, powers of Hamilton
cycles~\cite{fischer2018triangle, vskoric2018local}, bounded degree spanning
subgraphs~\cite{allen2016blow, bottcher2013almost}; for more see the excellent
surveys~\cite{bottcher2017large, sudakov2017robustness} and references therein.

Huang, Lee, and Sudakov~\cite{huang2012bandwidth} were the first to study
resilience of dense random graphs, that is when $p$ is a fixed constant, with
respect to having an (almost-)$H$-factor, for general $H$. Later, as a
consequence of resolving the counting version of the infamous K{\L}R-conjecture,
Conlon, Gowers, Samotij, and Schacht~\cite{conlon2014klr} extended this for $p =
o(1)$. In both a leftover is present, namely the obtained collection of copies
of $H$ covers all but a small fraction of vertices---hence an almost-$H$-factor.
Most recently, Nenadov and \v{S}kori\'{c}~\cite{nenadov2020komlos} went even
further and precisely determined conditions under which the random graph $\Gnp$
is w.h.p.\ resilient with respect to (almost-)$H$-factors and the leftover one
cannot avoid. Among other things they posed a conjecture regarding $C_t$-factors
and highlighted it as one of the more challenging problems to resolve. As the
main application of our blow-up lemma we confirm their conjecture.

\begin{theorem}\label{thm:main-theorem-res}
  Let $k \geq 2$ and $t \in \{2k, 2k+1\}$. For every $\alpha > 0$, there exists
  a positive $C$ such that if $p \geq Cn^{-(k-1)/k}$, then w.h.p.\ $\Gamma \sim
  \Gnp$ is $(1/\chi(C_t)-\alpha)$-resilient with respect to containing a
  $C_t$-factor.
\end{theorem}

This can be viewed as an extension of the result of Balogh, Lee, and
Samotij~\cite{balogh2012corradi} from triangles to longer cycles, and is an
improvement of the result of Allen, B\"{o}ttcher, Ehrenm\"{u}ller, and
Taraz~\cite{allen2020bandwidth} for all cycles of length at least four. (In the
latter, the result for $C_4$ and $C_5$ is already optimal up to the $(\log
n)^{1/2}$ factor in the density $p$, which we now get rid of.)

Our result is optimal in almost every aspect. Firstly, resilience value can be
seen to be the best possible for $C_{2k}$ by choosing a set of size
$n/2-1$ (for even $n$) and disconnecting it from the rest of the graph. As for
$C_{2k+1}$, it seems like the correct value should depend on the so-called
\emph{critical chromatic number} $\chi_\mathrm{cr}(H)$, defined as
\[
  \chi_\mathrm{cr}(H) = \frac{(\chi(H)-1)v(H)}{v(H)-\sigma(H)},
\]
where $\sigma(H)$ is the size of the smallest colour size in a colouring of $H$
with $\chi(H)$ colours (for more details on why this should be the correct
parameter, we refer the reader to \cite{komlos2000tiling, kuhn2009minimum}). In
particular, the resilience value for $C_{2k+1}$ in that case would be $k/(2k+1)$
which is significantly larger than $1/3$ for every $k \geq 2$. We believe that
the importance of obtaining such a result only for odd cycles does not outweigh
the technical difficulties one would face and do not pursue this direction
further.

Secondly, the density $p$ is asymptotically optimal. In order to see this,
assume $t = 5$ and let $v$ be an arbitrary vertex of $\Gamma$. Consider the
second neighbourhood of $v$, $N_\Gamma^2(v)$, and remove all of the edges with
both endpoints lying in it. Obviously, this prevents $v$ from being in a copy of
$C_5$ and moreover, the number of edges removed from any $u \in N_\Gamma^2(v)$
is roughly $(np)^2p$ (this requires proof, see~\cite{nenadov2020komlos}) which
is much smaller than $np$ if $p \ll n^{-1/2}$. This principle can be extended to
`isolate' more than just one vertex $v$ and works similarly for every $t \geq
3$; for more details and precise results for general $H$ we refer the reader
to~\cite{nenadov2020komlos}.

The proof of Theorem~\ref{thm:main-theorem-res} involves a standard argument
using the sparse regularity method and the blow-up lemma
(Theorem~\ref{thm:blow-up-lemma}) and is presented in
Section~\ref{sec:resilience}. There are some intricacies to it, but this is
nothing much out of the ordinary. We see it vaguely plausible that some of our
methods, specifically from the proof of the blow-up lemma, may be applied in
order to obtain a more general result regarding $H$-factors in random graphs
under certain restrictions.

\paragraph*{Notation.} We let $[n] := \{1, \dotsc, n\}$. For $x, y, \eps \in \R$
we write $x \in (y \pm \eps)$ to denote $y - \eps \leq x \leq y + \eps$. We use
standard asymptotic notation $o$, $O$, $\omega$, and $\Omega$, and use $f \ll g$
for $f = o(g)$ and $f \gg g$ for $f = \omega(g)$. Floors and ceilings are
suppressed whenever they are not crucial. If we write e.g.\ $D_{3.3}$, this is
to mean that the value $D$ is the one featured in the statement of
Lemma/Proposition/Claim~3.3. Let $G = (V, E)$ be a graph. For a vertex $v \in
V(G)$ and a set $X \subseteq V(G)$, we use $N_G^i(v, X)$ to denote the set of
vertices $x \in X$ for which there is a $vx$-path of length $i$ (consisting of
$i$ edges) in $G$; then $N_G(v, X)$ stands for $N_G^1(v, X)$. We use $N_G^i(v)$
to denote the \emph{$i$-th neighbourhood of $v$}, that is $N_G^i(v) := N_G^i(v,
V(G))$. Perhaps deviating from standard notation, we let $G - X$ be the graph
obtained from $G$ by removing a set of vertices $X$, and $G - \nabla(X)$ the
graph on the same vertex set as $G$ obtained by removing all edges with at least
one endpoint in $X$ from $G$. The $2$-density of a graph $H$, denoted by
$m_2(H)$, is defined as $m_2(H) := \max_{H' \subseteq H} (e(H')-1)/(v(H')-2)$,
where $H'$ ranges over all subgraphs with at least two edges. For a graph $H$ on
vertices $\{1,\dotsc,t\}$, $\cG(H, n, \eps, p)$ is the class of graphs $G$ whose
vertex set is a disjoint union $V_1 \cup \dotsb \cup V_t$, with all $V_i$ of
size $n$, and $(V_i,V_j)$ forms an $(\eps, p)$-regular pair of density
$(1\pm\eps)p$ if and only if $ij \in E(H)$, these being the only edges of $G$. A
\emph{canonical copy} of a graph $H$ in $G$ is a set $\{v_1,\dotsc,v_t\}$ for
which $v_i \in V_i$ for every $i \in [t]$ and $v_iv_j \in E(G)$ for every $ij
\in E(H)$.

\section{How to prove the blow-up lemma}\label{sec:absorbing-method}

Consider a subgraph $G \subseteq \Gamma$ of $\Gamma \sim \Gnp$ which also
belongs to $\cGEk(C_t, \tilde n, \eps, \alpha p) \subseteq \cG(C_t, \tilde n,
\eps, \alpha p)$. If aiming only to find a very large collection of $t$-cycles
in $G$, one could just employ the resolution of the \emph{K{\L}R-conjecture} in
random graphs, due to Saxton and Thomason~\cite{saxton2015hypergraph} and
independently Balogh, Morris, and Samotij~\cite{balogh2015independent} (see
\cite{conlon2014klr} for a statement most similar to the one tailored to random
graphs as below and \cite{nenadov2021new} for a new simplified proof).

\begin{theorem}[K{\L}R Conjecture]\label{thm:klr}
  For every graph $H$ and every $\alpha > 0$, there exists a positive constant
  $\eps$ with the following property. For every $\mu > 0$, there is a positive
  constant $C$ such that if $p \geq Cn^{-1/m_2(H)}$, then w.h.p.\ $\Gamma \sim
  \Gnp$ satisfies the following. Every $G \subseteq \Gamma$ which belongs to
  $\cG(H, \tilde n, \eps, \alpha p)$, with $\tilde n \geq \mu n$, contains a
  canonical copy of $H$.
\end{theorem}

The theorem above gives only one copy of a graph $H$, but combined with the
`slicing lemma' it easily gives $(1-o(1))\tilde n$ disjoint copies.

\begin{lemma}\label{lem:slicing-lemma}
  Let $0 < \eps_1 < \eps_2 \leq 1/2$, $p \in (0,1)$, and let $(V_1,V_2)$ be an
  $(\eps_1, p)$-regular pair. Then for every $V_i' \subseteq V_i$, $i \in [2]$,
  of size $|V_i'| \geq \eps_2|V_i|$, the pair $(V_1',V_2')$ is $(\eps_1/\eps_2,
  p)$-regular of density $d(V_1,V_2) \pm \eps_1p$.
\end{lemma}

On a very abstract level, the proof strategy for Theorem~\ref{thm:blow-up-lemma}
is now very natural and simple: iteratively find copies of $C_t$ until there are
only some $\rho\tilde n$ vertices remaining uncovered in each $V_i$, and then do
something to cover those as well. This `something' is a brilliant technique very
widely used in a variety of settings nowadays---\emph{the absorbing method}.

The absorbing method has been a key ingredient of numerous results in extremal
combinatorics regarding finding spanning structures both in dense and sparse
regimes. At the heart of the method lies the following idea. One would like to
find a certain (usually highly structured) graph $A \subseteq G$ and a
designated set $W \subseteq V(A)$, which allow for a great deal of flexibility
when constructing the desired spanning graph $\cS$. Namely, no matter how we
manage to embed a fixed subgraph $\cS' \subseteq \cS$ into $G$ so that it covers
all $V(G) \setminus V(A)$ and potentially uses some vertices of $W$, the
leftover vertices $V(G) \setminus V(\cS')$ can be `absorbed' into a complete
embedding of $\cS$. Usually we have no control over which vertices of the
designated set $W$ are used in this partial embedding, so for this to work, the
method fully depends on a very careful choice of the graph $A$---it must be
capable of completing the embedding no matter which vertices of $W$ are already
used. This technique first explicitly appeared in the work on Hamilton cycles in
hypergraphs by R\"{o}dl, Ruci\'{n}ski, and Szemer\'{e}di~\cite{rodl2006dirac},
but was previously used implicitly in the works of Erd\H{o}s, Gy\'{a}rfas, and
Pyber~\cite{erdHos1991vertex}, and Krivelevich~\cite{krivelevich1997triangle}.
As of today, there is quite a substantial body of work in (random) graph theory
utilising the absorbing method (for some specific examples, see
e.g.~\cite{glock2021resolution, joos2019optimal, levitt2010avoid,
montgomery2019spanning, montgomery2020hamiltonicity, montgomery2021proof} and
for a very non-standard application we have drawn some inspiration from, a
recent result~\cite{ferber2020dirac}).

\subsection{Absorbers in sparse regular pairs}

Our goal is to find a graph $A$ and a designated set $W$ in a subgraph $G$ of
$\Gnp$ belonging to $\cGEk(C_t, \tilde n, \eps, \alpha p)$, which have the
capability of `absorbing' the leftover vertices remaining after applying
Theorem~\ref{thm:klr} to $G - V(A)$. The first step in a usual way of doing this
is to find many disjoint \emph{absorbers}. An $R$-absorber, in our context, is a
graph $F$ which is rooted at a set of vertices $R = \{r_1,\dotsc,r_t\}$ and is
such that both $F$ and $F - R$ have a $C_t$-factor. Then the graph $A$ is
obtained by constructing disjoint absorbers rooted on some \emph{strictly
prescribed} $t$-element sets $R \subseteq W$.

Of course, if we were just aiming to construct many disjoint absorbers rooted at
\emph{some sets} $R \subseteq W$, we could turn to Theorem~\ref{thm:klr}, as
long as the absorber $F$ is of constant size. However, for the absorbing
property to be established, it is absolutely necessary that the roots of the
absorbers are chosen in a certain way which makes it impossible to employ this
strategy---Theorem~\ref{thm:klr} has no power of embedding graphs for which some
vertices are already fixed. A cheap attempt at repairing this would be to take
one of the prescribed $t$-element sets $R = \{r_1,\dotsc,r_t\}$ and apply it
with their neighbourhoods $N_G(r_1),\dotsc,N_G(r_t)$. Unfortunately, by looking
at neighbourhoods the regularity between sets is lost, as typically a vertex $v
\in V(G)$ has neighbourhood of size $np \ll n$ for $p = o(1)$ which is not
sufficiently large to `inherit' the $(\eps,p)$-regularity. (Actually, sets of
size $1/p$ typically inherit regularity as well, see~\cite{gerke2007small}, but
this is still not enough when $p = o(n^{-1/2})$.)

A slightly less cheap attempt, and a natural extension of this, is to `expand'
the neighbourhoods of every $r_1,\dotsc,r_t$ some $\ell \geq 1$ times until
$|N_G^\ell(r_i)| \geq \delta\tilde n$, for some $\delta > \eps > 0$, and then
apply Theorem~\ref{thm:klr} with $N_G^\ell(r_i)$. Based on the density $p \gg
n^{-(k-1)/k}$, one expects this to happen when $(np)^\ell \gg n$, i.e.\ when
$\ell = k$. As a dummy example of how this works consider the following scenario
with $k = 2$ and an absorber $F$ for $C_4$ from
Figure~\ref{fig:C4-dummy-absorber}. For every $r_i$ find a set $S_i \subseteq
N_G^2(r_i)$ of size at least $\delta\tilde n$ (as $(np)^2 \gg n$ this is
feasible), so that for every $s_i \in S_i$ there is a copy of the graph on
Figure~\ref{fig:C4-dummy-switcher} between $r_i$ and $s_i$. These sets and
graphs are chosen to be disjoint for different $i$. As $S_i$'s are sufficiently
large and `inherit regularity', we can apply Theorem~\ref{thm:klr} with them to
find the $4$-cycle $s_1,\dotsc,s_4$ with $s_i \in S_i$ which then, due to the
special choice of sets $S_i$, completes a copy of $F$ in $G$. Of course, the
real graph $F$ is going to be much more complex as well as the whole procedure.
In order for this to work it is of utmost importance that an $R$-absorber $F$ is
`locally sparse', or in other words each $N_F^j(r_i)$, $r_i \in R$, is an
independent set for all $1 \leq j \leq k-1$. Otherwise, for reasons going along
the lines of what is said in previous paragraphs, we cannot ensure that an edge
with both endpoints in some $N_F^j(r_i)$ exists in $N_G^j(r_i)$.

\begin{figure}[!htbp]
  \captionsetup[subfigure]{textfont=scriptsize}
  \centering
  \begin{subfigure}{.4\textwidth}
    \centering
    \includegraphics[scale=0.7]{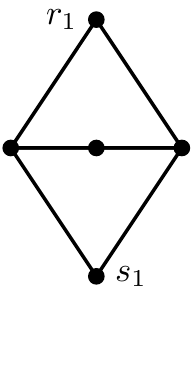}
    \caption{Dummy $C_4$-switcher}
    \label{fig:C4-dummy-switcher}
  \end{subfigure}%
  \begin{subfigure}{.6\textwidth}
    \centering
    \includegraphics[scale=0.7]{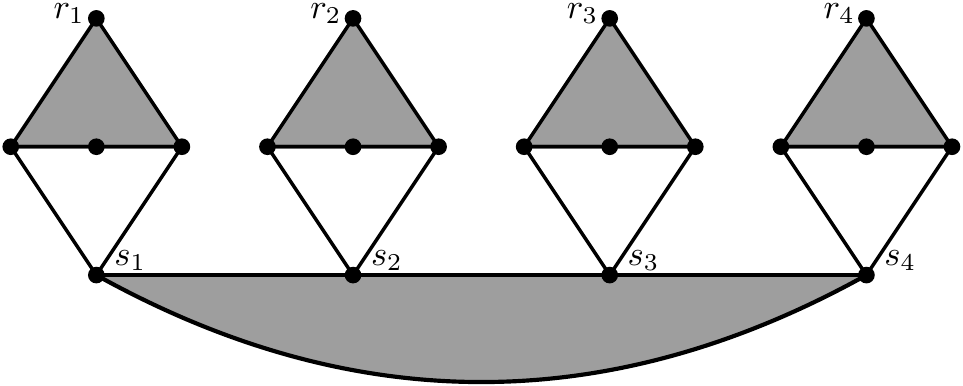}
    \caption{Dummy $C_4$-absorber}
    \label{fig:C4-dummy-absorber}
  \end{subfigure}%
  \caption{A dummy example of an absorber (right) and its building block (left)}
  \label{fig:C4-dummy}
\end{figure}

Lastly, let us mention that the graph $A$ cannot be built in $G$ by greedily
stacking disjoint $R$-absorbers. Namely, there are $\rho\tilde n$ vertices to
`absorb', and even in the best case with $R$-absorbers being of constant size,
there are also roughly $\rho\tilde n$ such graphs we need to `greedily stack'.
As $G$ is living in $\Gnp$ and has minimum degree roughly $np \ll n$, already
after $np$ iterations of using a greedy construction we potentially run out of
space: it can easily happen that the whole neighbourhood of some vertex $w \in
W$ which is prescribed to be the root is already taken. This is circumvented by
using Haxell's matching condition (see Theorem~\ref{thm:haxell-matching} below),
and all the absorbers are to be found in one fell swoop.

\subsection{Switchers, absorbers, and other graph definitions}

Before defining an $R$-absorber we break its structure into even smaller pieces
which we call \emph{switchers}. A \emph{switcher} with respect to a $C_t$-factor
(whenever we say `switcher' we mean `switcher with respect to a $C_t$-factor'),
is a graph $H$ which contains specified vertices $u$ and $v$ and is such that
both $H - v$ and $H - u$ have a $C_t$-factor. A construction that first comes to
mind is to take a path on $t-1$ vertices and connect its endpoints to both $u$
and $v$ (see Figure~\ref{fig:C4-dummy-switcher}). However, such a graph contains
$C_4$ as a subgraph, and as we plan on finding absorbers within sparse regular
pairs, and $m_2(C_4) = 3/2$, we should not hope to find anything that has $C_4$
as a subgraph at density $p = o(n^{-2/3})$, or whenever $k \geq 4$
(equivalently, $t \geq 7$). It turns out that finding a suitable construction as
above is easier said than done.

Let $T$ be a $(t-1)$-ary tree of depth $k$ rooted at a vertex $v$. Replace every
vertex of $T$ by a $t$-cycle $C_t$ and choose an arbitrary vertex from the cycle
on depth $0$ as the root and label it by $v$. Additionally, for every edge of
$T$, identify any two vertices belonging to two cycles corresponding to
endpoints of the edge in such a way that every vertex, other than $v$ and
$(t-1)^{k+1}$ vertices belonging to the cycles on depth $k$, belong to exactly
two cycles. We say that a graph obtained this way is a $C_t$-tree of depth $k$
rooted at $v$. We usually omit saying `of depth $k$' and `rooted at $v$' when
this is clear from the context. The definition of a $C_t$-tree has a much more
natural visual representation, as shown on Figure~\ref{fig:C3-tree}.

\begin{figure}[!htbp]
  \centering
  \includegraphics[scale=0.7]{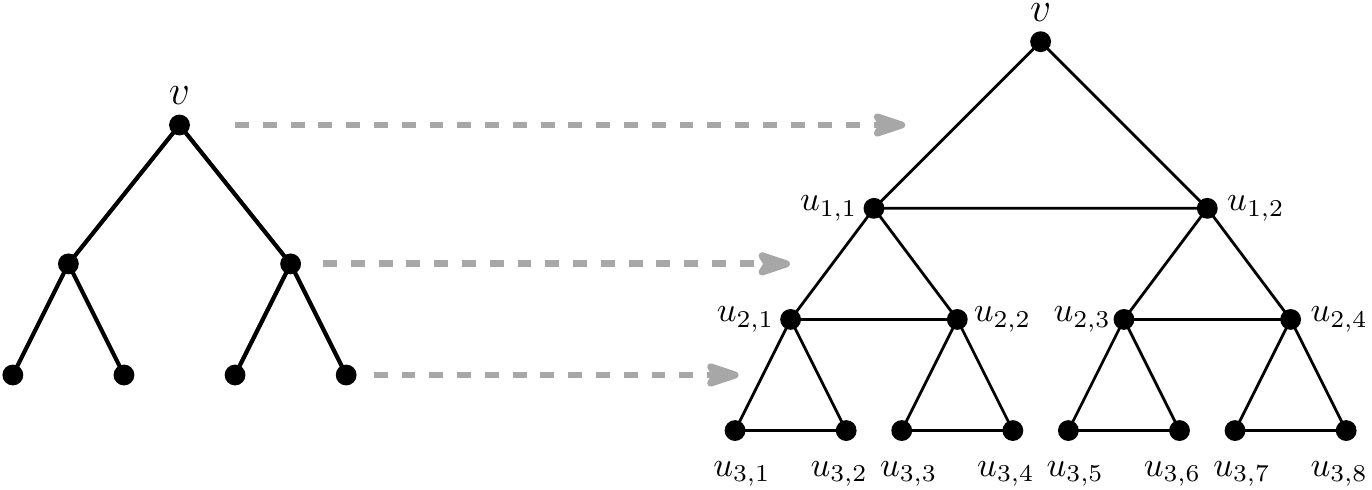}
  \caption{An example of a $C_3$-tree of depth $k = 2$.}
  \label{fig:C3-tree}
\end{figure}

One may think of the vertices of a $C_t$-tree as arranged on levels with $v$ as
the root, and level $i$ consisting of $(t-1)^i$ vertices split into groups of
$t-1$, each group closing a cycle with one (distinct) vertex on level $i-1$ (see
Figure~\ref{fig:C3-tree} again).

For ease of reference, which is used later in the proof, we label the vertices
of a $C_t$-tree:
\begin{itemize}
  \item (level $0$): vertex $v$ is considered to be the root and gets label
    $u_{0,1}$;
  \item (levels $1$ to $k+1$): vertices belonging to a cycle together with a
    vertex $u_{i,j}$, for some $0 \leq i \leq k$, $1 \leq j \leq (t-1)^i$, get
    labels $u_{i+1,(j-1)(t-1)+1}, u_{i+1,(j-1)(t-1)+2}, \dotsc,
    u_{i+1,(j-1)(t-1)+t-1}$ such that
    \[
      u_{i,j}, u_{i+1,(j-1)(t-1)+1}, \dotsc, u_{i+1,(j-1)(t-1)+t-1}
    \]
    is a $t$-cycle in a $C_t$-tree.
\end{itemize}

We next list several graphs which are used as gadgets in order to construct
switchers and combine them into an $R$-absorber.

An $(a, b)$-\textbf{ladder of length $\ell$}, for $\ell$ odd, is a graph $G$
defined as follows:
\begin{itemize}
  \item the vertex set of $G$ is
    \[
      V(G) = \bigcup_{\substack{i \in [\ell] \\ i \text{ odd}}} \{w_{i,1},
      \dotsc, w_{i,a}\} \cup \bigcup_{\substack{i \in [\ell] \\ i \text{ even}}}
      \{w_{i,1}, \dotsc, w_{i,b}\};
    \]
  \item $w_{1,1}, \dotsc, w_{\ell,1}$ and $w_{1,a}, w_{2,b}, \dotsc,
    w_{\ell-1,b}, w_{\ell,a}$
    are paths of length $\ell-1$;
  \item $w_{i,1}, \dotsc, w_{i,a}$ is a path for every odd $i$;
  \item $w_{i,1}, \dotsc, w_{i,b}$ is a path for every even $i$.
\end{itemize}
So this graph looks like a `ladder' where the `steps' are paths of two different
alternating lengths (see Figure~\ref{fig:ladder}).

\begin{figure}[!htbp]
  \centering
  \includegraphics[scale=0.7]{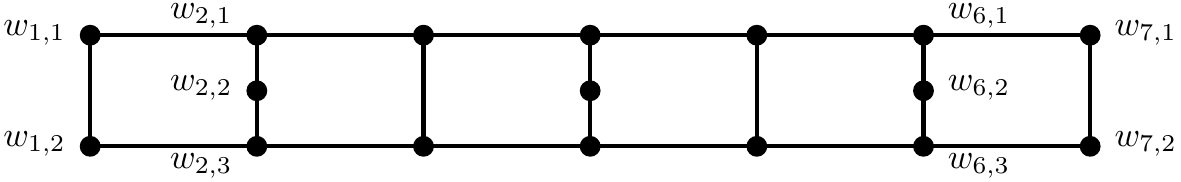}
  \caption{An example of a $(2,3)$-ladder of length $\ell = 7$.}
  \label{fig:ladder}
\end{figure}

Two $t$-cycles, $v_1, \dotsc, v_t$ and $u_1, \dotsc, u_t$, are said to be
$k$-\textbf{ladder-connected} if there exist an $(a_1,b_1)$-ladder and an
$(a_2,b_2)$-ladder, both of length $2k-1$, which are vertex-disjoint and such
that:
\begin{itemize}
  \item $a_1 = k-1$, $a_2 = t-k$, $b_i = t-a_i$, for $i \in \{1,2\}$;
  \item vertices $\{w_{1,j}\}_{j \in [a_1]}$ and $\{w_{1,j}\}_{j \in [a_2]}$ are
    identified with $v_2, \dotsc, v_k$ and $v_{k+1},\dotsc,v_t$;
  \item vertices $\{w_{2k-1,j}\}_{j \in [a_1]}$ and $\{w_{2k-1,j}\}_{j \in
    [a_2]}$ are identified with $u_2, \dotsc, u_k$ and $u_{k+1},\dotsc,u_t$;
\end{itemize}

Two $C_t$-trees of depth $k$ rooted at $v$ and $v'$ are said to be
$k$-\textbf{ladder-connected} if their respective cycles given by the vertices
on the $k$-th and $(k+1)$-st levels are all pairwise $k$-ladder-connected. That
is, for every $j \in [(t-1)^k]$, the two cycles
\[
  u_{k,j}, u_{k+1,(j-1)(t-1)+1}, \dotsc, u_{k+1,(j-1)(t-1)+t-1} \enskip
  \text{and} \enskip u_{k,j}', u_{k+1,(j-1)(t-1)+1}', \dotsc,
  u_{k+1,(j-1)(t-1)+t-1}'
\]
are $k$-ladder-connected. Finally, say that a graph obtained this way is a
$(v,v')$-\textbf{switcher} and denote it by $\Fsw$; indeed, it contains a
$C_t$-factor in both $\Fsw - v$ an $\Fsw - v'$, see
Figure~\ref{fig:C3-switcher}.

\begin{figure}[!htbp]
  \centering
  \begin{subfigure}{.49\linewidth}
    \centering
    \includegraphics[width=0.8\linewidth]{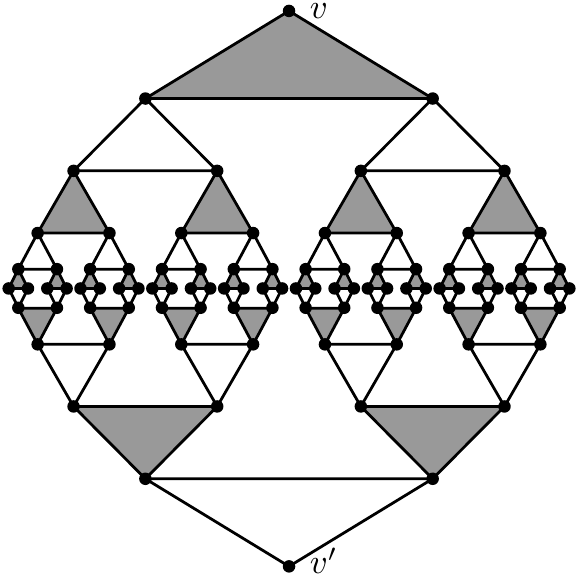}
  \end{subfigure}%
  \hfill
  \begin{subfigure}{.49\linewidth}
    \centering
    \includegraphics[width=0.8\linewidth]{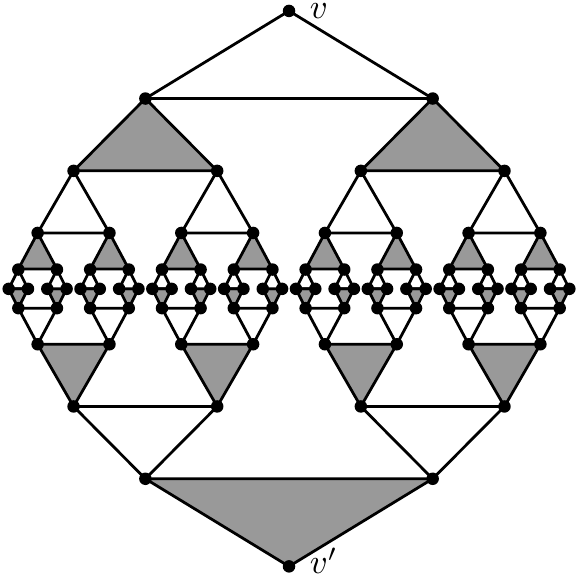}
  \end{subfigure}%
  \caption{An example of a $(v,v')$-switcher $\Fsw$ with two $C_3$-factors. The
  figure on the left represents a $C_3$-factor in $\Fsw - v$ and the
  figure on the right represents a $C_3$-factor in $\Fsw - v'$.}
  \label{fig:C3-switcher}
\end{figure}

An \textbf{$R$-absorber} $\Fabs$ for a set $R = \{r_1,\dotsc,r_t\}$ is a graph
which consists of a $t$-cycle $s_1,\dotsc,s_t$ and a collection of disjoint
$(s_i,r_i)$-switchers $\Fsw$ for every $i \in [t]$. We let $\Fconn$ denote the
graph obtained by contracting every $C_t$-tree of depth $k-1$ (not $k$!) rooted
at $r_i$ of $\Fabs$ individually into a vertex and $\Fabs^{-}$ the subgraph of
$\Fabs$ obtained by removing those $C_t$-trees. The proof of the following
proposition is rather straightforward (but tedious) and for cleaner exposition
we postpone it to the appendix.

\begin{proposition}\label{prop:absorber-properties}
  Let $k \geq 2$ and $t \in \{2k-1,2k\}$. Then the $R$-absorber $\Fabs$
  satisfies the following:
  \stepcounter{propcnt}
  \begin{alphenum}
    \item\label{abs-two-factors} both $\Fabs$ and $\Fabs - R$ have a
      $C_t$-factor,
    \item\label{abs-m2-density} $m_2(\Fconn) \leq k/(k-1)$, and
    \item\label{abs-blow-up} $\Fabs$ is a subgraph of $\cG(C_t, v(\Fabs), 0,
      1)$.
  \end{alphenum}
\end{proposition}

\subsection{From $R$-absorbers to the highly structured graph $A$}

Finally, in order to build the graph $A$ from $R$-absorbers, we rely on a
so-called \emph{template graph}. The first usage of this strategy goes back to
Montgomery~\cite{montgomery2019spanning} and is highly versatile when one has to
absorb several vertices at the same time. We make use of the following
straightforward generalisation of \cite[Lemma~6.1]{mousset2020covering} which is
itself a slight modification of \cite[Lemma~10.7]{montgomery2019spanning} of
Montgomery. It turns out to be a bit more tailored to our needs as opposed to
the original lemma.

\begin{lemma}\label{lem:abs-template}
  There is an integer $m_0$ such that, for every $m \geq m_0$, there exists a
  $t$-partite $t$-uniform hypergraph $B$ on vertex classes $B_1, \dotsc, B_t$,
  with $|B_i| = 2m$ and $\Delta(B) \leq 40^t$, as well as sets $B_i' \subseteq
  B_i$, with $|B_i'| = m$, satisfying the following. For every $Z \subseteq
  \bigcup_{i\in[t]} B_i'$ with $|Z \cap B_1'| = \dotsb = |Z \cap B_t'|$, the
  graph $B - Z$ contains a perfect matching.
\end{lemma}

To connect this to the previous part of the story, the template graph $B$ is
what \emph{strictly prescribes} which $t$-element subsets of the `designated
set' $W = W_1 \cup \dotsb \cup W_t$ need to be roots of an absorber $\Fabs$ in
the following way. Let $B$ be the template graph given by
Lemma~\ref{lem:abs-template} and let $f$ be a bijection mapping vertices of
$B_i$, $i \in [t]$, to $W_i \cup X_i$, where $X_i$'s are some disjoint sets, so
that $W_i \subseteq f(B_i')$. Then for every edge $e = \{b_1,\dotsc,b_t\} \in
E(B)$ construct an $R$-absorber in $G$ for $R = \{f(b_1),\dotsc,f(b_t)\}$. By
the defining property of $B$, for every set $Z \subseteq \bigcup_{i\in[t]} B_i'$
for which $|Z \cap B_1'| = \dotsb = |Z \cap B_t'|$ there is a perfect matching
in $B - Z$. For every edge in this perfect matching take the $C_t$-factor in the
$R$-absorber corresponding to this edge which covers all vertices of $\Fabs$,
and for all other edges take the $C_t$-factor in $\Fabs \setminus R$. The union
of all these copies of $\Fabs$ is then declared to be the graph $A$. This
essentially gives us $W_i$ as sets which we can use in order to `match up' the
$\rho\tilde n$ leftover vertices from Theorem~\ref{thm:klr}. Namely, for each of
$\rho\tilde n$ leftover vertices $v \in V_i$, we find a canonical copy of $C_t$
in $G\big[\{v\} \cup \bigcup_{j \in [t]\setminus\{i\}} W_j\big]$. The remainder
of $W$ is then `absorbed' into a $C_t$-factor using $A$.

\section{Random graphs and expansion}\label{sec:random-graphs}

An invaluable tool in random graph theory is Chernoff's inequality (see,
e.g.~\cite[Corollary~2.3]{janson2011random}).

\begin{lemma}[Chernoff's inequality]\label{lem:chernoff}
  Let $n \in \N$, $p \in [0,1]$, and let $X \sim \Bin(n,p)$. For every $\delta
  \in (0, 3/2)$,
  \[
    \Pr\big[X \notin (1\pm\delta)\E[X]\big] \leq 2e^{-\frac{\delta^2}{3} \E[X]}.
  \]
\end{lemma}
The inequality is also true if $X$ is a geometrically distributed random
variable (instead of binomially), which we use at several places in the proof.

Next, we list a couple couple of properties of random graphs which are no
surprise to experts and can be proven via a standard usage of Chernoff's
inequality and the union bound. First is a bound on the size of the $k$-th
neighbourhood of sets.

\begin{lemma}\label{lem:gnp-k-expansion}
  For every $k \in \N$ and $\nu > 0$, there exists a positive constant $C$ such
  that if $p \geq C\log n/n$ then w.h.p.\ $\Gamma \sim \Gnp$ satisfies the
  following. For every $X \subseteq V(\Gamma)$ of size $|X| \leq
  \nu/(n^{k-1}p^k)$, we have $|N_\Gamma^k(X)| \geq (1-k\nu)|X|(np)^k$.
\end{lemma}

We also need the following property about distribution of edges in random graphs
(see, e.g.~\cite[Corollary~2.3]{krivelevich2006pseudo}).

\begin{proposition}\label{prop:gnp-edges}
  With high probability $\Gamma \sim \Gnp$ satisfies the following for any $p :=
  p(n) \leq 0.99$. For every two (not necessarily disjoint) sets $X, Y \subseteq
  V(\Gamma)$, the number of edges with one endpoint in $X$ and the other in $Y$
  satisfies
  \[
    e_\Gamma(X, Y) \leq |X||Y|p + c\sqrt{|X||Y|np},
  \]
  for some absolute constant $c > 0$.
\end{proposition}

The next one comes in handy when wanting to show expansion of sets which is
implied only by a minimum degree condition in subgraphs of $\Gnp$.

\begin{lemma}\label{lem:gnp-min-deg-expansion}
  For every $\mu > 0$, there exists a positive constant $K$ such w.h.p.\ $\Gamma
  \sim \Gnp$ satisfies the following for every $p \in (0,1)$. There are no two
  sets $X, Y \subseteq V(\Gamma)$ with $|X| \geq K/p$, $|Y| \leq \mu n$, and
  $e_\Gamma(X, Y) \geq 2\mu|X|np$.
\end{lemma}

It turns out that the minimum degree assumption for a subgraph $G$ of $\Gamma
\sim \Gnp$ is sufficient to find many disjoint copies of $t$-cycles in $G$,
under certain conditions. For this (and things to come) we make use of a
hypergraph matching condition due to Haxell~\cite{haxell1995condition}.

\begin{theorem}[Haxell's condition~\cite{haxell1995condition}]
  \label{thm:haxell-matching}
  Let $\cH = (A \cup B; E)$ be an $\ell$-uniform hypergraph with $|e \cap A| =
  1$ and $|e \cap B| = \ell-1$ for every edge $e \in E$. Suppose that for every
  subset $A' \subseteq A$ and $B' \subseteq B$ with $|B'| \leq
  (2\ell-3)(|A'|-1)$, there is an edge $e \in E$ intersecting $A'$ but not $B'$.
  Then there is an $A$-saturating matching in $\cH$ (a collection of disjoint
  edges whose union contains $A$).
\end{theorem}

\begin{lemma}\label{lem:gnp-cycles-from-min-deg-disjoint}
  Let $k \geq 2$ and $t \in \{2k, 2k+1\}$. For every $\alpha, \mu > 0$, there
  exists a $\delta > 0$ with the following property. For every $D > 0$ there
  exists a $C > 0$ such that if $p \geq Cn^{-(k-1)/k}$ then w.h.p.\ $\Gamma \sim
  \Gnp$ satisfies the following. Let $G \subseteq \Gamma$ and $X, U \subseteq
  V(G)$ be disjoint sets of size $|U| \geq \mu n$ and $|X| \leq \delta|U|$.
  Assume $\deg_G(v, U) \geq \alpha|U|p$ for all $v \in X \cup U$ and all but
  $D/p$ vertices $u \in U$ satisfy $\deg_G(u, U) \geq (1/2+\alpha)|U|p$. Then
  there is a collection of disjoint $t$-cycles in $G[X \cup U]$ covering all
  vertices of $X$.
\end{lemma}
\begin{proof}
  Let $c = c_{\ref{prop:gnp-edges}}$ be the absolute constant from
  Proposition~\ref{prop:gnp-edges}. We choose $\delta = \delta(\alpha,\mu,t) >
  0$ sufficiently small and, given $D > 0$, choose $K = K(\alpha,\mu,t,D)$
  sufficiently large. Next, fix a small $\eps > 0$ and let $C$ be large enough
  with respect to all prior constants. As the conclusion of
  Proposition~\ref{prop:gnp-edges} holds with high probability for $\Gamma$, we
  may condition on this throughout the proof. We need an auxiliary claim first.

  \begin{claim}\label{cl:expansion}
    Let $S, T \subseteq X \cup U$ be disjoint sets with $|T| \geq |U|/t^2$ and
    assume every $v \in S$ satisfies $\deg_G(v, T) \geq \alpha|T|p/2$ and all
    but $K/p$ vertices $u \in S$ satisfy $\deg_G(u, T) \geq (1/2+\alpha/2)|T|p$.
    Then
    \[
      |N_G(S, T)| \geq
      \begin{cases}
        \eps|S|np, & \text{ if $|S| < 2K/p$}, \\
        (1/2+\alpha/4)|T|, & \text{ if $|S| \geq 2K/p$}.
      \end{cases}
    \]
  \end{claim}
  \begin{proof}
    If $|S| < 2K/p$, by setting $Z := N_G(S, T)$ from the minimum degree
    assumption and Proposition~\ref{prop:gnp-edges} we have
    \[
      \frac{\alpha\mu}{2t^2}|S|np \leq \frac{\alpha|S||T|p}{2} \leq e_G(S, Z)
      \leq |S||Z|p + c\sqrt{|S||Z|np} < 2K|Z| + c\sqrt{|S||Z|np},
    \]
    which leads to a contradiction if $|Z| < \eps|S|np$, for $\eps$
    sufficiently small. If $|S| \geq 2K/p$ let $S'$ be the set of vertices with
    degree at least $(1/2+\alpha/2)|T|p$ into $T$ and assume $|Z| <
    (1/2+\alpha/4)|T|$. Then
    \[
      |S'|(1/2+\alpha/2)|T|p \leq e_G(S', Z) \leq |S'||Z|p + c\sqrt{|S'||Z|np}
      < |S'|(1/2+\alpha/4)|T|p + c\sqrt{|S'||Z|np},
    \]
    which again leads to a contradiction as $c\sqrt{|S'||Z|np} <
    (\alpha/4)|S'||T|p$, for $K$ sufficiently large.
  \end{proof}

  Let $\cH$ be an auxiliary $t$-uniform hypergraph on vertex set $X \cup U$ in
  which there is an edge $\{x\} \cup Y$ for $x \in X$ and $Y \subseteq U$, $|Y|
  = t-1$, if and only if there is a $t$-cycle in $G$ induced by $x$ and $Y$.
  Let $X' \subseteq X$ and $U' \subseteq U$, $|U'| \leq (2t-3)|X'|$. By
  Theorem~\ref{thm:haxell-matching} in order to complete the proof it is
  sufficient to show that there is a cycle $C_t$ in $G$ with one vertex in $X'$
  and otherwise completely lying in $U \setminus U'$.

  Let $U_1 \cup \dotsb \cup U_{t-1}$ be a uniformly random equipartition of $U$.
  A simple application of Chernoff's inequality and the union bound shows that
  with high probability all $v \in X \cup U$ satisfy $\deg_G(v, U_i) \geq
  \alpha|U_i|p/2$, and all but $D/p$ vertices $u \in U$ satisfy $\deg_G(u, U_i)
  \geq (1/2+\alpha/2)|U_i|p$, for all $i \in [t-1]$. Fix a choice of such sets
  for the remainder. For a fixed choice of $X'$ and $U'$ as above, let $G_1 :=
  G[X' \cup (U_1 \cup \dotsb \cup U_k) \setminus U']$ and $G_2 := G[X' \cup
  (U_{t-1} \cup \dotsb \cup U_{t-k}) \setminus U']$ ignoring edges with both
  endpoints in some $U_i$.

  Let $S \subseteq X'$ be of size $\ceil{|X'|/4}$. In the following we show that
  there is a $v \in S$ for which $|N_{G_1}^k(v)| \geq (1/2+\alpha/8)|U_k|$.
  First, we argue how this implies what we want, i.e.\ a cycle $C_t$ with one
  vertex in $X'$ and otherwise lying in $U \setminus U'$. As $S$ is arbitrary, we
  conclude there are at least $3|X'|/4$ vertices $u \in X'$ with $|N_{G_1}^k(u)|
  \geq (1/2+\alpha/8)|U_k|$ and analogously at least $3|X'|/4$ vertices $v \in
  X'$ with $|N_{G_2}^k(v)| \geq (1/2+\alpha/8)|U_{t-k}|$. In particular, there
  is a vertex $x \in X'$ with both
  \[
    |N_{G_1}^k(x)| \geq (1/2+\alpha/8)|U_k| \qquad \text{and} \qquad
    |N_{G_2}^k(x)| \geq (1/2+\alpha/8)|U_{t-k}|.
  \]
  If $t = 2k$ this implies there is a cycle containing $x$ and otherwise
  completely in $U \setminus U'$. If $t = 2k+1$, then an edge in $G$ between
  $N_{G_1}^k(x)$ and $N_{G_2}^k(x)$ would again close such a cycle. This edge
  has to exist, as otherwise Claim~\ref{cl:expansion} applied with
  $N_{G_1}^k(v)$ (as $S$) and $U_{t-k} \setminus N_{G_2}^k(v)$ (as $T$)
  implies $U_{t-k} \setminus N_{G_2}^k(v)$ is larger than $|U_{t-k}|/2$, which
  is a contradiction.

  Therefore, we reduced our goal to showing that there is a $v \in S$ with
  $|N_{G_1}^k(v)| \geq (1/2+\alpha/8)|U_k|$. Assume first $|S| \geq 2K/p$. As
  there are at least $K/p$ vertices in $S$ with degree at least
  $(1/2+\alpha/2)|U_1|p$ into $U_1$, Claim~\ref{cl:expansion} applied with $S$
  and $U_1$ (as $T$) gives
  \[
    |N_{G_1}(S)| \geq |N_G(S, U_1)| - |U'| \geq (1/2+\alpha/4)|U_1| - 3t\delta n
    \geq (1/2+\alpha/8)|U_1|,
  \]
  for $\delta$ sufficiently small. By averaging, and as $|U_1| = \Omega(n)$,
  there is a non-empty set $S_1 \subseteq S$ of size
  \[
    |S_1| \leq \frac{|S|2K/p}{|U_1|/2} = O\Big(\frac{1}{np^2}\Big)
  \]
  for which $|N_{G_1}(S_1)| \geq 2K/p$. Repeatedly applying the above principle,
  that is Claim~\ref{cl:expansion} with $N_{G_1}^i(S_i)$ (as $S$) and $U_{i+1}$
  (as $T$) together with subsequent averaging, shows that there is a non-empty
  set $S_{k-1} \subseteq S$ of size
  \[
    |S_{k-1}| = O\Big(\frac{1}{n^{k-1}p^k}\Big).
  \]
  for which $|N_{G_1}^{k-1}(S_{k-1})| \geq 2K/p$. As $p \geq Cn^{-(k-1)/k}$,
  it follows there is a single vertex $v \in S$ for which $|N_{G_1}^{k-1}(v)|
  \geq 2K/p$ and again $|N_{G_1}^k(v)| \geq (1/2+\alpha/8)|U_k|$, as desired.

  Assume now $|S| < 2K/p$ and recall $|U'| \leq 8t|S| = o(|S|np)$. Using
  Claim~\ref{cl:expansion} with $S$ and $U_1$ (as $T$), we get
  \[
    |N_{G_1}(S)| \geq |N_G(S, U_1)| - |U'| \geq \eps|S|np -
    o(|S|np) \geq (\eps/2)|S|np.
  \]
  Let $z$ be the smallest integer for which $|N_{G_1}^z(S)| \geq 2K/p$; in
  particular, $|N_{G_1}^{z-1}(S)| < 2K/p$. Then this same expansion argument can
  be repeated to obtain
  \[
    |N_{G_1}^z(S)| \geq |N_G(N_{G_1}^{z-1}(S), U_z)| - |U'| \geq |S|(\eps
    np/2)^{z-1} \cdot \eps np - o(|S|np) \geq |S|(\eps np/2)^z.
  \]
  Similarly as before, by averaging there is a non-empty $S_z \subseteq S$ of
  size
  \[
    |S_z| \leq \frac{|S|2K/p}{|S|(\eps np/2)^z} =
    O\Big(\frac{1}{n^zp^{z+1}}\Big),
  \]
  for which $|N_{G_1}^z(S_z)| \geq 2K/p$. Again by Claim~\ref{cl:expansion}, we
  have
  \[
    |N_{G_1}^{z+1}(S_z)| \geq (1/2+\alpha/4)|U_{z+1}| - |U'| \geq
    (1/2+\alpha/4)|U_{z+1}| - 3t\delta n \geq
    (1/2+\alpha/8)|U_{z+1}|,
  \]
  for $\delta$ sufficiently small. Now analogously as in the case $|S| \geq
  2K/p$ find a non-empty set $S_{k-1} \subseteq S$ of size $|S_{k-1}| \leq 1$,
  and thus a single vertex $v \in S$, for which $|N_{G_1}^k(v)| \geq
  (1/2+\alpha/8)|U_k|$, as desired. This completes the proof.
\end{proof}

\begin{lemma}\label{lem:gnp-cycles-from-min-deg}
  Let $k \geq 2$ and $t \in \{2k, 2k+1\}$. For every $\alpha, \mu > 0$, there
  exist positive constants $\delta$ and $C$, such that if $p \geq Cn^{-(k-1)/k}$
  then w.h.p.\ $\Gamma \sim \Gnp$ satisfies the following. Let $G \subseteq
  \Gamma$ and $X, U \subseteq V(G)$ be disjoint sets of size $|U| \geq \mu n$
  and $|X| \leq \delta|U|$. Assume $\delta(G[X \cup U]) \geq (1/2+\alpha)|U \cup
  X|p$. Then there is a collection of disjoint $t$-cycles in $G[X \cup U]$
  covering all vertices of $X$.
\end{lemma}

\begin{proof}
  Let $c = c_{\ref{prop:gnp-edges}}$ be the absolute constant from
  Proposition~\ref{prop:gnp-edges}. Let $K =
  K_{\ref{lem:gnp-min-deg-expansion}}(\alpha\mu/4)$, $\gamma = \alpha\mu/2$,
  $\eps$ be sufficiently small, in particular much smaller than
  $\gamma/2^{k+1}$, and $\delta' =
  \delta_{\ref{lem:gnp-cycles-from-min-deg-disjoint}}(\eps,\mu/k)$. We choose
  $\delta > 0$ sufficiently small and $C \geq 1$ sufficiently large, all
  depending on $\alpha$, $\mu$, and $k$, so that the arguments below go through.
  As the conclusions of Proposition~\ref{prop:gnp-edges},
  Lemma~\ref{lem:gnp-min-deg-expansion}, and
  Lemma~\ref{lem:gnp-cycles-from-min-deg-disjoint} both hold with high
  probability for $\Gamma$, we may condition on this throughout the proof.

  Assume $|U| = \mu n$, as this has no effect on the proof but makes things
  cleaner. Set $Z := \varnothing$ and as long as there is a vertex $u \in U$
  with $\deg_G(u, U \setminus Z) < (1/2+\alpha/2)|U|p$, add it to $Z$. Stop this
  procedure at the first point when $|Z| = \delta|U|$. As then $e_G(Z, X \cup Z)
  \geq \alpha\mu|Z|np/2$ from Lemma~\ref{lem:gnp-min-deg-expansion} with
  $\alpha\mu/4$ (as $\mu$), we get that $|Z| < K/p$. It follows that there is a
  subset $U' \subseteq U$ of size $(1-o(1))|U|$ so that all vertices of $U'$
  have degree at least $(1/2+\alpha/2)|U|p$ into $U'$. Thus, for simplicity, we
  assume that $U$ is already such that $\delta(G[U]) \geq (1/2+\alpha)|U|p$ to
  begin with.

  Our goal is to apply Lemma~\ref{lem:gnp-cycles-from-min-deg-disjoint} to
  certain sets $X' \subseteq X$ and $U' \subseteq X \cup U$ until we cover the
  whole set $X$. For this we need that every vertex of $X'$ has sufficiently
  large degree into $U'$ and that the set of vertices in $U'$ with small degree
  is small.

  Let $Z_1$ be the largest subset of $X$ such that every vertex of $Z_1$ has
  degree less than $(1/2+\alpha/2)|U|p$ into $U$ and set $X_1 := X \setminus
  Z_1$. Then, for every $i \geq 2$, let $Z_i \subseteq Z_{i-1}$ be the largest
  subset such that every vertex of $Z_i$ has degree less than $\gamma
  np/2^{i-1}$ into $X_{i-1}$, and let $X_i := Z_{i-1} \setminus Z_i$. We claim
  that $|Z_i| = O(1/(n^{i-1}p^i))$ for all $i \geq 1$.

  For $i = 1$, observe that every $v \in Z_1$ satisfies
  \[
    \deg_G(v, X) = \deg_G(v, X \cup U) - \deg_G(v, U) \geq (1/2+\alpha)|U|p -
    (1/2+\alpha/2)|U|p \geq \alpha|U|p/2 = \gamma np.
  \]
  Consequently, $e_G(Z_1, X) \geq |Z_1| \alpha\mu np/2$ and by
  Lemma~\ref{lem:gnp-min-deg-expansion} with $\alpha\mu/4$ (as $\mu$), it
  follows that $|Z_1| < K/p$.

  Let $i \geq 2$ and observe that by definition of sets $Z_j$, every $v \in Z_i$
  satisfies
  \[
    \deg_G(v, Z_{i-1}) \geq \deg_G(v, X) - \sum_{j \in [i-1]} \deg_G(v, X_j)
    \geq \gamma np - \sum_{j \in [i-1]} \frac{\gamma np}{2^{j-1}} \geq
    \frac{\gamma np}{2^{i-1}}.
  \]
  By Proposition~\ref{prop:gnp-edges}, and as $|Z_{i-1}| =
  O(1/(n^{i-2}p^{i-1}))$ by induction hypothesis,
  \[
    |Z_i| \cdot \gamma np/2^{i-1} \leq e_G(Z_i,Z_{i-1}) \leq
    2\max\big\{|Z_i||Z_{i-1}|p, c\sqrt{|Z_i||Z_{i-1}|np}\big\} =
    2c\sqrt{|Z_i||Z_{i-1}|np}.
  \]
  Rearranging gives
  \[
    |Z_i| = O\Big(\frac{|Z_{i-1}|}{np}\Big) = O\Big(\frac{1}{n^{i-1}p^i}\Big),
  \]
  as desired.

  Note that, since $p \geq Cn^{-(k-1)/k}$, we have $Z_k = \varnothing$ for $C >
  0$ large enough. In conclusion, there exists a partition $X_1 \cup \dotsb \cup
  X_k = X$ such that
  \begin{enumerate}[label=(\emph{\roman*}), ref=(\emph{\roman*})]
    \item every $v \in X_1$ satisfies $\deg_G(v, U) \geq (1/2+\alpha/2)|U|p$,
    \item for $i \geq 2$, $|X_i| = O(1/(n^{i-2}p^{i-1}))$, and every $v \in X_i$
      satisfies $\deg_G(v, X_{i-1}) \geq \gamma np/2^{i-1}$.
  \end{enumerate}

  For every $i \in [k]$, let
  \[
    U_i \subseteq U \cup \bigcup_{j < i} X_j, \qquad |U_i| = |U|/k,
  \]
  be disjoint sets chosen uniformly at random. Then, by Chernoff's inequality
  and the union bound the following holds with high probability: for every $i
  \in [k]$ and every $v \in X_i$
  \[
    \deg_G(v, U_i) \geq (1-o(1))\deg_G\paren[\big]{v, U \cup \bigcup_{j < i}
    X_j} \cdot \frac{|U_i|}{|U \cup X|} \geq \gamma np/2^k \cdot
    \frac{|U_i|}{(1+\delta)n} \geq \eps|U_i|p,
  \]
  and similarly all but at most $2K/p$ vertices (those in $X_i$'s, $i \geq 2$)
  $u \in U_i$ satisfy $\deg_G(u, U_i) \geq (1/2+\eps)|U_i|p$. Fix such a choice
  of $U_i$'s. This puts us into the setting of
  Lemma~\ref{lem:gnp-cycles-from-min-deg-disjoint} which is applied with $\eps$
  (as $\alpha$), $\mu/k$ (as $\mu$), $2K$ (as $D$), $X_i$ (as $X$), and $U_i$
  (as $U$). We can indeed to this as $|X_i| \leq |X| \leq \delta|U| \leq
  \delta'|U_i|$.
\end{proof}

\subsection{Robustness of expansion in subgraphs of random graphs}

Let $\Gamma \sim \Gnp$ and $G \subseteq \Gamma$. For $k \in \N$, $\alpha, \gamma
> 0$, and disjoint vertex sets $V_1, \dotsc, V_k \subseteq V(G)$, all of size
$\tilde n$, a vertex $v \in V(G)$ is said to be {\em $(\gamma,k)$-expanding}
with respect to $V_1,\dotsc,V_k$, if $|N_G^i(v, V_i)| \geq (1-\gamma)(\tilde
n\alpha p)^i$, for all $i \in [k]$. Of course, to be fully formally correct, the
definition should also include parameters $\tilde n$, $\alpha$, and $p$, but we
omit those as they are always clear from the context and would just introduce
more clutter.

As with many similar properties, expansion is `inherited' to sufficiently large
random subsets.

\begin{lemma}\label{lem:reg-exp-partitioning}
  Let $k \in \N$. For every $\gamma, \delta > 0$ there exists a positive
  constant $\eps$ such that the following holds for sufficiently large $n$ and
  every $p = p(n) \in (0,1)$. Let $G$ be a graph on $n$ vertices, $V_1, \dotsc,
  V_k \subseteq V(G)$ be disjoint sets such that $|V_i| = \dotsb = |V_k| =
  \tilde n$, with $\tilde n \geq \log^2 n/p$, and suppose $\Delta(G[V_i,
  V_{i+1}]) \leq (1+\eps)\tilde np$. Let $U_i \subseteq V_i$ be chosen uniformly
  at random among all subsets of size $\delta\tilde n$. Then, with high
  probability, $\Delta(G[U_i,U_{i+1}]) \leq (1+\gamma)\delta\tilde np$, and
  every vertex that was $(\eps,k)$-expanding with respect to $V_1,\dotsc,V_k$ is
  $(\gamma,k)$-expanding with respect to $U_1,\dotsc,U_k$.
\end{lemma}
\begin{proof}
  First, a simple application of Chernoff's inequality and the union bound shows
  that $\Delta(G[U_i,U_{i+1}]) \leq (1+\gamma)\delta\tilde np$, with probability
  at least $1 - e^{-\Omega(\tilde np)}$.

  Write $s := \delta \tilde n$ and let $G' := G[U_1 \cup \dotsb \cup U_k]$. Fix
  $v$ which is $(\eps,k)$-expanding with respect to $V_i$'s and choose $\rho >
  0$ sufficiently small. Let $\cE_i$, for $i \in [k]$, denote the event that
  $|N_{G'}^i(v, U_i)| \geq (\delta^i-i\rho)|N_G^i(v, V_i)|$. We show that, for
  every $i \in [k-1]$, conditioning on $\cE_1 \land \dotsb \land \cE_i$, the
  event $\cE_{i+1}$ holds with probability at least $1 -
  e^{-\Omega(s^{i+1}p^{i+1})}$. This surely holds for $i = 1$ similarly as above
  for the maximum degree.

  Observe first that, for every $i \in [k-1]$, every set $X \subseteq
  N_G^i(v,V_i)$ of size $(\delta^i-i\rho)|N_G^i(v, V_i)|$ deterministically
  satisfies
  \begin{align*}
    |N_G(X, V_{i+1})| &\geq |N_G^{i+1}(v, V_{i+1})| - |N_G^i(v, V_i) \setminus
    X|(1+\eps)\tilde np \\
    &\geq |N_G^{i+1}(v, V_{i+1})| - (1-\delta^i+i\rho)|N_G^i(v,
    V_i)|(1+\eps)\tilde np.
  \end{align*}
  By the fact that $|N_G^i(v)|\tilde np \leq (1+\eps)^i/(1-\eps)|N_G^{i+1}(v)|$,
  this further implies (with room to spare)
  \[
    |N_G(X, V_{i+1})| \geq (\delta^i-i\rho-10k\eps)|N_G^{i+1}(v, V_{i+1})|,
  \]
  for sufficiently small $\eps > 0$. Therefore, as $U_{i+1} \subseteq V_{i+1}$
  is chosen uniformly at random, conditioning on $\cE_i$ and using $N_{G'}^i(v,
  U_i)$ as $X$, by Chernoff's inequality with probability at least
  $1-e^{-\Omega(s^{i+1}p^{i+1})}$ we have
  \[
    |N_{G'}^{i+1}(v, U_{i+1})| \geq (1-o(1)) \cdot \delta
    (\delta^i-i\rho-10k\eps)|N_G^{i+1}(v, V_{i+1})| \geq (1-\gamma)(sp)^{i+1},
  \]
  where we used the fact that we can choose $\eps$ and $\rho$ appropriately
  small depending on $\gamma$, $\delta$, and $k$.

  In conclusion, the probability that $v$ is $(\gamma,k)$-expanding with respect
  to $U_1,\dotsc,U_k$ is at least
  \[
    \prod_{i \in [k]}\big(1 - \Pr[\cE_i]\big) \geq \prod_{i \in [k]}\big(1 -
    e^{-\Omega(s^ip^i)}\big) \geq 1-o(n^{-6}).
  \]
  By the union bound over all vertices $v \in V(G)$ we get that with probability
  at least $1 - o(n^{-5})$ the desired property holds.
\end{proof}

The next couple of lemmas are very similar to each other. In a nutshell, they
all show that in a subgraph $G \subseteq \Gamma$, being $(\gamma,k)$-expanding
with respect to some sets $V_1,\dotsc,V_k$ is \emph{robust}. Namely, even after
the `removal' of a not too large set $Q$ most of the vertices remain
$(\gamma',k)$-expanding with respect to $V_1, \dotsc, V_k$ in $G - \nabla(Q)$,
for a suitable $\gamma'$. The different lemmas cover the different ranges on the
size of $Q$.

\begin{lemma}\label{lem:robust-reg-expansion-large-sets}
  For every $k \geq 1$ and all $\alpha, \gamma > 0$, there exist positive
  constants $\eps$ and $\delta$ with the following property. For every $\mu > 0$
  there exists a $K > 0$ such that for every $p \in (0,1)$ w.h.p.\ $\Gamma \sim
  \Gnp$ satisfies the following. Let $G \subseteq \Gamma$, $\tilde n = \mu n$,
  and let $U, V_1, \dotsc, V_k \subseteq V(G)$ be disjoint sets such that:
  \begin{itemize}
    \item $|V_1| = \dotsb = |V_k| = \tilde n$,
    \item $\deg_G(v, V_{i+1}) \leq (1+\eps)\tilde n\alpha p$, for all $v \in
      V_i$, $i \in [k-1]$, and
    \item every $v \in U$ is $(\eps,k)$-expanding with respect to
      $V_1,\dotsc,V_k$.
  \end{itemize}
  Then for every $Q \subseteq V(G) \setminus U$ of size $|Q| \leq \delta\tilde
  n$, all but $K/p$ vertices $v \in U$ are $(\gamma,k)$-expanding with respect
  to $V_1,\dotsc,V_k$ in $G - \nabla(Q)$.
\end{lemma}
\begin{proof}
  Given $k$, $\alpha$, and $\gamma$, let $\eps$ be sufficiently small for the
  argument below to go through, and let $\delta = \eps\alpha\mu/4$ and $K =
  K_{\ref{lem:gnp-min-deg-expansion}}(\eps\alpha\mu/2)$. Assume that $\Gamma
  \sim \Gnp$ is such that it satisfies the conclusion of
  Lemma~\ref{lem:gnp-min-deg-expansion}, which happens with high probability.

  Write $V_0 := U$, let $Z_k = \varnothing$ and for every $i = k-1, \dotsc, 0$,
  let $Z_i \subseteq V_i$ be defined as
  \[
    Z_i := \big\{v \in V_i : \deg_G(v, Q \cup Z_{i+1}) > \eps\tilde n\alpha
    p\big\}.
  \]
  For convenience, we write $G' := G - (Q \cup \bigcup_{i\in[k]} Z_i)$ and
  for $F \in \{G, G'\}$ and $v \in U$ use $N_F^i(v)$ to mean $N_F^i(v, V_i)$,
  for all $i \in [k]$. We claim that $|Z_i| < K/p$ for every $i \in [k-1]$. This
  readily follows from Lemma~\ref{lem:gnp-min-deg-expansion} with
  $\eps\alpha\mu/2$ (as $\mu$) and $Z_i$ (as $X$). Namely, by letting $Y = Q
  \cup Z_{i+1}$, we have
  \[
    e_\Gamma(Z_i, Y) \geq e_G(Z_i, Y) > |Z_i|\eps\tilde n\alpha p \geq
    \eps\alpha\mu|Z_i|np,
  \]
  and thus $|Y| > (\eps\alpha\mu/2)n = 2\delta\tilde n$---a contradiction with
  the assumption on the size of $Q$. In particular, all but $K/p$ vertices $v
  \in U$ satisfy $|N_{G'}(v)| \geq (1-2\eps)|N_G(v)|$.

  We aim to show that for every $v \in U \setminus Z_0$, $|N_{G'}^i(v)| \geq
  (1-2^i\eps)|N_G^i(v)|$, for all $i \in [k]$, which is sufficient for the lemma
  to hold. Consider $N_{G'}^i(v)$, for some $2 \leq i \leq k$. Let $x_{i-1} \in
  [0,1]$ denote the fraction of vertices in $N_G^{i-1}(v)$ which belong to $Q
  \cap Z_{i-1}$. Then a simple calculation using the bound on the maximum degree
  leads to
  \[
    |N_{G'}^i(v)| \geq |N_G^i(v)| - x_{i-1}|N_G^{i-1}(v)|(1+\eps)\tilde n\alpha
    p - (1-x_{i-1})|N_G^{i-1}(v)|\eps\tilde n\alpha p.
  \]
  Applying the induction hypothesis for $i-1$ we get
  \begin{align*}
    |N_{G'}^i(v)| & \geq |N_G^i(v)| - |N_G^{i-1}(v)|\big(2^{i-1}\eps +
    \eps\big)(1+\eps)\tilde n\alpha p \\
    & = |N_G^i(v)| - |N_G^{i-1}(v)|(1+2^{i-1})\eps(1+\eps)\tilde n\alpha p.
  \end{align*}
  Finally, using that $v$ is $(\eps,k)$-expanding in $G$ and the maximum degree
  bound on every $u \in V_{i-1}$, we have
  \[
    |N_{G'}^i(v)| \geq |N_G^i(v)| -
    (1-\eps)^{-1}(1+\eps)^i(1+2^{i-i})\eps|N_G^i(v)| \geq (1-2^i\eps)|N_G^i(v)|,
  \]
  for $\eps > 0$ sufficiently small. This completes the proof.
\end{proof}

\begin{lemma}\label{lem:robust-reg-expansion-medium-sets}
  For every $k \geq 1$ and all $\alpha, \gamma > 0$, there exists a positive
  constant $\eps$ with the following property. For every $c, \mu > 0$ there
  exists a $d > 0$ such that if $p \geq \log^2 n/n$, then w.h.p.\ $\Gamma \sim
  \Gnp$ satisfies the following. Let $G \subseteq \Gamma$, $\tilde n = \mu n$,
  and let $U, V_1, \dotsc, V_k \subseteq V(G)$ be disjoint sets such that:
  \begin{itemize}
    \item $|V_1| = \dotsb = |V_k| = \tilde n$,
    \item $\deg_G(v, V_{i+1}) \leq (1+\eps)\tilde n\alpha p$, for all $v \in
      V_i$, $i \in [k-1]$, and
    \item every $v \in U$ is $(\eps,k)$-expanding with respect to $V_1,\dotsc,
      V_k$.
  \end{itemize}
  Let $\ell \in [k]$ and suppose $Q \subseteq V(G) \setminus U$ is a subset of
  size $|Q| \leq c/(n^{\ell-1}p^{\ell})$. Then all but $d/(n^{\ell}p^{\ell+1})$
  vertices $v \in U$ are $(\gamma,k)$-expanding with respect to $V_1, \dotsc,
  V_k$ in $G - \nabla(Q)$.
\end{lemma}
\begin{proof}
  Given $k$, $\alpha$, and $\gamma$, let $\eps$ be sufficiently small for the
  argument below to go through, and additionally given $c, \mu > 0$ let $\nu >
  0$ be much smaller than $\eps(1-\eps)^k\mu^k\alpha^k$. For convenience, we
  write $G' := G - \nabla(Q)$ and for $F \in \{G, G'\}$ use $N_F^i(v)$ to mean
  $N_F^i(v, V_i)$, for all $i \in [k]$ and $v \in U$. Assume that $\Gamma \sim
  \Gnp$ is such that it satisfies the conclusion of
  Lemma~\ref{lem:gnp-k-expansion} and every $v \in V(\Gamma)$ satisfies
  $|N_\Gamma^i(v)| \leq (1+\nu)(np)^i$ for all $i \in [k]$, both of which happen
  with high probability.

  We show that there is a chain of sets $U = X_0 \supseteq X_1 \supseteq X_2
  \supseteq \dotsb \supseteq X_k$ such that for all $i \in [k]$:
  \stepcounter{propcnt}
  \begin{alphenum}
    \item\label{X-chain-size} $|X_i| \geq |U| - O(1/(n^\ell p^{\ell+1}))$, and
    \item\label{X-chain-exp} $|N_G^j(v) \cap Q| < \eps|N_G^j(v)|$ for every $v
      \in X_i$ and $j \in [i]$.
  \end{alphenum}
  This, for $i = k$, gives a set $X_k \subseteq U$ of size $|U| - d/(n^\ell
  p^{\ell+1})$ (for some large $d > 0$) in which all vertices satisfy
  \ref{X-chain-exp}. We then draw the conclusion we need as follows. For every
  $v \in X_k$ and all $j \in [k]$, we have
  \[
    |N_{G'}^j(v)| \geq (1-\eps)|N_G^j(v)| - \big|N_G^{j-1}(v) \setminus
    N_{G'}^{j-1}(v)\big|(1+\eps)\tilde n\alpha p.
  \]
  Telescoping this for any $j \in [k]$ gives
  \[
    |N_{G'}^j(v)| \geq (1-\eps)|N_G^j(v)| - \sum_{1 \leq z \leq j-1}
    \eps|N_G^z(v)| \big((1+\eps)\tilde n\alpha p\big)^{j-z}.
  \]
  Finally, as $v$ is $(\eps,k)$-expanding with respect to $V_1,\dotsc,V_k$, we
  have $|N_G^j(v)| \geq (1-\eps)(\tilde n\alpha p)^j$ and $|N_G^j(v)| \leq
  \big((1+\eps)\tilde n\alpha p\big)^j$ for all $j \in [k]$, and so we obtain
  \[
    |N_{G'}^j(v)| \geq (1-\eps)^2(\tilde n\alpha p)^j -
    (j-1)\eps\big((1+\eps)\tilde n\alpha p\big)^j \geq (1-\gamma)(\tilde n\alpha
    p)^j,
  \]
  as desired, by choosing $\eps > 0$ to be sufficiently small. It remains to
  show that there are sets fulfilling \ref{X-chain-size} and \ref{X-chain-exp}.
  We do this by induction on $i$.

  Consider some $i \in [k]$, a set $X_{i-1}$ which satisfies \ref{X-chain-size}
  and \ref{X-chain-exp} (for start, $X_0$ surely does), and assume first
  $|X_{i-1}| \leq \nu/(n^{i-1}p^i)$. Let $Z_i$ be a set of vertices $x \in
  X_{i-1}$ which violate \ref{X-chain-exp} for $j = i$, that is
  \[
    |N_G^i(x) \cap Q| \geq \eps|N_G^i(x)| \geq \eps(1-\eps)(\tilde n\alpha p)^i
    \geq \xi n^i p^i,
  \]
  for $\xi = \eps(1-\eps)\mu^k\alpha^k$. As $N_G^i(x) \cap Q \subseteq N_G^i(x)
  \subseteq N_\Gamma^i(x)$, we have
  \[
    |Q| \geq |Q \cap V_i| \geq \Big| \bigcup_{x \in Z_i} N_G^i(x) \cap Q \Big|
    \geq \Big| \bigcup_{x \in Z_i} N_\Gamma^i(x) \Big| - \Big| \bigcup_{x \in
    Z_i} N_\Gamma^i(x) \setminus (N_G^i(x) \cap Q) \Big|.
  \]
  Now, as $|X_{i-1}|n^{i-1}p^{i-1} \leq \nu/p$ by assumption, we can apply
  Lemma~\ref{lem:gnp-k-expansion} with $Z_i$ (as $X$) together with the fact
  that $|N_\Gamma^i(x)| \leq (1+\nu)n^ip^i$, to get
  \[
    |Q| \geq |Q \cap V_i| \geq |Z_i|(1-k\nu)n^ip^i - |Z_i|\big((1+\nu)n^ip^i
    - \xi n^ip^i \big) \geq |Z_i|(\xi/2)n^ip^i.
  \]
  Using the bound on the size of $Q$ in the statement of the lemma we conclude
  \[
    |Z_i| \leq \frac{c}{n^{\ell-1}p^{\ell}} \cdot \frac{2}{\xi n^ip^i} =
    O\Big(\frac{1}{n^\ell p^{\ell+1}}\Big).
  \]
  We set $X_i := X_{i-1} \setminus Z_i$, which, by induction hypothesis,
  satisfies \ref{X-chain-size}.

  On the other hand, if $\nu/(n^{i-1}p^i) < |X_{i-1}|$ then, by exactly the same
  argument as above, in every subset of $X_{i-1}$ of size precisely
  $\floor{\nu/(n^{i-1}p^i)}$, taking $Z_i$ to be its subset of vertices not
  satisfying \ref{X-chain-exp} for $j=i$, we get
  \[
    |Z_i| \leq \frac{c}{n^{\ell-1}p^{\ell}} \cdot \frac{2}{\xi n^ip^i} =
    o\Big(\frac{1}{n^{i-1}p^i}\Big),
  \]
  since $n^\ell p^\ell \gg 1$ as $\ell \geq 1$. Thus, with room to spare, all
  but at most $O(1/(n^\ell p^{\ell+1}))$ vertices in $X_{i-1}$ satisfy
  \ref{X-chain-exp}, and we proclaim these to be $X_i$, fulfilling
  \ref{X-chain-size}.
\end{proof}

\begin{lemma}\label{lem:robust-reg-expansion-any-size-sets}
  For every $k \in \N$ and all $c, \alpha, \gamma > 0$, there exist positive
  constants $\eps$ and $\delta$ with the following property. For every $\mu >
  0$, if $p \geq \log^2 n/n$, then w.h.p.\ $\Gamma \sim \Gnp$ satisfies the
  following. Let $G \subseteq \Gamma$, $\tilde n = \mu n$, and let $U, V_1,
  \dotsc, V_k \subseteq V(G)$ be disjoint sets such that:
  \begin{itemize}
    \item $|V_1| = \dotsb = |V_k| = \tilde n$,
    \item $\deg_G(v, V_{i+1}) \leq (1+\eps)\tilde n\alpha p$, for all $v \in
      V_i$, $i \in [k-1]$, and
    \item every $v \in U$ is $(\eps,k)$-expanding with respect to $V_1,\dotsc,
      V_k$.
  \end{itemize}
  Then for every $Q \subseteq V(G) \setminus U$ of size $|Q| \leq \min\{c|U|,
  \delta\tilde n\}$, there are at most $\gamma|U|$ vertices $v \in U$ which are
  not $(\gamma,k)$-expanding with respect to $V_1, \dotsc, V_k$ in $G -
  \nabla(Q)$.
\end{lemma}
\begin{proof}
  Observe that if $|U| > K/(\gamma p)$, for $K =
  K_{\ref{lem:robust-reg-expansion-large-sets}}(\alpha,\gamma,\mu)$, then the
  statement follows from Lemma~\ref{lem:robust-reg-expansion-large-sets} by
  choosing $\delta$ sufficiently small so that $c\delta <
  \delta_{\ref{lem:robust-reg-expansion-large-sets}}(\alpha,\gamma)$.
  Otherwise, if $|U| \leq K/(\gamma p)$, we aim to show that in every $X
  \subseteq U$ of size $\gamma|U|$ there is a vertex which is
  $(\gamma,k)$-expanding with respect to $V_1, \dotsc, V_k$ in $G - \nabla(Q)$.
  We show that there is a chain of sets $U = X_0 \supseteq X_1 \supseteq X_2
  \supseteq \dotsb \supseteq X_k$ such that for all $i \in [k]$:
  \begin{itemize}
    \item $|X_i| \geq |X| - O(|X|/\log n)$, and
    \item $|N_G^j(v) \cap Q| < \eps|N_G^j(v)|$ for every $v \in X_i$ and $j \in
      [i]$.
  \end{itemize}
  The rest of the proof proceeds (almost) identically as the proof of
  Lemma~\ref{lem:robust-reg-expansion-medium-sets}.
\end{proof}

\section{Proof of the blow-up lemma}\label{sec:blow-up}

In this section we give the proof of Theorem~\ref{thm:blow-up-lemma} which
roughly follows the outline given in Section~\ref{sec:absorbing-method}. That
being said, next lemma is the crux of the argument. For some disjoint sets $W_1,
\dotsc, W_t \subseteq V(G)$ and $\cW = (W_1,\dotsc,W_t)$ we say that a graph $A$
is a $\cW$-\emph{absorber} if for every $Z \subseteq \bigcup_{i\in[t]} W_i$, such
that $|Z \cap W_1| = \dotsb = |Z \cap W_t|$, there is a $C_t$-factor in $A - Z$.

\begin{lemma}[Absorbing Lemma]\label{lem:absorbing-lemma}
  Let $k \geq 2$ and $t \in \{2k-1, 2k\}$. For every $\alpha, \gamma > 0$, there
  exist positive constants $\eps$ and $\xi$ with the following property. For
  every $\mu > 0$ there is a $C > 0$ such that if $p \geq Cn^{-(k-1)/k}$, then
  w.h.p.\ $\Gamma \sim \Gnp$ satisfies the following. For every $G \subseteq
  \Gamma$ in $\cGEk(C_t, \tilde n, \eps, \alpha p)$, with $\tilde n \geq \mu
  n$, there are sets $W_i \subseteq V_i$, such that:
  \stepcounter{propcnt}
  \begin{alphenum}
    \item\label{abs-class-exp} The graph $G[W_1 \cup \dotsb \cup W_t]$ belongs
      to $\cGEk(C_t, \xi\tilde n, \gamma, \alpha p)$.
    \item\label{abs-expansion} For all $i \in [t]$ every $v \in V_i$ is
      $(\gamma, k-1)$-expanding with respect to $W_{i+1}, \dotsc, W_{i+(k-1)}$
      and $W_{i-1}, \dotsc, W_{i-(k-1)}$.
    \item\label{abs-existence} There is a $\cW$-absorber $A \subseteq G$, for
      $\cW = (W_1,\dotsc,W_t)$, such that $|V(A) \cap V_i| = |V(A) \cap V_j|
      \leq \gamma\tilde n$.
  \end{alphenum}
\end{lemma}

Before we begin, let us establish an important observation. Consider $G \in
\cGEk(C_t, \tilde n, \eps, \alpha p)$ and let $G' = G - \nabla(Q)$, for some $Q
\subseteq V(G)$, $|Q| \leq \tilde n/2$. Then, if $v \in V_1$ is
$(\gamma,k-1)$-expanding in $G'$ for $\eps \ll \gamma < 1/2$ with respect to
$V_{1+i}$ and $V_{t-i+1}$, for $i \in [k-1]$, then there is a canonical copy of
$C_t$ in $G'$ which contains $v$. Indeed, let $N_k := N_{G'}^{k-1}(v, V_k)$ and
$N_{t-k+2} := N_{G'}^{k-1}(v, V_{t-k+2})$. As $v$ is $(\gamma,k-1)$-expanding,
$N_k$ and $N_{t-k+2}$ incorporate a sufficiently large fraction of $N_G^{k-1}(v,
V_k)$ and $N_G^{k-1}(v, V_{t-k+2})$, so that, if $t = 2k$ then $(N_k, V_{k+1})$
and $(N_{t-k+2}, V_{k+1})$ are $(2\eps,p)$-lower-regular in $G'$, and similarly
if $t = 2k-1$ then $(N_k, N_{t-k+2})$ is $(2\eps,p)$-lower-regular in $G'$. In
the former there is then a vertex $u \in V_{k+1}\setminus Q$ which together with
$v$ closes a canonical copy of $C_t$, and in the latter there is an edge $uw \in
G'[N_k, N_{t-k+2}]$ which together with $v$ closes a canonical copy of $C_t$. We
use this several times in the proof and do not mention it explicitly.

\begin{proof}
  Given $k, t, \alpha$, and $\gamma$, let $h := v(\Fabs)$, $c = 4ht40^t$,
  and furthermore
  \[
    \lambda = \frac{1}{2h}, \enskip
    \rho' = \frac{1}{2t}, \enskip
    \rho = \min\{\eps_{\ref{lem:robust-reg-expansion-large-sets}}(\alpha,\rho'),
    \eps_{\ref{lem:robust-reg-expansion-medium-sets}}(\alpha,\rho')\},
    \enskip \text{and} \enskip
    \delta =
    \min\{\delta_{\ref{lem:robust-reg-expansion-large-sets}}(\alpha,\rho'),
    \delta_{\ref{lem:robust-reg-expansion-any-size-sets}}(c,\alpha,\rho')\}.
  \]
  Next, we let
  \[
    \eps' \leq
    \min\{\eps_{\ref{lem:robust-reg-expansion-large-sets}}(\alpha,\rho),
      \eps_{\ref{lem:robust-reg-expansion-medium-sets}}(\alpha,\rho),
      \eps_{\ref{lem:robust-reg-expansion-any-size-sets}}(\alpha,\rho), \eta \cdot
    \eps_{\ref{thm:klr}}(\Fconn,\alpha)\}, \enskip
    \eps \leq
    \min\{\eps_{\ref{lem:reg-exp-partitioning}}(\eps', \lambda),
    \eps_{\ref{lem:reg-exp-partitioning}}(\gamma, \xi)\},
  \]
  where
  \[
    \eta = \frac{\delta}{2t} \quad \text{and} \quad \xi =
    \frac{\lambda}{2c(3t)^k}\min\{\delta, \gamma\}.
  \]
  Additionally, given $\mu$, we take
  \[
    c_{k-1} = K_{\ref{lem:robust-reg-expansion-large-sets}}(\alpha, \rho',
    \lambda\mu)
    \quad \text{and} \quad
    c_i = d_{\ref{lem:robust-reg-expansion-medium-sets}}(2tc_{i+1}, \alpha,
    \rho', \lambda\mu) \text{ for every $i \in [k-2]$.}
  \]
  Lastly, let $C > 0$ be as large as necessary for the arguments below to go
  through; in particular so that all the lemmas can be applied with their
  respective parameters and $(1-\rho)(\lambda\tilde n\alpha p)^{k-1} \gg c_1/p$.

  Assume that $\Gamma \sim \Gnp$ is such that it satisfies:
  \stepcounter{propcnt}
  \begin{alphenum}
    \item\label{g-sparse-embedding} the conclusion of
      Theorem~\ref{thm:klr} applied with $\Fconn$ (as
      $H$) and $\eta\lambda\mu$ (as $\mu$);
    \item\label{g-large-exp} the conclusion of
      Lemma~\ref{lem:robust-reg-expansion-large-sets} applied with $\rho$ (as
      $\gamma$), $\eps'$ (as $\eps$), and $\lambda\mu$ (as $\mu$) as well as
      with $\rho'$ (as $\gamma$), $\rho$ (as $\eps$), and $\lambda\mu$ (as
      $\mu$);
    \item\label{g-medium-exp} the conclusion of
      Lemma~\ref{lem:robust-reg-expansion-medium-sets} applied with $\rho$ (as
      $\gamma$), $\eps'$ (as $\eps$), and $\lambda\mu$ (as $\mu$) as well as
      with $\rho'$ (as $\gamma$), $\rho$ (as $\eps$), and $\lambda\mu$ (as
      $\mu$), for every $2tc_i$ (as $c$), $2 \leq i \leq k-1$;
    \item\label{g-all-exp} the conclusion of
      Lemma~\ref{lem:robust-reg-expansion-any-size-sets} applied with $\rho$ (as
      $\gamma$), $\eps'$ (as $\eps$), and $\lambda\mu$ (as $\mu$);
  \end{alphenum}
  This all happens with high probability and from now on we condition on these
  events.

  Let $s := \lambda \tilde n$, and $\cGEk(\Fabs, s, \eps', \alpha p) \subseteq
  \cG(\Fabs, s, \eps', \alpha p)$ be a class of graphs in which every copy of
  $C_t$ in $\Fabs$ belongs to $\cGEk(C_t, s, \eps', \alpha p)$. We first
  partition the vertex set of $G$ for convenience of embedding an absorber. Let
  $R_1, \dotsc, R_t, U_1, \dotsc, U_{h-t}$, be a collection of disjoint subsets
  of $V_1, \dotsc, V_t$, each of size $s$, such that $R_i \subseteq V_i$ and the
  graph in $G$ induced by them belongs to the class $\cGEk(\Fabs, s, \eps',
  \alpha p)$, with $\{R_1,\dotsc,R_t\}$ as the set $R$ in an $R$-absorber
  $\Fabs$. Let $G'$ denote this graph throughout. Additionally, let $W_i, X_i
  \subseteq R_i$, $i \in [t]$, be disjoint sets with $|W_i| = |X_i| = \xi\tilde
  n$ and suppose $G[W_1 \cup \dotsb \cup W_t]$ belongs to $\cGEk(C_t, \xi\tilde
  n, \gamma, \alpha p)$ and every $v \in V_i$ is $(\gamma, k-1)$-expanding with
  respect to $W_{i+1}, \dotsc, W_{i+(k-1)}$ and $W_{i-1}, \dotsc, W_{i-(k-1)}$,
  where indices are taken so that $t+i = i$ and $1-i = t-i+1$. As both $C_t$ and
  $\Fabs$ are a subgraph of $\cG(C_t, h, 0, 1)$ by
  Proposition~\ref{prop:absorber-properties}~\ref{abs-blow-up}, all of the sets
  as discussed above can be shown to exist by several applications of
  Lemma~\ref{lem:reg-exp-partitioning}. Lastly, set $\cW = (W_1,\dotsc,W_t)$.

  The key part of the proof is to make use of the template graph given by
  Lemma~\ref{lem:abs-template} to construct copies of $\Fabs$ in $G'$. Let $B =
  (B_1,\dotsc,B_t; E_B)$ be the template graph given by
  Lemma~\ref{lem:abs-template} applied for $\xi\tilde n$ (as $m$) and let
  \[
    f \colon V(B) \to \bigcup_{i\in[t]} W_i \cup X_i
  \]
  be a bijection mapping vertices of $B_i$ to $W_i \cup X_i$ such that $W_i
  \subseteq f(B_i')$ for all $i \in [t]$. In the remainder of the proof, for
  every $t$-edge $e = \{b_1,\dotsc,b_t\} \in E_B$ with $b_i \in B_i$ we aim to
  find a copy of $\Fabs$ in $G'$ rooted at vertices $f(b_1),\dotsc,f(b_t)$, so
  that all of these copies are internally disjoint (that is, other than `roots'
  $f(b_1),\dotsc,f(b_t)$). For ease of further reference, we let $R_e$ denote
  this $t$-element set $f(b_1),\dotsc,f(b_t)$ which correspond to an edge $e \in
  E_B$. Let $A$ be the graph obtained as a union of those graphs $\Fabs$.
  In order to see the `absorbing property' of $A$, consider some $Z \subseteq
  \bigcup_{i \in [t]} W_i$ such that $|Z \cap W_1| = \dotsb = |Z \cap W_t|$ and
  its corresponding set $f^{-1}(Z) \subseteq \bigcup_{i \in [t]} B_i'$ in the
  template $B$. Then, by the defining property of $B$ (see
  Lemma~\ref{lem:abs-template}), the hypergraph $B - f^{-1}(Z)$ has a perfect
  matching. For every edge $e$ in this matching, in the $R_e$-absorber $\Fabs$
  take the $C_t$-factor which contains the set $R_e$, and for all other edges
  $e'$ take the $C_t$-factor which does not contain the set $R_{e'}$. This
  assembles the desired $C_t$-factor in $A - Z$.

  In order to construct this disjoint collection of graphs $\Fabs$, we turn to
  Haxell's hypergraph matching theorem (Theorem~\ref{thm:haxell-matching}). Let
  $\cH$ be an (auxiliary) $(h-t+1)$-uniform hypergraph with vertex set
  \[
    V(\cH) = \big\{R_e : e \in E_B\big\} \cup \big(V(G') \setminus \bigcup_{i\in
    [t]} (W_i \cup X_i)\big)
  \]
  as $A_{\cH} \cup B_{\cH}$, and an $(h-t+1)$-edge for every $e \in E_B$ and
  every $Y \subseteq B_{\cH}$ of size $h-t$, for which there is an
  $R_e$-absorber $\Fabs$ in $G'$ whose internal vertices belong completely to
  $Y$. An $A_\cH$-saturating matching in $\cH$ corresponds exactly to what we
  need, that is internally disjoint copies of $R_e$-absorbers $\Fabs$ in $G'$
  for every $e \in E_B$.

  What remains is to verify the condition in Theorem~\ref{thm:haxell-matching}
  holds. In particular, for every $E \subseteq E_B$ and $Q \subseteq V(G')
  \setminus \bigcup_{i\in[t]} (W_i\cup X_i)$ of size $|Q| \leq 2h|E|$, we need
  to find \emph{at least one} edge $e \in E$ so that there is an $R_e$-absorber
  $\Fabs$ in $G' - Q$. Fix sets $E \subseteq E_B$ and $Q$ as above, and let $E'
  \subseteq E$ be a set of pairwise disjoint edges, $|E'| \geq |E|/(2t40^t)$,
  which we can greedily find as $\Delta(B) \leq 40^t$. For $\ell \in [t]$, let
  $S_\ell \subseteq W_\ell \cup X_\ell$ be the vertices which appear in at least
  one edge of $E'$ and note that by construction $|S_1| = \dotsb = |S_t| =
  |E'|$. Recall the labelling of the vertices of the $C_t$-tree (see
  Figure~\ref{fig:C3-tree}) and, with a possible slight abuse of notation, for a
  vertex $v$ in some $S_\ell \subseteq R_\ell$, let
  $\{U_{i,j}^\ell\}_{i\in[k+1],j\in[(t-1)^i]}$ stand for sets in $G'$ which
  together with $R_\ell$ `induce' the $C_t$-tree of depth $k$ rooted at
  $R_\ell$. Then, by our choice of constants
  \begin{equation}
    \label{eq:Q-size-ub}
    |Q| \leq 2h|E| \leq 2t40^t \cdot 2h|E'| = c|S_\ell| \qquad \text{and} \qquad
    |Q| \leq c|S_\ell| \leq c \cdot 2\xi\tilde n \leq
    \frac{\delta}{(3t)^k} s
  \end{equation}
  The following claim is crucial.

  \begin{claim}\label{cl:main-claim-singletones}
    There is a set $S' \subseteq S_\ell$ of size $|S'| \geq
    (1-\frac{1}{t+1})|E'|$ such that for every $v \in S'$ and every choice of
    $(1-\eta)s$ vertices from each of $U_{k,1}^\ell, \dotsc,
    U_{k,(t-1)^k}^\ell$, there is a copy of a $C_t$-tree of depth $k-1$ in $G' -
    Q$ with each $u_{k,j}$, $j \in [(t-1)^k]$, mapped into (exactly) one of the
    chosen sets.
  \end{claim}
  \begin{proof}
    For simplicity of notation we drop the index $\ell$ and write just $S$,
    $U_{i,j}$. We refer to the sets $U_{i,1},\dotsc,U_{i,(t-1)^i}$ as the $i$-th
    level. Moreover, whenever we say that a vertex $v \in U_{i,j}$ (or $v \in
    S$) is $(\cdot,k-1)$-expanding, we mean it with respect to both groups of
    sets on the level below, namely
    $U_{i+1,(j-1)(t-1)+1},\dotsc,U_{i+1,(j-1)(t-1)+k-1}$ and
    $U_{i+1,(j-1)(t-1)-1},\dotsc,U_{i+1,(j-1)(t-1)-(k-1)}$. The proof is a
    tedious and technical cleaning procedure of the vertex sets representing a
    $C_t$-tree and relies on multiple applications of properties
    \ref{g-large-exp}--\ref{g-all-exp}. On a high level we proceed as follows.
    Consider $G'[U_{k-1,1} \cup U_{k,1} \cup \dotsb \cup U_{k,t-1}]$ and recall
    that it belongs to the class $\cGEk(C_t, s, \eps', \alpha p)$. By
    \ref{g-large-exp} and \ref{g-medium-exp}, all but $O(|Q|)$ vertices in
    $U_{k-1,1} \setminus Q$ are still $(\rho,k-1)$-expanding in $G' -
    \nabla(Q)$. Adding the non-expanding $O(|Q|)$ vertices to $Q$, and
    proceeding in a bottom-up fashion, we clean all the sets $U_{i,j}$ so that
    in the end there are at least $(1-\rho)|E'|$ expanding vertices remaining in
    $S$, which we group into $S'$. Furthermore, while doing this we also ensure
    that all the vertices remaining in each $U_{i,j}$ are
    $(\rho,k-1)$-expanding. The second part of the proof is almost
    analogous---we fix a vertex $v \in S'$, remove additionally an arbitrary set
    of $\eta s$ vertices from each $U_{k,j}$, $j \in [(t-1)^k]$, and proceed
    with cleaning in a bottom-up fashion using \ref{g-large-exp} and
    \ref{g-medium-exp}. It is time to roll up our sleeves and start to grind.

    We claim that, for all $i \in [k]$ and all $j \in [(t-1)^i]$, there is a set
    $U_{i,j}' \subseteq U_{i,j} \setminus Q$ such that
    \begin{itemize}
      \item $|U_{i,j}'| \geq s - (3t)^{k-i}|Q|$, and
      \item every $v \in U_{i,j}'$ is $(\rho,k-1)$-expanding.
    \end{itemize}
    This clearly holds for $i = k$ (ignoring the expanding part which is not
    needed), so consider some $i \in [k-1]$ and $G'[U_{i,1} \cup U_{i+1,1} \cup
    \dotsb \cup U_{i+1,t-1}]$ which belongs to $\cGEk(C_t, s, \eps', \alpha p)$.
    Let $Q' := \bigcup_{j \in [t-1]} (U_{i+1,j} \setminus U_{i+1,j}')$. By
    induction hypothesis and \eqref{eq:Q-size-ub}
    \[
      |Q'| \leq t \cdot (3t)^{k-(i+1)} |Q| \leq \delta s.
    \]
    If $c_{k-1}/p \leq |Q'|$ we use \ref{g-large-exp} and if
    $c_{k-1-z}/(n^zp^{z+1}) \leq |Q'| < c_{k-z}/(n^{z-1}p^z)$ for some $z \in
    [k-1]$, then we use \ref{g-medium-exp} with $c_{k-z}$ (as $c$) and $z$ (as
    $\ell$). In both cases, we get a set $U_{1,i}' \subseteq U_{1,i} \setminus
    Q$ of size $s - 3|Q'| \geq s - (3t)^{k-i}|Q|$, with the property that every
    $v \in U_{i,1}'$ is $(\rho, k-1)$-expanding. Since there is nothing special
    about $U_{i,1}$ nor $U_{i+1,1}, \dotsc, U_{i+1,t-1}$, we come to the same
    conclusion for every $U_{i,j}$, $j \in [(t-1)^i]$.

    Next, consider $S$ and recall that $S \subseteq R$ and $G'[R \cup U_{1,1}
    \cup \dotsb \cup U_{1,t-1}]$ belongs to the class $\cGEk(C_t, s, \eps',
    \alpha p)$. At this point, we use \ref{g-all-exp} with $S$ (as $U$) and
    $\bigcup_{j\in[t-1]} (U_{1,j}\setminus U_{1,j}')$ (as $Q$).
    Since by the prior cleaning procedure
    \[
      \Big|\bigcup_{j\in[t-1]} (U_{1,j} \setminus U_{1,j}')\Big| \leq t \cdot
      (3t)^{k-1}|Q| \osref{\eqref{eq:Q-size-ub}}\leq \min\{c|S|, \delta s\}
    \]
    we may indeed do so. In conclusion, there are at least $(1-\rho)|S| >
    (1-\frac{1}{t+1})|S|$ vertices $S' \subseteq S$ which are
    $(\rho,k-1)$-expanding, completing the first part of the proof.

    The second phase is slightly trickier but of very similar flavour. Fix some
    $v \in S'$ and let $G'' := G' - \nabla(Q')$ where $Q'$ is a union of all
    $U_{i,j} \setminus U_{i,j}'$, that is $Q$ and all the iteratively removed
    non-expanding vertices in the prior procedure. We use that the maximum
    degree of $G''[U_{i,j},U_{i,j+1}]$ is bounded by $(1+\rho)\tilde n\alpha p$
    throughout, which is required whenever using properties \ref{g-large-exp} or
    \ref{g-medium-exp}. Choose arbitrary $\eta s$ vertices in each $U_{k,j}'$,
    $j \in [(t-1)^k]$, and remove them to obtain sets $U_{k,j}''$. To establish
    the claim, it is sufficient to find a copy of a $C_t$-tree of depth $k-1$
    rooted at $v$ with each $u_{k,j} \in U_{k,j}''$.

    Consider sets $U_{k-1,1}', U_{k,1}', \dotsc, U_{k,t-1}'$ and let $Q'' :=
    \bigcup_{j \in [t-1]} (U_{k,j}' \setminus U_{k,j}'')$. Note that $|Q''| \leq
    t \cdot \eta s \leq \delta s$ by our choice of constants. We can hence use
    \ref{g-large-exp} to obtain a set $U_{k-1,1}'' \subseteq U_{k-1,1}'$ of size
    $|U_{k-1,1}'| - 2c_{k-1}/p$ with the property that every $u \in U_{k-1,1}''$
    is $(\rho',k-1)$-expanding in $G'' - \nabla(Q'')$. In particular, by our
    observation from the beginning of this section, every $u \in U_{k-1,1}''$
    belongs to a canonical copy of $C_t$ in $G''[U_{k-1,1}'' \cup U_{k,1}'' \cup
    \dotsb \cup U_{k,t-1}'']$.

    We show the following by induction on $i = k-1, \dotsc, 1$: for all $j \in
    [(t-1)^i]$, there is a set $U_{i,j}'' \subseteq U_{i,j}'$ of size
    $|U_{i,j}'| - 2c_i/(n^{k-i-1}p^{k-i})$ with the property that every $u \in
    U_{i,j}''$ belongs to a canonical copy of $C_t$ in
    \[
      G'[U_{i,j}'' \cup U_{i+1,(j-1)(t-1)+1}'' \cup \dotsb \cup
      U_{i+1,(j-1)(t-1)+t-1}''].
    \]
    Clearly, by what we just proved, this is true for $i = k-1$.

    Let now $1 \leq i < k-1$. Let $Q'' := \bigcup_{j \in [t-1]}(U_{i+1,j}'
    \setminus U_{i+1,j}'')$ and observe that by induction hypothesis
    \[
      |Q''| \leq \frac{2tc_{i+1}}{n^{k-1-(i+1)}p^{k-(i+1)}}.
    \]
    Hence, it follows by using \ref{g-medium-exp} (with $\ell = k-(i+1)$),
    that there is a set $U_{i,1}'' \subseteq U_{i,1}'$ of size
    \[
      |U_{i,1}''| \geq |U_{i,1}'| - \frac{2c_i}{n^{k-1-i}p^{k-i}}
    \]
    with the property that every $u \in U_{i,1}''$ is $(\rho',k-1)$-expanding in
    $G'' - \nabla(Q'')$, and thus belongs to a canonical copy of $C_t$ in
    $G''[U_{i,1}'' \cup U_{i+1,1}'' \cup \dotsb \cup U_{i+1,t-1}'']$. This can
    analogously be shown to hold for all $U_{i,j}'$, $j \in [(t-1)^i]$, and its
    corresponding sets on level $i+1$.

    Finally, consider $v$ and let $N_{1,j}' \subseteq U_{1,j}'$ and $N_{1,t-j}'
    \subseteq U_{1,t-j}'$, for $j \in [k-1]$, be the $j$-th neighbourhoods of
    $v$ in $U_{1,j}'$ and $U_{1,t-j}'$. Recall that $|N_{1,j}'|, |N_{1,t-j}'|
    \geq (1-\rho)(s\alpha p)^j$, for all $j \in [k-1]$. Let $N_{1,j}'' :=
    N_{1,j}' \cap U_{1,j}''$. From what we previously showed, we can conclude
    that
    \begin{equation}\label{eq:size-nij}
      |N_{1,j}''| \geq |N_{1,j}'| - \frac{2c_1}{n^{k-2}p^{k-1}} \geq
      (1-o(1))|N_{1,j}'| \quad \text{and} \quad |N_{1,t-j}''| \geq
      (1-o(1))|N_{1,t-j}'|
    \end{equation}
    for all $j \in [k-1]$, as $p \geq Cn^{-(k-1)/k}$ and by choosing $C$
    sufficiently large. What remains is to show that $v$ belongs to a canonical
    copy of $C_t$ in $\tilde G := G''\big[\{v\} \cup \bigcup_{j \in [t-1]}
    N_{1,j}''\big]$, where we ignore the edges with both endpoints in
    $N_{1,j}''$. It is sufficient to prove that $|N_{\tilde G}^{k-1}(v)| \geq
    (1/2)|N_{G'}^{k-1}(v, U_{1,k-1})|$, once again by the observation from the
    beginning of this section.

    Recall that every $u \in U_{1,j}$ satisfies $\deg_G(u, U_{1,j+1}) \leq
    (1+\eps')s\alpha p$ for all $j \in [k-1]$, and so we have
    \[
      |N_{\tilde G}^i(v)| \geq |N_{1,i}''| - |N_{1,i-1}' \setminus
      N_{\tilde G}^{i-1}(v)|(1+\eps')s\alpha p,
    \]
    for all $1 < i \leq k-1$. Telescoping for $i = k-1$, and using the fact that
    $|N_{1,i}' \setminus N_{1,i}''| = o(|N_{1,i}'|)$ from \eqref{eq:size-nij},
    gives
    \[
      |N_{\tilde G}^{k-1}(v)| \geq |N_{1,k-1}''| - \sum_{1 \leq j \leq k-2}
      o\big((s\alpha p)^j\big) \cdot \big((1+\eps')s\alpha p\big)^{k-1-j}.
    \]
    Since $|N_{1,k-1}''| \geq (3/4)(s\alpha p)^{k-1}$ by \eqref{eq:size-nij} we
    obtain
    \[
      |N_{\tilde G}^{k-1}(v)| \geq (3/4)(s\alpha p)^{k-1} - o\big((s\alpha
      p)^{k-1}\big) \geq (1/2)|N_{G'}^{k-1}(v, U_{1,k-1})|,
    \]
    where the last inequality follows from the fact that $|N_{G'}^{k-1}(v)| \leq
    \big((1+\eps')(s\alpha p)\big)^{k-1}$ and our choice of constants.
  \end{proof}

  The claim that $|S_i'| \geq (1-\frac{1}{t+1})|E'|$ for each $i \in [t]$
  implies by pigeonhole principle that there exists $e = \{b_1,\dotsc,b_t\} \in
  E'$ for which $f(b_i) \in S_i'$. Fix the corresponding $R_e$ for the rest of
  the proof. For $v \in R_e$ let $\cL(v)$ be the family of $\eta s$ disjoint
  tuples $\bfm{v} = (v_1,\dotsc,v_{(t-1)^k})$ for which there is a $C_t$-tree in
  $G' - Q$ with $v$ as the root and vertices of the $k$-th level (see
  Figure~\ref{fig:C3-tree}) bijectively mapped into $v_1,\dotsc,v_{(t-1)^k}$.

  Recall, the sets $R_1,\dotsc,R_t,U_1,\dotsc,U_{h-t}$ `induce' a copy of
  $\Fabs$ in $G'$, and let $r_1,\dotsc,r_t,u_1,\dotsc,u_{h-t}$ be the
  corresponding vertices of $\Fabs$. Let $\tilde G$ be a graph obtained from
  $G'$ by the following `contraction' process (we remark that this idea is
  inspired by a procedure from \cite{ferber2020almost} which was further refined
  in \cite{ferber2020dirac}). Start with $G'[\bigcup U_x']$, where $U_x'
  \subseteq U_x \setminus Q$, $|U_x'| = \eta s$, for which $u_x \in
  V(\Fabs^{-})$, that is $U_x$ does not correspond to any of the vertices of a
  $C_t$-tree of depth $k-1$ rooted at any $r_i$ in $\Fabs$ (see
  Section~\ref{sec:absorbing-method}). Additionally, for every $v \in R_e$ and
  $\bfm{v} \in \cL(v)$ add a new vertex $\bfm{v}$ to $\tilde G$. Denote the set
  of $\bfm{v}$ originating from the same $v \in R_e$ by $\bfm{V}_v$, and note
  that this adds a total of $t \cdot \eta s$ new vertices. Lastly, for every $y
  \in U_x'$, $u_x \in V(\Fabs^{-})$, add an edge $\bfm{v}y$ to $\tilde G$ if and
  only if $yz$ is an edge of $G'$ for some $z \in \bfm{v}$.

  This finally enables us to complete the proof. As all $U_x'$ and $\bfm{V}_v$
  as above are of size exactly $\eta s$, and all edges between corresponding
  sets are transferred from $G'$ to $\tilde G$, Lemma~\ref{lem:slicing-lemma}
  implies that the graph $\tilde G$ belongs to $\cG(\Fconn, \eta s, \eps'/\eta,
  \alpha p)$. Since by
  Proposition~\ref{prop:absorber-properties}~\ref{abs-m2-density}, we have
  $m_2(\Fconn) \leq k/(k-1)$, from \ref{g-sparse-embedding} we conclude that
  there is a canonical copy of $\Fconn$ in $\tilde G$. Lastly, as every $\bfm{v}
  \in \bfm{V}_v$ corresponds to a $C_t$-tree rooted at $v \in R_e$ in $G'$, and
  the remaining edges exist in $G'$ already, we can reverse the contraction
  operation at each $\bfm{v}$ and deduce that such a copy of $\Fconn$ completes
  a copy of an $R_e$-absorber $\Fabs$ in $G'- Q$ as desired.

  Note that $|W_i \cup X_i| = 2\xi\tilde n$, every $v \in W_i \cup X_i$ belongs
  to at most $40^t$ distinct $R$-absorbers by the maximum degree bound on the
  template graph $B$ (see Lemma~\ref{lem:abs-template}), and each $R$-absorber
  is of size $h$. If the collection of these graphs does not intersect each
  $V_i$ in exactly the same number of vertices, we can just repeat the whole
  construction in a cyclic way for all $i$ and thus we get
  \[
    |V_i \cap V(A)| = |V_j \cap V(A)| \leq t^2 \cdot 2\xi\tilde n \cdot 40^t
    \cdot h \leq \gamma \tilde n,
  \]
  as promised.
\end{proof}

The proof of Theorem~\ref{thm:blow-up-lemma} is now but a formality.

\begin{proof}[Proof of Theorem~\ref{thm:blow-up-lemma}]
  Given $k, t$, and $\alpha$, let $c = 2t^2$, $\gamma =
  \eps_{\ref{lem:robust-reg-expansion-any-size-sets}}(c,\alpha,1/2)$, $\xi =
  \xi_{\ref{lem:absorbing-lemma}}(\alpha,\gamma)$, $\delta =
  (\xi/2)\delta_{\ref{lem:robust-reg-expansion-any-size-sets}}(c,\alpha,1/2)$,
  $\rho = \delta\xi/t$, and $\eps \leq
  (\rho/2)\eps_{\ref{thm:klr}}(C_t,\alpha)$. Let $C$ be
  sufficiently large, in particular $C \geq
  \max\{C_{\ref{thm:klr}}(C_t,\alpha,\rho\mu/2),
  C_{\ref{lem:robust-reg-expansion-any-size-sets}}(c,\alpha,1/2,\xi\mu),
  C_{\ref{lem:absorbing-lemma}}(\alpha, \gamma, \mu)\}$. Assume $\Gamma \sim
  \Gnp$ is such that it satisfies the conclusion of
  Theorem~\ref{thm:klr} applied with $C_t$ (as $H$) and
  $\rho\mu/2$ (as $\mu$), Lemma~\ref{lem:robust-reg-expansion-any-size-sets}
  applied with $1/2$ (as $\gamma$) and $\xi\mu$ (as $\mu$), and
  Lemma~\ref{lem:absorbing-lemma}. This happens with high probability and from
  now on we condition on these three events.

  Let $A$ be the $\cW$-absorber given by an application of
  Lemma~\ref{lem:absorbing-lemma} with $\cW = (W_1,\dotsc,W_t)$ and $W_i
  \subseteq V_i$, each $W_i$ of size precisely $\xi\tilde n$. Let $U_i := V_i
  \setminus V(A)$, and so $s := |U_i| \geq (1-\gamma)\tilde n$ by
  \ref{abs-existence}. Lastly, by \ref{abs-expansion}, each $u \in U_i$ is
  $(\gamma,k-1)$-expanding with respect to both $W_{i+1}, \dotsc, W_{i+(k-1)}$
  and $W_{i-1}, \dotsc, W_{i-(k-1)}$, where indices are taken so that $t+i = i$
  and $1-i = t-i+1$.

  By Lemma~\ref{lem:slicing-lemma}, sets $U_i$ induce in $G$ a graph which
  belongs to $\cG(C_t, s, 2\eps, \alpha p)$. Therefore, we can repeatedly apply
  Theorem~\ref{thm:klr} to find a family of disjoint canonical copies of $C_t$
  in $G[\bigcup_{i \in [t]} U_i]$, covering all but precisely $\rho\tilde n$
  vertices in each $U_i$. Denote these leftover vertices by $Z_i$.

  Next, we make use of Haxell's matching theorem to match vertices of each $Z_i$
  with some vertices in $W_1 \cup \dotsb \cup W_t$ into copies of $C_t$.
  Consider an auxiliary $t$-uniform hypergraph $\cH$ with vertex set $\bigcup_{i
  \in [t]} Z_i \cup \bigcup_{i \in [t]} W_i$ and add a $t$-edge to $\cH$ for
  every $v \in Z_i$ and $Y \subseteq \bigcup_{j \in [t] \setminus \{i\}} W_j$
  with $|Y \cap W_j| = 1$, for which there is a copy of $C_t$ in $G$ induced by
  $\{v\} \cup Y$. Now, if for every $Z \subseteq \bigcup_{i\in[t]}Z_i$ and every
  $Q \subseteq \bigcup_{i\in[t]} W_i$, $|Q| \leq 2t|Z|$, there is a canonical
  copy of $C_t$ with one vertex in $Z \cap U_i$, for some $i \in [t]$, and the
  remaining $t-1$ vertices in $\bigcup_{j\in[t]\setminus\{i\}} W_j \setminus Q$,
  then there is a $Z$-saturating matching in $\cH$ by
  Theorem~\ref{thm:haxell-matching}. This immediately gives a family of $t \cdot
  \rho\tilde n$ disjoint canonical copies of $C_t$ in $G[\bigcup_{i\in[t]} Z_i
  \cup W_i]$ which in particular cover all the vertices of $Z_i$'s and
  \emph{exactly} $(t-1)\rho\tilde n < \xi\tilde n$ vertices in each $W_i$. At
  this point it is not too difficult to see that this is indeed the case. Fix
  sets $Z$ and $Q$ as above. Assume without loss of generality $Z \cap V_1$ is
  largest among $Z \cap V_i$, $i \in [t]$. Recall that, every $v \in V_1$ is
  $(\gamma,k-1)$-expanding with respect to $W_2, \dotsc, W_k$ by
  \ref{abs-expansion} and $\deg_G(u, W_{i+1}) \leq (1+\gamma)\xi\tilde n\alpha
  p$, for all $u \in W_i$, $i \in [k-1]$, by \ref{abs-class-exp}. Moreover,
  $|Z_1| \leq \rho\tilde n \leq \delta\xi\tilde n$ and $|Q| \leq 2t|Z| \leq
  2t^2|Z_1| = c|Z_1|$. Hence, we can apply
  Lemma~\ref{lem:robust-reg-expansion-any-size-sets} with $\xi\mu$ (as $\mu$),
  $Z \cap V_1$ (as $U$) and $W_2,\dotsc,W_k$ (as $V_1,\dotsc,V_{k-1}$) to obtain
  a vertex $v \in Z \cap V_1$ which is $(1/2,k-1)$-expanding with respect to
  both $W_2 \setminus Q, \dotsc, W_k \setminus Q$ and $W_t \setminus Q, \dotsc,
  W_{t-(k-1)} \setminus Q$. In particular, $v$ belongs to a cycle $C_t$ which
  does not intersect $Q$.

  Denote by $Q_i$ the used vertices in each $W_i$, that is the ones belonging to
  all the previously found cycles $C_t$ used to cover $Z_i$'s. Finally, by
  definition of a $\cW$-absorber $A$ and as the previously found cycles
  intersect each $W_i$ in exactly $(t-1)\rho\tilde n$ vertices, there is a
  family of disjoint copies of $C_t$ covering all the vertices of $G -
  \bigcup_{i \in [t]} Q_i$, completing the proof.
\end{proof}

\section{Resilience of cycle factors in random graphs}\label{sec:resilience}

To give a proof of Theorem~\ref{thm:main-theorem-res} we need some standard
concepts first. For an $n$-vertex graph $G$, a partition of $V(G)$ into sets
$(V_i)_{i=0}^{\ell}$ is said to be $(\eps, p)$-regular if $|V_0| \leq \eps n$,
$|V_1| = \dotsb = |V_\ell|$, and at most $\eps \ell^2$ pairs $(V_i,V_j)$ are not
$(\eps, p)$-regular. An $(\eps, \alpha, p)$-reduced graph $R$ of a partition
$(V_i)_{i=0}^{\ell}$ is a graph on vertex set $[\ell]$ where $ij \in E(R)$ if
and only if $(V_i,V_j)$ is $(\eps, p)$-regular (in $G$) with density $d(V_i,V_j)
\geq \alpha p$. We make use of the `minimum degree variant' of the sparse
regularity lemma for random graphs (see, e.g.~\cite{noever2017local}).

\begin{theorem}
  \label{thm:sparse-reg-lemma}
  For every $d, \eps > 0$ and $\ell_0 \in \N$, there exists an $L > 0$ such that
  for every $\alpha \in (0,1)$, if $p \gg 1/n$, then w.h.p.\ $\Gamma \sim \Gnp$
  satisfies the following. Every spanning subgraph $G \subseteq \Gamma$ with
  minimum degree $\delta(G) \geq d np$ admits an $(\eps,p)$-regular partition
  $(V_i)_{i=0}^{\ell}$ with $\ell_0 \leq \ell \leq L$ whose
  $(\eps,\alpha,p)$-reduced graph $R$ is of minimum degree $\delta(R) \geq
  (d-\alpha-\eps)|R|$.
\end{theorem}

\subsection{Expansion within sparse regular pairs}\label{sec:regularity}

In an attempt to keep notation more concise, we first introduce a definition.
For a graph $G \in \cG(C_t, n, \eps, p)$, we say that a vertex $v \in V_i$ is
$(\eps,k)$-\emph{typical} if:
\begin{itemize}
  \item $v$ is $(\eps,k-1)$-expanding with respect to both
    $V_{i+1},\dotsc,V_{i+{k-1}}$ and $V_{i-1},\dotsc,V_{i-(k-1)}$,
  \item if $t = 2k - 1$, its $(k-1)$-st neighbourhoods into $V_{i+(k-1)}$ and
    $V_{i-(k-1)}$ form an $(\eps,p)$-lower-regular pair;
  \item if $t = 2k$, its $(k-1)$-st neighbourhoods into $V_{i+(k-1)}$ and
    $V_{i-(k-i)}$ form an $(\eps,p)$-lower-regular pair each with $V_{i+k} =
    V_{i-k}$.
\end{itemize}
Recall, in the definition of $\cGEk(C_t, n, \eps, p)$ this is exactly what
\emph{every vertex} satisfies, namely every $v$ is $(\eps,k)$-typical. As it
turns out, an overwhelming majority of graphs in $\cG(C_t, n, \eps, p)$ are such
that, for a suitable choice of constants, all but $\gamma n$ vertices in each
$V_i$ are $(\gamma,k)$-typical to begin with.

In order to capture this formally, we unfortunately need another definition. For
$m \in \N$, the class $\cG(C_t, n, m, \eps, p)$ consists of all graphs on vertex
set $V_1 \cup \dotsb \cup V_t$, each $V_i$ of size $n$, and where every $G[V_i,
V_{i+1}]$ is $(\eps, p)$-regular with exactly $m$ edges. The following statement
is a modification of \cite[Lemma~5.9]{gerke2007small}; as such, the proof can be
read off from the proof of \cite[Lemma~5.9]{gerke2007small}, but we nevertheless
spell out (most of) the details in Appendix~\ref{sec:appendix-expanders}.

\begin{proposition}\label{prop:almost-all-expanders}
  Let $k \geq 2$ and $t \in \{2k-1,2k\}$. For every $\beta, \gamma > 0$ there
  exist positive constants $\eps_0$ and $C$, such that for all $0 < \eps \leq
  \eps_0$ and $Cn^{-(k-1)/k} \leq p \ll n^{-(k-2)/(k-1)}$, the number of graphs
  in $\cG(C_t, n, m, \eps, p)$, with more than $\gamma|V_1|$ vertices in $V_1$
  which are not $(\gamma,k)$-typical, is at most
  \[
    \beta^m \binom{n^2}{m}^t,
  \]
  for all $m \geq n^2p$.
\end{proposition}

We point out that, even though the upper bound on $p$ seems artificial, the
reason we introduced it is to at all times have $(np)^{k-2} \ll 1/p$; we are
confident this can be avoided but would introduce additional technicalities both
in the definitions and the proofs. As for our application it makes no
difference, we opted for a simpler proof, but slightly less pleasing to the eye
statement.

It is a straightforward first moment calculation then to show that w.h.p.\ none
of the `bad graphs' above appear as a subgraph of the random graph $\Gnp$.

\begin{proposition}\label{prop:gnp-reg-expanders}
  Let $k \geq 2$ and $t \in \{2k-1,2k\}$. For every $\alpha, \gamma > 0$ there
  exists a positive constant $\eps$ with the following property. For every $\mu
  > 0$ there exists a $C > 0$ such that if $Cn^{-(k-1)/k} \leq p \ll
  n^{-(k-2)/(k-1)}$, then w.h.p.\ $\Gamma \sim \Gnp$ satisfies the following.
  Let $G \subseteq \Gamma$ belong to $\cG(C_t, \tilde n, \eps, \alpha p)$, with
  $\tilde n \geq \mu n$. Then there are most $\gamma\tilde n$ vertices $v \in
  V_1$ which are not $(\gamma,k)$-typical. \qed
\end{proposition}

At this point, we can utilise the lemmas about robustness of expansion from
Section~\ref{sec:random-graphs} to show that in $\Gnp$ one can easily convert a
graph $G \in \cG(C_t, \tilde n, \eps, \alpha p)$ into a member of $\cGEk(C_t, s,
\gamma, \alpha p)$, for a suitable choice of constants. Moreover, this is done
without `losing' too many vertices, that is $s = (1-o(1))\tilde n$. Basically,
this strengthens the point of view that restricting ourselves only to the class
$\cGEk(C_t, s, \gamma, \alpha p)$ for the blow-up lemma in $\Gnp$ is not such a
huge deal---the two classes are practically the same up to a minor difference in
size of the sets within.

\begin{lemma}\label{lem:gnp-all-reg-expander-properties}
  For every $k \geq 2$, $t \in \{2k-1, 2k\}$, and all $\alpha, \gamma > 0$ there
  exists a positive constant $\eps$ with the following property. For every $\mu
  > 0$ there exists a $C > 0$ such that if $Cn^{-(k-1)/k} \leq p \ll
  n^{-(k-2)/(k-1)}$, then w.h.p.\ $\Gamma \sim \Gnp$ satisfies the following.
  Every $G \subseteq \Gamma$ which belongs to $\cG(C_t, \tilde n, \eps, \alpha
  p)$, with $\tilde n \geq \mu n$, contains a subgraph $G' \subseteq G$ which
  belongs to $\cGEk(C_t, s, \gamma, \alpha p)$, for some $s \geq
  (1-\gamma)\tilde n$.
\end{lemma}
\begin{proof}
  Given $\alpha, \gamma$, let $\delta \leq (1/4)\min\{\gamma,
  \delta_{\ref{lem:robust-reg-expansion-large-sets}}(\alpha,\gamma),
  \eps_{\ref{lem:robust-reg-expansion-large-sets}}(\alpha,\gamma)\}$, and let
  $\eps > 0$ be sufficiently small and $C > 0$ sufficiently large for the
  arguments below to go through. We identify indices $t+i$ with $i$ and $1-i$
  with $t-i+1$. Assume that $\Gamma \sim \Gnp$ is such that it satisfies the
  conclusion of Proposition~\ref{prop:gnp-reg-expanders} applied with $\delta$
  (as $\gamma$) and Lemma~\ref{lem:robust-reg-expansion-large-sets} which
  happens with high probability.

  Observe that, by definition of regular pairs, in every $V_i$ there are at most
  $2\eps\tilde n$ vertices which have more than $(1+\eps)\tilde n\alpha p$
  neighbours in either $V_{i-1}$ or $V_{i+1}$. By removing all these vertices,
  we get sets $V_i'$ of size at least $(1-2\eps)\tilde n$, each $v \in V_i'$
  having $\deg_G(v, V_{i+1}') \leq (1+\eps)\tilde n\alpha p \leq
  (1+4\eps)|V_{i+1}'|\alpha p$, and by Lemma~\ref{lem:slicing-lemma}, every
  $(V_i', V_{i+1}')$ is $(4\eps,\alpha p)$-regular (the same holds for
  $V_{i-1}'$). So, for simplicity, we may assume that all $v \in V_i$ are of
  bounded degree to neighbouring sets to begin with.

  We first apply Proposition~\ref{prop:gnp-reg-expanders} with $\delta$ (as
  $\gamma$) to get that for every $i \in [t]$ there is a set $Q_i \subseteq V_i$
  of at most $\delta\tilde n$ vertices which are not $(\delta,k)$-typical in
  $G$. We repeat the following process for all $i \in [t]$: if there is a vertex
  $v \in V_i \setminus Q_i$ which is not $(\gamma,k-1)$-expanding with respect
  to $V_{i+1} \setminus Q_{i+1}, \dotsc, V_{i+(k-1)} \setminus Q_{i+(k-1)}$ or
  $V_{i-1} \setminus Q_{i-1}, \dotsc, V_{i-(k-1)} \setminus Q_{i-(k-1)}$, add it
  to $Q_i$. Suppose towards contradiction there is a point at which some $|Q_i|
  \geq 2\delta\tilde n$. In particular, this means there are at least
  $(\delta/2)\tilde n$ vertices in $Q_i$ which are $(\delta,k-1)$-expanding with
  respect to, say, $V_{i+1},\dotsc,V_{i+(k-1)}$, but not
  $(\gamma,k-1)$-expanding with respect to $V_{i+1}\setminus Q_{i+1}, \dotsc,
  V_{i+(k-1)} \setminus Q_{i+(k-1)}$. As all $Q_{i+1},\dotsc,Q_{i+(k-1)}$ are of
  size at most $2\delta\tilde n$ and $2\delta \leq
  \delta_{\ref{lem:robust-reg-expansion-large-sets}}(\alpha,\gamma)$, this is a
  contradiction with the conclusion of
  Lemma~\ref{lem:robust-reg-expansion-large-sets}. To establish that these
  vertices are also typical, that is their $(k-1)$-st neighbourhoods are
  $(\gamma,\alpha p)$-lower-regular with necessary sets (see above), we just
  appeal to Lemma~\ref{lem:slicing-lemma}.

  Let $s = (1-2\delta)\tilde n$ and assume (by removing additional vertices if
  needed or taking random subsets) that all $|V_i \setminus Q_i| = s$. Thus, for
  every $v \in V_i \setminus Q_i$,
  \[
    \deg_G(v, V_{i+1} \setminus Q_{i+1}) \leq \deg_G(v, V_{i+1}) \leq
    (1+\eps)\tilde n\alpha p = \frac{1+\eps}{1-2\delta} s\alpha p \leq
    (1+\gamma)s\alpha p,
  \]
  and so $G[V_1 \setminus Q_1 \cup \dotsb \cup V_t \setminus Q_t]$ belongs
  to $\cGEk(C_t, s, \gamma, \alpha p)$, as desired.
\end{proof}

\subsection{Proof of Theorem~\ref{thm:main-theorem-res}}

From here on the proof follows a usual structure for a strategy based on the
regularity method. Think of $t$ being even. After applying the sparse regularity
lemma (Theorem~\ref{thm:sparse-reg-lemma}) to a subgraph $G \subseteq \Gamma$
with $\delta(G) \geq (1/2+\alpha)np$, the minimum degree in the
$(\eps,\alpha,p)$-reduced graph $R$ of the obtained $(\eps,p)$-regular partition
is sufficiently large for it to contain a Hamilton cycle on vertices
$1,\dotsc,2\ell$. We first clean-up all the sets $V_i$, moving some vertices to
$V_0$ along the way. The goal here is to, for every $i \in [\ell]$, find as
large sets $V_i^1, V_i^3, \dotsc, V_i^{t-1} \subseteq V_{2i-1}$ and $V_i^2,
V_i^4, \dotsc, V_i^t \subseteq V_{2i}$, such that $G[V_i^1 \cup V_i^2 \cup
\dotsb \cup V_i^{t-1} \cup V_i^t]$ belongs to $\cGEk(C_t, \tilde n, \gamma,
\alpha p)$, with $\tilde n = \Omega(n)$. Then we handle the `garbage' $V_0$ by
finding a collection of disjoint $t$-cycles covering all of its vertices. As we
unfortunately have no control over these, we need to avoid using up all the
vertices from some set $V_i$ while doing the former. This is easily accomplished
by taking an appropriately sized random subset of $V(G) \setminus V_0$ and using
it to find this collection. Another problem that arises after covering all the
vertices of $V_0$, is that the remaining sets $\tilde V_i^j \subseteq V_i^j$, $j
\in [t]$, can be of different sizes, making the blow-up lemma
(Theorem~\ref{thm:blow-up-lemma}) inapplicable for them. This is dealt with by
several usages of the (resolution of) K{\L}R-conjecture (Theorem~\ref{thm:klr})
and is strongly inspired by a similar procedure from \cite{balogh2012corradi}.
Lastly, all this has to be done so that the initially established
$(\gamma,k)$-expansion property is not damaged too heavily in the process, so,
everything is happening within randomly selected subsets before in the end
applying the blow-up lemma to whatever remains and covering the majority of
$V(G)$ with $t$-cycles provided by it.

\begin{proof}[Proof of Theorem~\ref{thm:main-theorem-res}]
  For a cleaner exposition, we focus only on the case when $t$ is even; the case
  when $t$ is odd is very similar and at the end of the proof we point out the
  main differences. For given $k$ and $\alpha$ let $\gamma =
  \eps_{\ref{thm:blow-up-lemma}}(\alpha)$, and choose $\gamma''$, $\delta_w$,
  and $\delta_x$ to be sufficiently small for the arguments below to go through;
  in particular, $(1-\gamma'')(1-\delta_w-\delta_x)^k \geq 1-\gamma$,
  $(1+\gamma'')/(1-\delta_w-\delta_x) \leq 1+\gamma$, and $\delta_w \leq
  \alpha\delta_x/(32t^2)$. Next, for $\ell_0 \in \N$ large enough, let
  \begin{gather*}
    \eps' \leq \delta_x/(4t), \quad \delta' =
    \delta_{\ref{lem:gnp-cycles-from-min-deg}}(\alpha/2, \delta_w/2), \quad
    \gamma' \leq \min\{\delta'\delta_w/4,
    \eps_{\ref{lem:reg-exp-partitioning}}(\alpha,\gamma'',\delta_w),
    \alpha\delta_w/20\},
    \\
    \eps \leq \frac{\eps'}{2t^2}\min\{\alpha/4, \gamma/4,
    \eps_{\ref{thm:klr}}(C_t, \alpha),
    \eps_{\ref{lem:gnp-all-reg-expander-properties}}(\alpha,\gamma')\}, \quad L
    = L_{\ref{thm:sparse-reg-lemma}}(1/2+2\alpha, \eps, \ell_0), \quad
    \text{and} \quad \mu = \frac{1-\eps}{Lt^2}.
  \end{gather*}
  Finally, let $C^\star =
  C_{\ref{lem:gnp-all-reg-expander-properties}}(\alpha,\gamma,\mu)$ and choose
  $C > 0$ sufficiently large, in particular such that $C \geq \max\{2C^\star,
  C_{\ref{thm:blow-up-lemma}}(\alpha,\gamma,\eps'\mu),
  C_{\ref{thm:klr}}(C_t,\alpha,\eps'\mu),
  C_{\ref{lem:gnp-cycles-from-min-deg}}(\alpha/2,\delta_w/2)\}$.

  Assume that $\Gamma \sim \Gnp$ is such that $\delta(\Gamma) \geq
  (1-\alpha)np$, and it satisfies the conclusion of
  Theorem~\ref{thm:blow-up-lemma} applied with $\gamma$ (as $\eps)$ and
  $\eps'\mu$ (as $\mu$), Theorem~\ref{thm:klr} applied with $C_t$ (as $H$),
  Lemma~\ref{lem:gnp-cycles-from-min-deg} applied with $\alpha/2$ (as $\alpha)$
  and $\delta_w/2$ (as $\mu$), Theorem~\ref{thm:sparse-reg-lemma} applied with
  $1/2+2\alpha$ (as $d$), and Lemma~\ref{lem:gnp-all-reg-expander-properties}
  applied with $\gamma'$ (as $\gamma$). This happens with high probability.

  As $\delta(\Gamma) \geq (1-\alpha)np$, we have $\delta(G) \geq
  (1/2+2\alpha)np$ (we are cheating here for simplicity of notation a bit and
  assuming $(1/\chi(C_t)-3\alpha)$-resilience). Let $(V_i)_{i = 0}^{2\ell}$ be
  an $(\eps,p)$-regular partition obtained after applying the sparse regularity
  lemma (Theorem~\ref{thm:sparse-reg-lemma}) with $1/2+2\alpha$ (as $d$) to $G$,
  and let $R$ be its $(\eps, \alpha, p)$-reduced graph. As $\delta(R) \geq
  (1/2+\alpha-\eps)|R| \geq (1/2+\alpha/2)|R|$, there is a Hamilton cycle in
  $R$, which is without loss of generality given by vertices $1,\dotsc,2\ell$
  and let $e_i = \{2i-1,2i\}$, for $i \in [\ell]$. Let $\tilde n :=
  2(1-\gamma')|V_i|/t$. An important thing to keep in mind is that for any edge
  $ij \in E(R)$ and any choice of pairwise disjoint sets $S_1,\dotsc,S_{t/2}
  \subseteq V_i$ and $T_1,\dotsc,T_{t/2} \subseteq V_j$, with $|S_i|, |T_i| \geq
  \eps'\tilde n$, as these sets inherit regularity by
  Lemma~\ref{lem:slicing-lemma}, we can apply Theorem~\ref{thm:klr} to $G[S_1
  \cup T_1 \cup S_2 \cup \dotsb \cup S_{t/2} \cup T_{t/2}]$ and find a canonical
  copy of $C_t$ in it. We use this observation several times throughout the
  proof without explicitly mentioning which sets we use.

  Let $S_1 \cup \dotsb \cup S_{t/2} = V_{2i-1}$ and $T_1 \cup \dotsb \cup
  T_{t/2} = V_{2i}$ be equipartitions such that every $(S_i, T_j)$ is
  $(t\eps/\alpha, \alpha q)$-regular with density \emph{precisely} $\alpha q$,
  where $q = C^\star n^{-(k-1)/k}$. (This is a standard way of controlling the
  density between regular pairs; see, e.g.~\cite[Lemma~4.3]{gerke2005sparse}, or
  simply think of taking a random subset of edges.) As $t\eps/\alpha \leq
  \eps_{\ref{lem:gnp-all-reg-expander-properties}}(\alpha, \gamma')$, we can
  apply Lemma~\ref{lem:gnp-all-reg-expander-properties} with $\gamma'$ (as
  $\gamma$) to these sets to conclude that there exist sets $V_i^1, V_i^3,
  \dotsc, V_i^{t-1} \subseteq V_{2i-1}$ and $V_i^2, V_i^4, \dotsc, V_i^t
  \subseteq V_{2i}$ such that $G[V_i^1 \cup \dotsb \cup V_i^t]$ belongs to the
  class $\cGEk(C_t, \tilde n, \gamma', \alpha q)$. Let $V_0' := V(G) \setminus
  \bigcup_{i \in [\ell], j \in [t]} V_i^j$, and note that $V_0 \subseteq V_0'$
  and $|V_0'| \leq 2\gamma' n$. The first mini-goal is to find a collection of
  disjoint $t$-cycles covering all vertices of $V_0'$, without hurting the
  $(\gamma',k)$-typical property of vertices in $G[V_i^1 \cup \dotsb \cup
  V_i^t]$ drastically.

  Let $W_i^j \cup X_i^j \cup U_i^j = V_i^j$ be a partition of each $V_i^j$
  chosen uniformly at random such that
  \[
    |W_i^j| = \delta_w\tilde n, \qquad |X_i^j| = \delta_x\tilde n, \qquad
    \text{and} \qquad |U_i^j| = (1-\delta_w-\delta_x)\tilde n,
  \]
  all cardinalities divisible by $t$; in particular $|W_i^j| \ll |X_i^j| \ll
  |U_i^j|$. Let $W := \bigcup_{i \in [\ell], j \in [t]} W_i^j$ and note $|W|
  \geq (\delta_w/2)n$. By Lemma~\ref{lem:reg-exp-partitioning} applied with
  $\gamma''$ (as $\gamma$) and $\delta_w$ (as $\delta$), w.h.p.\ for every $i
  \in [\ell]$, $j \in [t]$ we have:
  \stepcounter{propcnt}
  \begin{enumerate}[label=\normalfont{(\Alph{propcnt}{\arabic*})},leftmargin=3em]
    \item\label{main-exp-inh} every $v \in V_i^j$ is $(\gamma'',k)$-expanding
      (with $\alpha q$ as $p$) with respect to $U_i^{j+1} \cup \dotsb \cup
      U_i^{j+k}$ and $U_i^{j-1} \cup \dotsb \cup U_i^{j-k}$.
  \end{enumerate}
  Observe that every $v \in V(G)$ has either $\deg_G(v, V(G) \setminus V_0')
  \geq (1/2+\alpha)np$ or $\deg(v, V_0') \geq \alpha np \geq (1/2+\alpha)|W \cup
  V_0'|p$. Hence, as a consequence of Chernoff's inequality and the union bound,
  w.h.p.\
  \[
    \delta(G[W \cup V_0']) \geq (1-o(1))(1/2+\alpha)|W|p \geq (1/2+\alpha/2)|W
    \cup V_0'|p,
  \]
  where the last inequality follows from $|V_0'| \leq 2\gamma' n$ and $\gamma'$
  being small enough with respect to $\delta_w$ and $\alpha$. We fix a partition
  $W_i^j \cup X_i^j \cup U_i^j$ of each $V_i^j$ satisfying all of the above.
  This puts us in the setting of Lemma~\ref{lem:gnp-cycles-from-min-deg} which
  is applied with $\alpha/2$ (as $\alpha$), $\delta_w/2$ (as $\mu$), $W$ (as
  $U$), $V_0'$ (as $X$), and we conclude that there is a collection of disjoint
  $t$-cycles covering all vertices of $V_0'$ in $G$ and some vertices of $W$.
  This can be done as $|V_0'| \leq 2\gamma'n \leq \delta'|W|$ by our choice of
  constants.

  Let $\tilde X_i^j$ be sets obtained by pushing the unused vertices for the
  previously found collection from each $W_i^j$ into $X_i^j$. At this point we
  would ideally use our blow-up lemma (Theorem~\ref{thm:blow-up-lemma}) for
  every $G[(\tilde X_i^1 \cup U_i^1) \cup \dotsb \cup (\tilde X_i^t \cup
  U_i^t)]$ to cover all the remaining vertices, however, the sets $\tilde X_i^1,
  \dotsc, \tilde X_i^t$ are not necessarily balanced any more, i.e.\ we only
  know that $\big||\tilde X_i^{j_1}|-|\tilde X_i^{j_2}|\big| \leq \delta_w\tilde
  n$, for all $1 \leq j_1 < j_2 \leq t$. The remainder of the proof consists of
  balancing these sets and then applying the blow-up lemma. It is convenient to
  do so when the cardinality of $\tilde X_i^j \cup U_i^j$ is divisible by $t$ so
  we first make sure this is the case.

  The idea is to find a set $Q$ so that $G[Q]$ contains a $C_t$-factor and the
  cardinality of each $\tilde X_i^j \setminus Q$ is divisible by $t$. We do so
  iteratively, for every $i = 1, \dotsc, \ell$, by adding some collection of
  $t$-cycles to the set $Q$ in every step of the way. Recall, $\tilde X_i^{j_1}
  \subseteq V_{2i-1}$ for odd $j_1 \in [t]$ and $\tilde X_i^{j_2} \subseteq
  V_{2i}$ for even $j_2 \in [t]$. If for all $j \in [t]$ the cardinality of
  $X_i^j$ is divisible by $t$, continue to the next index $i$. Suppose $|\tilde
  X_i^1| \bmod t = x$, for some $1 \leq x \leq t-1$. We apply
  Theorem~\ref{thm:klr} to $G[\tilde X_i^1 \cup \tilde X_i^2 \cup \tilde
  X_{i+1}^3 \cup \tilde X_i^4 \cup \tilde X_{i+1}^5 \cup \dotsb \cup \tilde
  X_i^t]$ to find $x$ canonical copies of $C_t$ and then to $G[\tilde X_{i+1}^1
  \cup \tilde X_i^2 \cup \tilde X_{i+1}^3 \cup \dotsb \cup \tilde X_i^t]$ to
  find $t - x$ canonical copies of $C_t$, whose vertices we all add to $Q$. In
  particular, these cycles are such that
  \[
    |\tilde X_i^1 \cap Q| = x, \quad |\tilde X_{i+1}^1 \cap Q| = t-x, \quad
    \text{and} \quad |\tilde X_i^{j_2} \cap Q| = |\tilde X_{i+1}^{j_1} \cap Q| =
    t,
  \]
  for all odd $j_1 \in [t] \setminus \{1\}$ and even $j_2 \in [t]$. This can be
  done as $|\tilde X_i^j| \geq \delta_x\tilde n \geq 2\eps'\tilde n$. We can
  repeat this in a similar fashion for all $\tilde X_i^j$, $j \in [t]$, which
  ensures the number of remaining vertices in $\tilde X_i^j$ are each divisible
  by $t$. While `sliding' the divisibility issue across the sets $X_1^\star,
  \dotsc, X_\ell^\star$, analogously as above, we construct a set $Q$ of
  constant size (at most $t^3\ell$), such that $G[Q]$ has a $C_t$-factor, and
  \[
    |\tilde X_\ell^{j_1} \setminus Q| \bmod t = y_{j_1} \quad \text{and} \quad
    |\tilde X_\ell^{j_2} \setminus Q| \bmod t = y_{j_2},
  \]
  for odd $j_1 \in [t]$ and even $j_2 \in [t]$. Let $y_1 = \sum_{} y_{j_1}$ and
  $y_2 = \sum_{} y_{j_2}$ and note that $(y_1 + y_2) \bmod t = 0$. Assume
  without loss of generality that $0 \leq y_1 \leq y_2 \leq t$. Otherwise, we
  can just apply Theorem~\ref{thm:klr} to subsets of $\tilde X_\ell^1, \tilde
  X_\ell^2, \dotsc, \tilde X_\ell^{t-1}, \tilde X_\ell^t$ to find several copies
  of $C_t$ until this is the case. Let now $z \in [\ell-1]$ be an index so that
  either $\{2\ell-1, 2\ell, 2z-1\}$ or $\{2\ell-1, 2\ell, 2z\}$ is a triangle in
  $R$; this index exists as $\delta(R) \geq (1/2+\alpha/2)|R|$ and assume
  $\{2\ell-1, 2\ell, 2z-1\}$ is this triangle. Thus,
  $V_{2\ell-1},V_{2\ell},V_{2z-1}$ are pairwise $(\eps,p)$-regular with density
  at least $\alpha p$, and we can use a similar trick as above, this time with
  sets $\tilde X_\ell^1, \dotsc, \tilde X_\ell^t, \tilde X_z^1$, to find a set
  $Q'$ disjoint from $Q$ of constant size so that $G[Q']$ has a $C_t$-factor. In
  particular, these cycles are such that
  \[
    |V_{2\ell-1} \cap Q'| = y_1, \quad |V_{2\ell} \cap Q'| = y_2, \quad
    \text{and} \quad |\tilde X_z^1 \cap Q'| = 2t - (y_1 + y_2).
  \]
  More importantly, the cardinalities of sets $\tilde X_\ell^{j_1} \setminus (Q
  \cup Q')$ and $\tilde X_\ell^{j_2} \setminus (Q \cup Q')$ are all divisible by
  $t$. As this whole procedure removes only a constant number of vertices from
  each $\tilde X_i^j$, we may as well assume that every $\tilde X_i^j$ is such
  that $|\tilde X_i^j| \bmod t = 0$ to begin with.

  We proceed with the balancing procedure. Let $\phi$ be a function $\phi
  \colon [2\ell] \to [\ell]$ such that
  \begin{enumerate}[label=(\emph{\roman*}), ref=(\emph{\roman*})]
    \item\label{phi-reg} $\{i\} \cup e_{\phi(i)}$ is a triangle in $R$, for all
      $i \in [2\ell]$,
    \item\label{phi-reverse} $|\phi^{-1}(z)| \leq 2/\alpha$ for all $z \in
      [\ell]$.
  \end{enumerate}
  Clearly as $\delta(R) \geq (1/2+\alpha/2)|R|$, fulfilling \ref{phi-reg} is
  trivial. For \ref{phi-reverse}, let $\hbar_i$ be the number of indices $z \in
  [\ell]$ for which \ref{phi-reg} holds for a fixed $i$ by setting $\phi(i) :=
  z$. Then again by the minimum degree of $R$ we have $2\hbar_i + (\ell-\hbar_i)
  \geq (1+\alpha)\ell$, giving $\hbar_i \geq \alpha\ell$. Thus, there exists an
  assignment satisfying \ref{phi-reg} so that every index in $[\ell]$ is chosen
  at most $2\ell/(\alpha\ell) = 2/\alpha$ times.

  The goal is to construct a set $Q''$ such that $G[Q'']$ has a $C_t$-factor and
  $|\tilde X_i^1 \setminus Q''| = \dotsb = |\tilde X_i^t \setminus Q''|$ for all
  $i \in [\ell]$. We do so iteratively (greedily), at the beginning having $Q''$
  as an empty set. We let (with slight abuse of notation perhaps) $\tilde X_i^j
  := \tilde X_i^j \setminus Q''$ throughout the process. The edge $e_i$ in $R$
  is said to be \emph{balanced} if the underlying sets $\tilde X_i^j$ are of
  equal size. Assume we have so far balanced all the edges $e_1, \dotsc,
  e_{i-1}$, and let us balance the edge $e_i$. Without loss of generality,
  $|\tilde X_i^1| \geq \dotsb \geq |\tilde X_i^t|$ and $|\tilde X_i^1| - |\tilde
  X_i^t| = \delta\tilde n$, for some $1 \leq \delta \leq \delta_w$. As $|\tilde
  X_i^1|$ and $|\tilde X_i^t|$ are divisible by $t$, it follows that
  $\delta\tilde n \bmod t = 0$. Let $\phi(2i-1) = z$, so by \ref{phi-reg} we
  have that $V_{2i-1}, V_{2z-1}, V_{2z}$ are pairwise $(\eps, p)$-regular with
  density at least $\alpha p$ in $G$. Importantly, as we establish later,
  $\tilde X_i^j \setminus Q''$ are throughout the rebalancing process of size at
  least $\delta_x\tilde n/2 \geq \eps'\tilde n$ so that Theorem~\ref{thm:klr}
  can be applied.

  That being said, we apply Theorem~\ref{thm:klr} to $G[(\tilde X_i^1 \setminus
  Q'') \cup (\tilde X_z^2 \setminus Q'') \cup \dotsb \cup (\tilde X_z^t
  \setminus Q'')]$, to find $\delta\tilde n/t$ canonical copies of $C_t$ whose
  vertices we add to $Q''$. Repeat this $t-1$ more times, where the $j$-th time
  the set $\tilde X_z^j$ is the one left out, and then once more where $\tilde
  X_i^1$ is the one left out. This moves exactly $\delta\tilde n$ vertices of
  $\tilde X_i^1$ to $Q''$ and while using some vertices of $\tilde X_z^1,
  \dotsc, \tilde X_z^t$, their number is exactly the same and divisible by
  $t$---namely it is $(t-1)\delta\tilde n/t + \delta\tilde n/t$ in each. By
  proceeding in the same way with $\tilde X_i^2, \dotsc, \tilde X_i^{t-1}$ we
  balance the edge $e_i$.

  We now give the promised bound on the size of the set $Q''$ throughout the
  process. For every $i \in [\ell]$ we add at most $\delta_w\tilde n$ new
  vertices to $Q''$ from $\tilde X_i^j$. Additionally, by~\ref{phi-reverse}, at
  most $(2/\alpha)t\delta_w\tilde n$ vertices from it are used for balancing
  other edges. Hence, $|\tilde X_i^j \setminus Q''| \geq \delta_x\tilde n -
  (2t/\alpha + 1)\delta_w\tilde n \geq \delta_x\tilde n/2$ as promised, by our
  choice of constants.

  Finally, let $\tilde V_i^j$ denote the set of vertices obtained by adding the
  remaining vertices of each $\tilde X_i^j$ back into $U_i^j$. Write $G_i :=
  G[U_i^1 \cup \dotsb \cup U_i^t]$ and $\tilde G_i := G[\tilde V_i^1 \cup \dotsb
  \cup \tilde V_i^t]$. We claim that every $v \in V(\tilde G_i)$ is
  $(\gamma,k)$-typical in $\tilde G_i$. Using \ref{main-exp-inh} for every $v
  \in \tilde V_i^1$ and $j \in [k-1]$, we have
  \[
    |N_{\tilde G_i}^j(v)| \geq |N_{G_i}^j(v)| \osref{\ref{main-exp-inh}}\geq
    (1-\gamma'')\big((1-\delta_w-\delta_x)\tilde n\alpha q\big)^j \geq
    (1-\gamma)(\tilde n\alpha q)^j.
  \]
  Moreover, as $|N_{\tilde G_i}^{k-1}(v)| \geq |N_{G_i}^{k-1}(v)|/2$, it follows
  that $(N_{\tilde G_i}^{k-1}(v, \tilde V_i^k), \tilde V_i^{k+1})$ is $(\gamma,
  p)$-lower-regular. Lastly, as $G[V_i^1 \cup \dotsb \cup V_i^t] \in \cGEk(C_t,
  \tilde n, \gamma', \alpha q)$,
  \[
    \deg_{\tilde G_i}(v, \tilde V_i^{j+1}) \leq \deg_{G_i}(v, V_i^{j+1}) \leq
    (1+\gamma')\tilde n\alpha q \leq (1+\gamma)|\tilde V_i^{j+1}|\alpha q,
  \]
  for every $v \in \tilde V_i^j$ and $j \in [t]$. For every $i \in [\ell]$ let
  $s_i := |\tilde V_i^1| = \dotsb = |\tilde V_i^t|$. So, each $\tilde G_i$
  belongs to the class $\cGEk(C_t, s_i, \gamma, \alpha q)$ and we apply the
  blow-up lemma (Theorem~\ref{thm:blow-up-lemma}) with $\mu/2$ (as $\mu$) to
  find a $C_t$-factor in each $\tilde G_i$ and complete the proof.

  In order to make this whole thing work for an odd $t$, instead of a Hamilton
  cycle one would first find the square of a Hamilton cycle in $R$. Then,
  instead of working with edges $e_i$ throughout one would work with triangles.
  Lastly, the minimum degree of $R$ is then $(2/3+\alpha/2)|R|$, so for the
  balancing procedure one can use copies of $K_4$ each triangle belongs to. The
  rest of the proof remains basically identical.
\end{proof}

\paragraph*{Acknowledgements.} The author would like to thank Rajko Nenadov with
whom the topic of cycle factor resilience had been discussed a couple of years
ago. A big thanks goes to Kalina Petrova for carefully reading a prior version
of the manuscript, which greatly helped in improving the exposition.

{\small \bibliographystyle{abbrv} \bibliography{references}}

\appendix
\section{The missing technical proofs}
\label{sec:appendix}

Here we provide the missing proofs of Proposition~\ref{prop:absorber-properties}
and Proposition~\ref{prop:almost-all-expanders}.

\subsection{Proof of Proposition~\ref{prop:absorber-properties}}
  Property \ref{abs-two-factors} should be clear from construction and
  \ref{abs-blow-up} is trivial: starting from cycles of length $t$ containing
  $r_1,\dotsc,r_t$ greedily assign labels $1, \dotsc, t$ to vertices of every copy
  of $C_t$ in $\Fabs$ such that each $r_i$ receives a different label and every
  copy of $C_t$ has all labels represented. Then embed all the vertices with label
  $i$ into class $V_i$ of $\cG(C_t, v(\Fabs), 0, 1)$.

  We prove~\ref{abs-m2-density} in the remainder. For a graph $H$ with $e(H)
  \geq 2$, let $d_2(H) := (e(H)-1)/(v(H)-2)$; then $m_2(H) = \max_{H' \subseteq
  H} d_2(H')$. The proof for $t = 2k$ is almost trivial. By construction
  $\Fconn$ has girth at least $t$ and is planar. It is well known (and easy to
  prove using Euler's formula) that every planar graph $H$ with girth at least
  $t$ satisfies $e(H) \leq \frac{t}{t-2}(v(H)-2)$. Therefore,
  \[
    m_2(\Fconn) \leq \frac{\tfrac{t}{t-2}(v(H)-2) - 1}{v(H)-2} < \frac{t}{t-2} =
    \frac{k}{k-1},
  \]
  as desired.

  The proof for $t = 2k-1$ is much more cumbersome. We extensively and without
  referencing make use of the fact that for $a, b, c, d, q > 0$, $a/b \leq q$ and
  $c/d \leq q$ implies $(a+c)/(b+d) \leq q$. The following observation is very
  useful.

  \begin{claim}\label{cl:2-density-intersection}
    Two connected graphs $H_1$ and $H_2$ which intersect in a vertex and have no
    edges between them satisfy $m_2(H_1 \cup H_2) \leq \max\{m_2(H_1),
    m_2(H_2)\}$. \qed
  \end{claim}

  Let $F_i$ be the graphs obtained by removing the edges of the $t$-cycle
  $s_1,\dotsc,s_t$ from $\Fconn$. Since $m_2(C_t) = (t-1)/(t-2) \leq k/(k-1)$, by
  Claim~\ref{cl:2-density-intersection} it is sufficient to show that $m_2(F_i)
  \leq k/(k-1)$. We do this by iteratively applying the next claim.

  \begin{claim}\label{cl:recursive-density}
    Let $v_1,\dotsc,v_{t-1}, u$ be vertices and let $H_1,\dotsc,H_{t-1}$ be
    graphs with $v_i, u \in V(H_i)$ and which otherwise are pairwise disjoint.
    Suppose $e(F)/(v(F)-2) \leq k/(k-1)$ for every $F \subseteq H_i$ which
    contains $v_i, u$, and $m_2(H_i) \leq k/(k-1)$. Let $H$ be a graph obtained
    by adding a vertex $v$ to $\bigcup_i H_i$ and adding a copy of $C_t$ on $v,
    v_1, \dotsc, v_{t-1}$. Then $e(F)/(v(F)-2) \leq k/(k-1)$ for every $F
    \subseteq H$ which contains $v, u$ and $m_2(H) \leq k/(k-1)$.
  \end{claim}
  \begin{proof}
    Consider $F \subseteq H$ which contains $v, u$ and let $F_i := F \cap H_i$,
    $e_i := e(F_i)$, and $v_i := v(F_i)$. Then
    \[
      \frac{e(F)}{v(F)-2} \leq \frac{\sum_{i=1}^{t-1}e_i + t}{\sum_{i=1}^{t-1}v_i
      - (t-2) + 1 - 2} = \frac{\sum_{i=1}^{t-1}e_i + t}{\sum_{i=1}^{t-1}(v_i-1)}.
    \]
    Using the assumption $e_i \leq \frac{k}{k-1} (v_i-2)$, the above can further
    be bounded by
    \[
      \frac{\frac{k}{k-1}\sum_{i=1}^{t-1}(v_i-2) + t}{\sum_{i=1}^{t-1}(v_i-1)} =
      \frac{\frac{k}{k-1}\sum_{i=1}^{t-1}(v_i-1) - \frac{k}{k-1}(t-1) +
      t}{\sum_{i=1}^{t-1}(v_i-1)}.
    \]
    The conclusion then follows as $k(t-1) \geq t(k-1)$.

    For the second part, if $F \subseteq H$ contains both $v, u$ then $d_2(F) \leq
    k/(k-1)$ by the above. Similarly, if $F$ contains at most one of $v, u$ then
    $d_2(F) \leq k/(k-1)$. Lastly, if $F$ contains neither $v$ nor $u$, then
    $d_2(F) \leq k/(k-1)$ follows from $m_2(H_i) \leq k/(k-1)$ and
    Claim~\ref{cl:2-density-intersection}.
  \end{proof}

  For a definition of ladders and $\Fconn$ we refer the reader to
  Section~\ref{sec:absorbing-method} and in particular Figure~\ref{fig:ladder}
  and Figure~\ref{fig:C3-switcher}. Let $\mathrm{CL}_k$ stand for a graph
  consisting of two $t$-cycles which are $k$-ladder-connected, with $x$ and $y$
  denoting the vertices $v_1$ and $u_1$ and let $\mathrm{L}_k$ stand for a
  $(k-1,k)$-ladder of length $2k-1$ (just `ladder' in what is to come), with $a$
  and $b$ denoting the vertices $w_{1,1}$ and $w_{2k-1,1}$. Let
  $\mathrm{CL}_k^+$ be a graph obtained by starting from two cycles of length
  $t$ on vertices $\{v, x_1,\dotsc,x_{t-1}\}$ and $\{u, y_1,\dotsc,y_{t-1}\}$,
  and adding disjoint copies of $\mathrm{CL}_k$ between each pair $x_i, y_i$.
  For a better visual representation, see Figure~\ref{fig:density}.

  \begin{figure}[!htbp]
    \centering
    \includegraphics[scale=0.7]{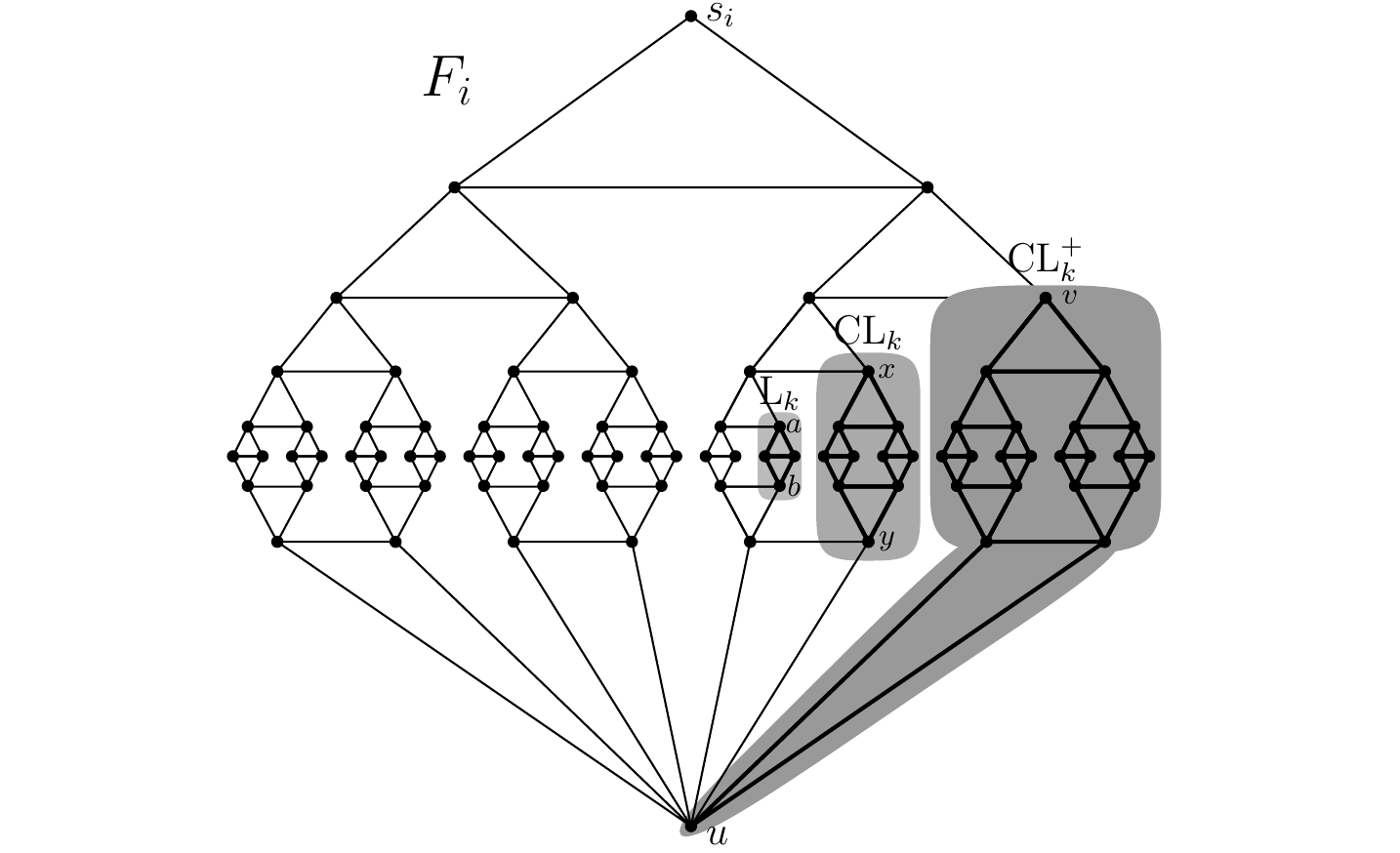}
    \caption{An example of $F_i$ and its subgraphs defined above for $k = 2$ and
    $t = 3$}
    \label{fig:density}
  \end{figure}

  Crucially, observe that $F_i$ can be obtained by an iterative procedure: set
  $H := \mathrm{CL}_k^+$ and take $t-1$ copies of $H$ which share the vertex $u$
  and are otherwise disjoint, let $v_i$ stand for the vertex $v$ of the $i$-th
  copy of $H$, and add a vertex $v_t$; add a $t$-cycle on $v_1,\dotsc,v_{t-1},
  v_t$; redeclare the newly obtained graph to be $H$, set $v := v_t$, and
  continue the process $k$ times until $H = F_i$, i.e.\ until $v = s_i$ is
  `reached'. Therefore, by Claim~\ref{cl:recursive-density} in order to complete
  the proof we need to show that $e(F)/(v(F)-2) \leq k/(k-1)$ for every $F
  \subseteq \mathrm{CL}_k^+$ which contains $v, u$, and $m_2(\mathrm{CL}_k^+)
  \leq k/(k-1)$. We work our way from the ground up.

  \begin{claim}\label{cl:rooted-extended-ladder}
    Let $H$ be a graph obtained by removing one of the ladders from
    $\mathrm{CL}_k$. Then $\frac{e(F)+1}{v(F)-2} \leq \frac{k}{k-1}$ for every $F
    \subseteq H$ with $x, y \in V(F)$.
  \end{claim}
  \begin{proof}
    Let $e := e(F)$, $v := v(F)$, and let $c$ denote the number of induced cycles
    in $F$. It is not too difficult to see that $e = v + c - 1$. In order to show
    $(e+1)/(v-2) \leq k/(k-1)$ it is thus sufficient to establish $v \geq c(k-1) +
    2k$. If $c = 0$, $F$ is a tree and trivially $v/(v-2) \leq k/(k-1)$ as $v
    \geq 2k+1$. If $c \geq 1$ is odd, then the number of vertices in $F$ is
    at least: $2k+1$ for an $xy$-path and $(c+1)/2 \cdot (t-2)$ to close $c$
    cycles. So,
    \[
      v \geq 2k+1 + \frac{c+1}{2}(t-2) = 2k + c(k-1) + k - \frac{1}{2} -
      \frac{c}{2}.
    \]
    On the other hand, if $c \geq 2$ is even, then the number of vertices in $F$
    is at least: $2k+1$ for an $xy$-path and $c/2 \cdot (t-2) + k - 2$ to close
    $c$ cycles. So,
    \[
      v \geq 2k+1 + \frac{c}{2}(t-2) + k-2 = 2k + c(k-1) + k - 1 - \frac{c}{2}.
    \]
    As $c \leq t-1 = 2(k-1)$ for even $c$ and $c \leq t-2 = 2(k-1)-1$ for odd $c$,
    the above in both cases gives $v \geq 2k+c(k-1)$ as desired.
  \end{proof}

  \begin{claim}\label{cl:rooted-ladder}
    $\frac{e(F)}{v(F)-2} \leq \frac{k}{k-1}$ for every $F \subseteq \mathrm{L}_k$
    with $a, b \in V(F)$.
  \end{claim}
  \begin{proof}
    The proof is almost identical to that of the previous claim.
  \end{proof}

  \begin{claim}\label{cl:rooted-cycle-ladder}
    $\frac{e(F)}{v(F)-2} \leq \frac{k}{k-1}$ for every $F \subseteq \mathrm{CL}_k$
    with $x, y \in V(F)$.
  \end{claim}
  \begin{proof}
    Let $H_1$ be the graph containing $x, y$ obtained by removing one of the
    ladders from $\mathrm{CL}_k$, and $H_2$ defined similarly by removing the
    other. In particular, $V(H_1) \cap V(H_2) = \{x,y\}$, $E(H_1) \cap E(H_2) =
    \varnothing$, and $\mathrm{CL}_k = H_1 + H_2 + a_1a_2 + b_1b_2$. Consider some
    $F \subseteq \mathrm{CL}_k$ which contains $x, y$ and let $F_i := F \cap H_i$,
    $e_i := e(F_i)$, and $v_i := v(F_i)$. As $e(F) \leq e_1 + e_2 + 2$ and $v(F)-2
    \geq v_1 + v_2 - 4$, and by Claim~\ref{cl:rooted-extended-ladder}
    $(e_i+1)/(v_i-2)\leq k/(k-1)$ for every $F_i$ containing $x$ and $y$, the
    desired conclusion follows.
  \end{proof}

  \begin{claim}
    $\frac{e(F)}{v(F)-2} \leq \frac{k}{k-1}$ for every $F \subseteq
    \mathrm{CL}_k^+$ with $v, u \in V(F)$.
  \end{claim}
  \begin{proof}
    Let $F_i$ denote the subgraph of $\mathrm{CL}_k$ between $x_i, y_i$ which
    belongs to $F$, and let $e_i := e(F_i)$ and $v_i := v(F_i)$. Then
    \[
      \frac{e(F)}{v(F)-2} \leq \frac{\sum_{i=1}^{t-1} e_i + 2t}{\sum_{i=1}^{t-1}
      v_i}.
    \]
    By Claim~\ref{cl:rooted-cycle-ladder} we have $\sum_{i=1}^{t-1} e_i \leq
    \frac{k}{k-1}\sum_{i=1}^{t-1}(v_i-2)$. Plugging this into the estimate above
    gives
    \[
      \frac{e(F)}{v(F)-2} \leq \frac{\frac{k}{k-1}\sum_{i=1}^{t-1} v_i -
      \frac{2k(t-1)}{k-1} + 2t}{\sum_{i=1}^{t-1}v_i} < \frac{k}{k-1},
    \]
    where the last inequality follows from $k(t-1) > t(k-1)$.
  \end{proof}

  Observe that this shows $d_2(F) \leq \frac{k}{k-1}$ for every $F \subseteq
  \mathrm{CL}_k^+$ which contains both $v, u$, and similarly which contains at
  least one of $v, u$. It remains to show $d_2(F) \leq \frac{k}{k-1}$ for every $F
  \subseteq \mathrm{CL}_k^+$ which does not contain $v, u$. We again go from the
  ground up.

  \begin{claim}\label{cl:2-density-ladder}
    $m_2(\mathrm{L}_k) \leq \frac{k}{k-1}$.
  \end{claim}
  \begin{proof}
    Note that any subgraph that maximises the $2$-density has to be $2$-connected.
    Now, every such subgraph $F$ of a ladder can be obtained by starting from one
    copy of a cycle of length $\ell \geq t$ and iteratively attaching $c \geq 0$
    paths of length at least $k$ by their endpoints. So
    \[
      d_2(F) = \frac{t - 1 + ck}{t - 2 + c(k-1)} \leq \frac{k}{k-1},
    \]
    which holds as $(t-1)/(t-2) \leq k/(k-1)$ for $k \geq 2$. One easily checks
    that starting with a cycle longer than $t$ or adding paths longer than $k$
    gives an even smaller density estimate.
  \end{proof}

  \begin{claim}\label{cl:2-density-joint-ladders}
    Let $H$ be the graph obtained by removing $x$ and $y$ from $\mathrm{CL}_k$.
    Then $m_2(H) \leq \frac{k}{k-1}$.
  \end{claim}
  \begin{proof}
    Let $H_1, H_2$ be the copies of ladders. Consider $F \subseteq H$ and let
    $F_i := F \cap H_i$, $e_i := e(F_i)$, and $v_i := v(F_i)$. If $F$ contains
    at most one of the edges $a_1a_2$ and $b_1b_2$ then by
    Claim~\ref{cl:2-density-ladder} and Claim~\ref{cl:2-density-intersection}
    $d_2(F) \leq k/(k-1)$. Otherwise,
    \[
      \frac{e(F)-1}{v(F)-2} \leq \frac{e_1+e_2+2-1}{v_1+v_2-2} = \frac{(e_1+1) +
      e_2}{(v_1-1) + (v_2-1)}.
    \]
    Using Claim~\ref{cl:rooted-ladder} and the fact that $e/(v-2) \geq
    (e+1)/(v-1)$ for every connected graph, we have $(e_1+1)/(v_1-1) \leq
    e_1/(v-2) \leq k/(k-1)$ and trivially $e_2/(v_2-1) \leq e_2/(v_2-2) \leq
    k/(k-1)$. The conclusion then follows.
  \end{proof}

  \begin{claim}
    Let $H$ be the graph obtained by removing $u$ and $v$ from
    $\mathrm{CL}_k^+$. Then $m_2(H) \leq \frac{k}{k-1}$.
  \end{claim}
  \begin{proof}
    Let $H_i$, $i \leq t-1$, be the copy of $\mathrm{CL}_k$ between $x_i$ and
    $y_i$. Consider $F \subseteq H$ and let $F_i := F \cap H_i$, $e_i :=
    e(F_i)$, and $v_i := v(F_i)$. If $F$ contains all of the vertices
    $x_1,\dotsc,x_{t-1}$ and $y_1,\dotsc,y_{t-1}$, then as $e_i \leq
    \frac{k}{k-1}(v_i-2)$ by Claim~\ref{cl:rooted-cycle-ladder}
    \begin{align*}
      \frac{e(F)-1}{v(F)-2}
      & \leq \frac{\sum_{i=1}^{t-1}e_i + 2(t-2) - 1}{\sum_{i=1}^{t-1}v_i - 2} \leq
      \frac{\frac{k}{k-1}\sum_{i=1}^{t-1}(v_i-2) +
      2(t-2)-1}{\sum_{i=1}^{t-1}v_i-2} \\
      & = \frac{\frac{k}{k-1}(\sum_{i=1}^{t-1}v_i - 2) - \frac{2k(t-2)}{k-1} +
      2(t-2)-1}{\sum_{i=1}^{t-1}v_i-2} < \frac{k}{k-1}.
    \end{align*}
    Otherwise, if $F$ does not contain some $x_i$ or $y_i$, then
    Claim~\ref{cl:2-density-joint-ladders} and
    Claim~\ref{cl:2-density-intersection} give the same result.
  \end{proof}

\subsection{Proof of Proposition~\ref{prop:almost-all-expanders}}
\label{sec:appendix-expanders}

We first list a couple of lemmas from \cite{gerke2007small} which are used as
tools in the proof, namely \cite[Lemma~3.1]{gerke2007small} and
\cite[Corollary~3.8]{gerke2007small}

\begin{lemma}\label{lem:reg-small-sets-expand}
  For all $\beta, \lambda > 0$ there exists a positive $\eps_0 = \eps_0(\beta,
  \gamma)$ such that for all $\eps \leq \eps_0$, $p > 0$, and $\tilde q
  \leq \lambda/p$, every $(\eps, p)$-lower-regular graph $G(V_1 \cup V_2, E)$
  satisfies that, for any $q \geq \tilde q$, the number of sets $Q \subseteq
  V_1$ of size $q$ with $|N_G(Q)| < (1-3\lambda)\tilde q|V_2|p$ is at most
  \[
    \beta^q \binom{|V_1|}{q}.
  \]
\end{lemma}

\begin{lemma}\label{lem:reg-small-set-inheritance}
  For all $\beta, \gamma > 0$ there exist positive $\eps_0 = \eps_0(\beta,
  \gamma)$ and $D = D(\gamma)$, such that for all $0 < \eps \leq \eps_0$ and $0
  < p < 1$, the following holds. Let $G(V_1 \cup V_2, E)$ be an $(\eps,
  p)$-lower-regular graph and suppose $q_1, q_2 \geq Dp^{-1}$. Then the number
  of pairs $(Q_1, Q_2)$ with $Q_i \subseteq V_i$ and $|Q_i| = q_i$ ($i = 1,2$)
  which do not form a $(\gamma, p)$-lower-regular graph is at most
  \[
    \beta^{\min\{q_1,q_2\}} \binom{|V_1|}{q_1} \binom{|V_2|}{q_2}.
  \]
\end{lemma}

Next lemma is a two-sided version of \cite[Lemma~5.8]{gerke2007small} and its
proof follows exactly the same steps.

\begin{lemma}\label{lem:bad-extensions-pairs}
  Let $c \geq 1$ and let $\beta, \delta > 0$. Then there exists a positive
  $\gamma = \gamma(\beta, \delta)$ such that the following holds. Let $V_1, V_2$
  be sets of size $|V_i| = n$, such that for all $q_1,q_2 \geq c$ at most
  $\gamma^{\min\{q_1,q_2\}}\binom{n}{q_1}\binom{n}{q_2}$ pairs $(Q_1, Q_2)$,
  with $Q_i \subseteq V_i$ and $|Q_i| = q_i$, are marked. Then there are at most
  \[
    \beta^m \binom{n s}{m_1} \binom{n s}{m_2}
  \]
  graphs $G$ on vertex set $V_1 \cup V_2 \cup S_1 \cup S_2$ with $|S_i| = s$,
  $m/2 \leq m_i \leq m$ edges in $G[V_i,S_i]$, and $m \geq 4s\log(ns)$, for
  which there exist pairwise disjoint pairs of sets $(X_1,Y_1), (X_2,Y_2),
  \dotsc$ such that $X_i \subseteq S_1$, $Y_i \subseteq S_2$, with $\sum_i
  \min\{|X_i|, |Y_i|\} \geq \delta s$, and for each $i$, $|N_G(X_i)| \geq
  \max\{|X_i|m_1/(2n), c\}$, $|N_G(Y_i)| \geq \max\{|Y_i|m_2/(2n), c\}$, and
  $(N_G(X_i), N_G(Y_i))$ is a marked pair.
\end{lemma}
\begin{proof}
  Firstly, we select pairwise disjoint sets $X_1, X_2, \dots$ and $Y_1, Y_2,
  \dots$ for which there are $s^{2s} \leq 2^m$ choices, as there are at most $s$
  sets $X_i$ and likewise $Y_i$. Secondly, for each $i$, we select the sizes of
  neighbourhoods $d_x(i) := |N_G(X_i)|$, $d_y(i) := |N_G(Y_i)|$, and the number
  of edges $m_x(i)$ between $X_i$ and $V_1$ and $m_y(i)$ between $Y_i$ and
  $V_2$. This can be done in at most
  \[
    n^{2s} \cdot m_1^s \cdot m_2^s \leq 2^m
  \]
  ways. Thirdly, for each $i$, we select sets $Q_x$ of size $d_x(i)$ in $V_1$
  and $Q_y$ of size $d_y(i)$ in $V_2$ such that $(Q_x, Q_y)$ is a marked pair,
  and select edges between $X_i, Y_i$ and the chosen sets $Q_x, Q_y$. As $X_i$
  and $Y_i$ are all disjoint, writing $x_i := |X_i|$ and $y_i := |Y_i|$, for
  every $i$ there are at most
  \[
    \gamma^{\min\{d_x(i), d_y(i)\}}\binom{n}{d_x(i)} \binom{n}{d_y(i)} \cdot
    \binom{x_i d_x(i)}{m_x(i)} \binom{y_i d_y(i)}{m_y(i)}
  \]
  choices in total. Lastly, we select the edges in $G[V_1, S_1 \setminus
  \bigcup_i X_i]$ and $G[V_2, S_2 \setminus \bigcup_i Y_i]$. There are at most
  \[
    \binom{n (s-x)}{m_1 - \tilde m_1} \binom{n (s-y)}{m_2 - \tilde m_2},
  \]
  ways to do this, where $\tilde m_1 = \sum_i m_x(i)$, $\tilde m_2 = \sum_i
  m_y(i)$, and $x = \sum_i |X_i|$ and $y = \sum_i |Y_i|$. In total, after
  selecting sets $X_1, X_2, \dotsc$ and $Y_1, Y_2, \dotsc$, sizes of the
  neighbourhoods of the sets, and the number of edges between $X_i, V_1$
  and $Y_i,V_2$, there are at most
  \begin{equation}
    \label{eq:bad-pairs-graph-count}
      \binom{n (s-x)}{m_1 - \tilde m_1} \binom{n (s-y)}{m_2 - \tilde m_2}
      \bigg(\prod_i \gamma^{\min\{d_x(i), d_y(i)\}}\binom{n}{d_x(i)}
      \binom{n}{d_y(i)} \binom{x_i d_x(i)}{m_x(i)} \binom{y_i
      d_y(i)}{m_y(i)}\bigg)
  \end{equation}
  undesired graphs. It remains to show that \eqref{eq:bad-pairs-graph-count} is
  at most
  \[
    e^{4m} \gamma^{\delta m/4} \binom{n s}{m_1} \binom{n s}{m_2},
  \]
  as we can then choose $\gamma = (\beta/(4e^4))^{4/\delta}$ and the using fact
  that there are at most $4^m$ choices of sets $X_i, Y_i$, sizes of their
  neighbourhoods, and edges fixed above, we draw the desired conclusion.

  By using standard bounds on the binomial coefficients, we have
  \[
    \binom{n}{b} \binom{a b}{c} \leq \Big(\frac{n e}{b}\Big)^b \Big(\frac{e
    ab}{c}\Big)^c = \Big(\frac{n a}{c}\Big)^{c} \cdot \frac{e^{c + b}
    b^{c-b}}{n^{c - b}} \leq e^{c+b} \binom{n a}{c}.
  \]
  Hence
  \[
    \binom{n}{d_x(i)} \binom{x_i d_x(i)}{m_x(i)} \leq e^{2m_x(i)} \binom{n
    x_i}{m_x(i)} \qquad \text{and} \qquad \binom{n}{d_y(i)} \binom{y_i
    d_y(i)}{m_y(i)} \leq e^{2m_y(i)} \binom{n y_i}{m_y(i)}.
  \]
  From Vandermonde's identity in the form $\binom{a}{c} \binom{b}{d} \leq
  \sum_{k=0}^{c+d} \binom{a}{k} \binom{b}{c+d-k} = \binom{a+b}{c+d}$,
  \eqref{eq:bad-pairs-graph-count} can be bounded by
  \[
    e^{2(m_1 + m_2)} \gamma^{\sum_i \min\{d_x(i), d_y(i)\}}
    \binom{n s}{m_1} \binom{n s}{m_2}.
  \]
  Observing that
  \[
    \sum_i \min\{d_x(i), d_y(i)\} \geq \sum_i
    \min\set[\Big]{\frac{|X_i|m_1}{2n}, \frac{|Y_i|m_2}{2n}} \geq
    \frac{\delta}{4} m
  \]
  completes the proof.
\end{proof}

Note that if in a graph $G$ there are disjoint sets $V_1, V_2$ in which all
families of disjoint pairs $(X_i, Y_i)$ with $X_i \subseteq V_1$ and $Y_i
\subseteq V_2$ that satisfy some bad property are such that $\sum_i
\min\{|X_i|,|Y_i|\} < \delta n$ then one can delete at most $\delta n$ vertices
in each of $V_1, V_2$ and none of the remaining pairs satisfy the bad property.

\begin{proof}[Proof of Proposition~\ref{prop:almost-all-expanders}]
  Given $k, \beta, \gamma$, we choose several constants so that the arguments
  below follow through. Let $\rho, \delta, \lambda > 0$ be such
  that
  \[
    (1-\rho)^{k-1} \geq 1-\gamma, \quad \delta \leq \min\{\gamma/2, 1/4\}, \quad
    \text{and} \quad (1-3\lambda)(1-\delta) \geq 1-\rho.
  \]
  Next, let $\beta_{k-1} = \beta/2$, and for every $i = k-2, \dotsc, 1$, set
  $\beta_i = \beta_{i+1}/2$. Having fixed these,
  let
  \[
    \tilde\gamma \leq \min_{1 \leq i \leq
      k-1}\{\gamma_{\ref{lem:bad-extensions-pairs}}(\beta_i/2,\delta),
      \gamma/2\} \quad \text{and} \quad \eps_0 \leq \min\{\lambda/4,
      {\eps_0}_{\ref{lem:reg-small-set-inheritance}}(\tilde\gamma,\rho),
    {\eps_0}_{\ref{lem:reg-small-sets-expand}}(\tilde\gamma/2,\lambda)\}.
  \]
  Finally, let $D = D_{\ref{lem:reg-small-set-inheritance}}(\rho)$, and choose $C$
  such that $((1-\rho)np)^{k-1} > D/p$. We present the proof in detail only for
  $t = 2k-1$. The case $t = 2k$ is similar and even easier, and we mention how
  to deduce it in the end.

  Let $i \in [k-1]$, $\ell = 2i$, and let $G$ belong to $\cG(P_\ell, n, m, \eps,
  p)$. We say that a pair of sets $(Q_1,Q_\ell)$ with $Q_1 \subseteq V_1$ and
  $Q_\ell \subseteq V_\ell$ is $(\rho,\lambda)$-expanding if for all $j \in
  [i-1]$:
  \[
    |N_G^j(Q_1)| \geq \min\{|Q_1|(1-\rho)^j(m/n)^j, \lambda n/2\} \quad
    \text{and} \quad |N_G^j(Q_\ell)| \geq \min\{|Q_\ell|(1-\rho)^j(m/n)^j,
    \lambda n/2\},
  \]
  and $\big(N_G^{i-1}(Q_1), N_G^{i-1}(Q_\ell)\big)$ is $(\rho,p)$-lower-regular.
  Observe that, for our choice of $\lambda$ and $\eps$, if $(Q_1, Q_\ell)$
  satisfy $|N_G^j(Q_1)|, |N_G^j(Q_\ell)| \geq \lambda n/2$ for some $0 \leq j
  \leq i-1$, then $|N_G^{j'}(Q_1)|, |N_G^{j'}(Q_\ell)| \geq \lambda n/2$ (with
  room to spare) for all $j' > j$, by Lemma~\ref{lem:slicing-lemma}.

  Next claim is the crux of the argument.

  \begin{claim}\label{cl:inductive-expansion}
    Let $i \in [k-1]$ and $\ell = 2i$. Then all but at most $\beta_i^m
    \binom{n^2}{m}^{\ell-1}$ graphs $G \in \cG(P_\ell, n, m, \eps, p)$ satisfy
    the following. There are sets $X_1 \subseteq V_1$ and $X_\ell \subseteq
    V_\ell$ with $|X_1|, |X_\ell| \leq \delta n$, such that for all $q_1, q_\ell
    \geq (1-\rho)^{k-i}(m/n)^{k-i}$ all but at most
    $\tilde\gamma^{\min\{q_1,q_\ell\}} \binom{n}{q_1} \binom{n}{q_\ell}$ pairs
    $(Q_1, Q_\ell) \in \binom{V_1 \setminus X_1}{q_1} \times \binom{V_\ell
    \setminus X_\ell}{q_\ell}$ are $(\rho,\lambda)$-expanding in $G$.
  \end{claim}

  This is sufficient for the proposition to hold as we show next. By the claim
  applied for $i = k-1$, all but at most
  \begin{equation}\label{eq:num-bad-graphs-claim}
    \beta_{k-1}^m \binom{n^2}{m}^{2k-3} \leq \Big(\frac{\beta}{2}\Big)^m
    \binom{n^2}{m}^{t-2}
  \end{equation}
  graphs $G \in \cG(P_{t-1}, n, m, \eps, p)$ on vertex set $V_2 \cup \dotsb \cup
  V_t$ contain sets $X_2 \subseteq V_2$ and $X_t \subseteq V_t$ with $|X_2|,
  |X_t| \leq \delta n$, such that for $q_2, q_t \geq (1-\rho)(m/n)$
  all but at most $\tilde\gamma^{\min\{q_2,q_t\}} \binom{n}{q_2} \binom{n}{q_t}$
  pairs $(Q_2, Q_t) \in \binom{V_2 \setminus X_2}{q_2} \times \binom{V_t
  \setminus Q_t}{q_t}$ are $(\rho,\lambda)$-expanding. It remains to bound the
  number of $(\eps,p)$-regular graphs with $m$ edges $G[V_1,V_2]$ and
  $G[V_1,V_t]$ which have more than $\gamma n$ vertices in $V_1$ whose
  neighbourhoods into $V_2, V_t$ are of size at least $(1-\rho)(m/n)$ and do not
  fall within expanding pairs. This computation is identical to, e.g.,
  \cite[Lemma~3.2]{noever2017local} and shows that there are at most
  \[
    \Big(\frac{\beta}{2}\Big)^m \binom{n^2}{m}^2
  \]
  such bad choices for $G[V_1,V_2]$ and $G[V_1,V_t]$. Combining it with
  \eqref{eq:num-bad-graphs-claim} and the fact that there are at most
  $\binom{n^2}{m}$ choices for a graph with $m$ edges between two sets of size
  $n$, shows that there are at most
  \[
    \Big(\frac{\beta}{2}\Big)^m \binom{n^2}{m}^{t-2} \binom{n^2}{m}^2 +
    \Big(\frac{\beta}{2}\Big)^m \binom{n^2}{m}^2 \binom{n^2}{m}^{t-2} \leq
    \beta^m \binom{n^2}{m}^t
  \]
  `bad' graphs in $\cG(C_t, n, m, \eps, p)$ as desired.

  \begin{proof}[Proof of Claim~\ref{cl:inductive-expansion}]
    The proof is by induction on $i$. For $i = 1$ it follows by applying
    Lemma~\ref{lem:reg-small-set-inheritance} with $\tilde\gamma$ (as $\beta$),
    $\rho$ (as $\gamma$), and $m/n^2$ (as $p$) since $G[V_{k-1},V_k]$ is
    $(\eps,p)$-regular with $m \geq n^2p$ edges, and thus
    $(\eps,m/n^2)$-lower-regular, and $(1-\rho)^{k-1}(m/n)^{k-1} \geq Dn^2/m$ by
    the bound on $p$ from the statement of the proposition; we even have $X_1 =
    X_2 = \varnothing$.

    We want to show that it holds for $2 \leq i \leq k-1$ assuming it holds for
    $i-1$. By induction hypothesis all but at most
    \[
      \beta_{i-1}^m \binom{n^2}{m}^{\ell-3} \leq \Big(\frac{\beta_i}{2}\Big)^m
      \binom{n^2}{m}^{\ell-3}
    \]
    graphs in $\cG(P_{\ell-2}, n, m, \eps, p)$, on vertex set $V_2 \cup \dotsb
    \cup V_{\ell-1}$, are in the set $\cS$ of `expanding' graphs. In particular,
    every graph in $\cS$ contains sets $X_2 \subseteq V_2$ and $X_{\ell-1}
    \subseteq V_{\ell-1}$ of size $|X_2|, |X_{\ell-1}| \leq \delta n$, such that
    for all $q_2, q_{\ell-1} \geq (1-\rho)^{k-(i-1)}(m/n)^{k-(i-1)}$ at most
    $\tilde\gamma^{\min\{q_2,q_{\ell-1}\}}\binom{n}{q_2}\binom{n}{q_{\ell-1}}$
    pairs $(Q_2,Q_{\ell-1}) \in \binom{V_2 \setminus X_2}{q_2} \times
    \binom{V_{\ell-1} \setminus X_{\ell-1}}{q_{\ell-1}}$ are not
    $(\rho,\lambda)$-expanding.

    We count in how many ways we can `extend' a graph from $\cS$ to obtain a
    `non-expanding' graph. Since the graphs $G[V_1,V_2]$ and
    $G[V_\ell,V_{\ell-1}]$ should be $(\eps,p)$-regular with $m$ edges, it
    follows that both $G[V_1, V_2 \setminus X_2]$ and $G[V_\ell, V_{\ell-1}
    \setminus X_{\ell-1}]$ must contain between $m$ and
    $(1-\eps)(m/n^2)(1-\delta)n^2 \geq m/2$ edges. For each graph in $\cS$ we
    apply Lemma~\ref{lem:bad-extensions-pairs} with
    $(1-\rho)^{k-(i-1)}(m/n)^{k-(i-1)}$ (as $c$), $\beta_i/2$ (as $\beta$),
    $V_1$, $V_\ell$ (as $S_1$, $S_2$), $V_2 \setminus X_2$, $V_{\ell-1}
    \setminus X_{\ell-1}$ (as $V_1$, $V_2)$, and with all pairs
    $(Q_2,Q_{\ell-1})$ as above marked, to conclude that there are at most
    \[
      \Big(\frac{\beta_i}{2}\Big)^m \binom{(1-\delta)n^2}{m_1}
      \binom{(1-\delta)n^2}{m_\ell}
    \]
    `non-expanding extensions', that is graphs $G[V_1, V_2 \setminus X_2]$ and
    $G[V_\ell, V_{\ell-1} \setminus X_{\ell-1}]$, with $m_1$ and $m_\ell$ edges
    with the following property: there is no choice of sets $X_1 \subseteq V_1$
    and $X_\ell \subseteq V_\ell$ of size $|X_1|, |X_\ell| \leq \delta n$, for
    which all pairs $(Q_1, Q_\ell)$ with $Q_1 \subseteq V_1 \setminus X_1$ and
    $Q_\ell \subseteq V_\ell \setminus X_\ell$ that satisfy
    \[
      |N_G(Q_j)| \geq \max\set[\Big]{|Q_j|\frac{m_j}{2n},
      \Big((1-\rho)\frac{m}{n}\Big)^{k-(i-1)}} \enskip \text{for $j \in
      \{1,\ell\}$},
    \]
    have $(N_G(Q_1), N_G(Q_\ell))$ which is $(\rho,\lambda)$-expanding. In
    particular, if $|Q_j| \geq (1-\rho)^{k-i}(m/n)^{k-i}$ and $|N_G(Q_j)| \geq
    (1-\rho)|Q_j|(m/n) \geq |Q_j|m/(2n)$ such pairs $(Q_1,Q_\ell)$ are also
    $(\rho,\lambda)$-expanding.

    As $m \geq 10n\log n$, by Vandermonde's identity, i.e.\ the fact that
    $\binom{a+b}{c} = \sum_{k} \binom{a}{k} \binom{b}{c-k}$, there are at most
    \begin{multline*}
      \Big(\frac{\beta_i}{2}\Big)^m \sum_{m_1 \geq m/2}
      \binom{(1-\delta)n^2}{m_1} \binom{\delta n^2}{m - m_1} \times \sum_{m_\ell
      \geq m/2} \binom{(1-\delta)n^2}{m_\ell} \binom{\delta n^2}{m - m_\ell}
      \leq
      \Big(\frac{\beta_i}{2}\Big)^m \binom{n^2}{m}^2
    \end{multline*}
    `non-expanding extensions' $G[V_1,V_2]$ and $G[V_\ell,V_{\ell-1}]$. Therefore,
    in total there are at most
    \[
      \Big(\frac{\beta_i}{2}\Big)^m \binom{n^2}{m}^{\ell-3} \binom{n^2}{m}^2 +
      \Big(\frac{\beta_i}{2}\Big)^m \binom{n^2}{m}^2 \binom{n^2}{m}^{\ell-3}
      \leq \beta_i^m \binom{n^2}{m}^{\ell-1}
    \]
    graphs $G \in \cG(P_\ell, n, m, \eps, p)$ such that either $G[V_2 \cup \dotsb
    \cup V_{\ell-2}]$ is not in $\cS$ or it is in $\cS$ but its extension is
    `non-expanding'.

    It remains to show that we counted all bad graphs in $\cG(P_\ell, n, m,
    \eps, p)$ or in other words, to show that all remaining graphs contain sets
    $X_1 \subseteq V_1$ and $X_\ell \subseteq V_\ell$, of size $|X_1|, |X_\ell|
    \leq \delta n$, such that for all $q_1, q_\ell \geq
    (1-\rho)^{k-i}(m/n)^{k-i}$ there are at most
    $\tilde\gamma^{\min\{q_1,q_\ell\}} \binom{n}{q_1} \binom{n}{q_\ell}$ pairs
    $(Q_1,Q_\ell) \in \binom{V_1 \setminus X_1}{q_1} \times \binom{V_\ell
    \setminus X_\ell}{q_\ell}$ for which either $|N_G(Q_1)| <
    (1-\rho)|Q_1|(m/n)$ or $|N_G(Q_\ell)| < (1-\rho)|Q_\ell|(m/n)$.

    By Lemma~\ref{lem:slicing-lemma} the graph $G[V_1, V_2 \setminus X_2]$ is
    $(2\eps, m/n^2)$-lower-regular. Hence, from
    Lemma~\ref{lem:reg-small-sets-expand}, applied with $\tilde\gamma/2$ (as
    $\beta$) and $m/n^2$ (as $p$), it follows that for all
    $(1-\rho)^{k-i}(m/n)^{k-i} \leq q \leq \lambda n^2/m$, all but at most
    $(\tilde\gamma/2)^q \binom{n}{q}$ sets $Q \subseteq V_1$, $|Q| = q$, satisfy
    \[
      |N_G(Q, V_2 \setminus X_2)| \geq (1-3\lambda)q|V_2 \setminus X_2|m/n^2
      \geq (1-3\lambda)q(1-\delta)(m/n) \geq (1-\rho)q(m/n).
    \]
    On the other hand, if $q > \lambda n^2/m$, then a set of size $q$ does not
    have a neighbourhood of size at least $(1-\rho)\lambda n$ only if all of its
    subsets of size exactly $\lambda n^2/m$ do not have a neighbourhood of size
    $(1-\rho)\lambda n$, and there are at most $(\tilde\gamma/2)^q\binom{n}{q}$
    of those (this is a simple counting argument, for a proof see, e.g., proof
    of \cite[Theorem~3.6]{gerke2007small}).

    As analogously there are at most $(\tilde\gamma/2)^q\binom{n}{q}$ `bad' sets
    $Q \subseteq V_\ell$, in total there are at most
    \[
      \Big(\frac{\tilde\gamma}{2}\Big)^{q_1} \binom{n}{q_1}\binom{n}{q_\ell} +
      \Big(\frac{\tilde\gamma}{2}\Big)^{q_\ell} \binom{n}{q_\ell} \binom{n}{q_1}
    \]
    `bad' pairs $(Q_1,Q_\ell)$ as desired.
  \end{proof}

  In order to prove the proposition for $t = 2k$ one would first fix $V_{k+1}$,
  and show that there are at most $(\beta/2)^m \binom{n^2}{m}^k$ graphs on
  $G[V_1 \cup V_2 \cup \dotsb \cup V_{k+1}]$ which have more than $\delta n$
  vertices in $V_1$ which are not $(\gamma,k-1)$-expanding or whose $(k-1)$-st
  neighbourhood does not form a $(\gamma,p)$-lower-regular pair with $V_{k+1}$.
  In the same way there are at most $(\beta/2)^m \binom{n^2}{m}^k$ graphs on
  $G[V_1 \cup V_t \cup \dotsb \cup V_{k+1}]$ which have more than $\delta n$
  vertices in $V_1$ which are not $(\gamma,k-1)$-expanding or whose $(k-1)$-st
  neighbourhood does not form a $(\gamma,p)$-lower-regular pair with $V_{k+1}$.
  Combining the two completes the proof.
\end{proof}

\end{document}